\documentclass[12pt,reqno]{amsart}

\usepackage{amsfonts}
\usepackage{amscd}
\usepackage{amssymb}
\usepackage{amsfonts}
\usepackage{amscd}
\usepackage{amsmath} 
\usepackage{mathrsfs}
\usepackage{dsfont}
\usepackage{enumerate}
\usepackage{xcolor}
\usepackage{float} 
\usepackage[pdftex]{graphicx}
\allowdisplaybreaks
\usepackage{url}
\usepackage[hidelinks]{hyperref} 
 \hypersetup{
     colorlinks,
     linkcolor={red!70!black},
     citecolor={blue!80!black},
     urlcolor={blue!60!black}
 }

\date{\today}

\newtheorem*{theorem*}{Theorem}
\newtheorem{theorem}{Theorem}[section]
\newtheorem{corollary}[theorem]{\bf{Corollary}}
\newtheorem{lemma}[theorem]{Lemma}
\newtheorem{proposition}[theorem]{Proposition}
\theoremstyle{definition}

\theoremstyle{remark}
\newtheorem{remark}[theorem]{\bf{Remark}}

\newtheorem{exmp}{\bf{Example}}[section]
\numberwithin{equation}{section}

\newcommand{\beas}{\begin{eqnarray*}}
\newcommand{\eeas}{\end{eqnarray*}}
\newcommand{\bes} {\begin{equation*}}
\newcommand{\ees} {\end{equation*}}
\newcommand{\be} {\begin{equation}}
\newcommand{\ee} {\end{equation}}
\newcommand{\bea} {\begin{eqnarray}}
\newcommand{\eea} {\end{eqnarray}}

\newcommand{\R}{\mathbb R}
\newcommand{\C}{\mathbb C}
\newcommand{\Z}{\mathbb Z}%
\newcommand{\N}{\mathbb N}

\newcommand{\X}{\mathbb{X}}
\newcommand{\what}{\widehat}

\newcommand{\setm}{\setminus}
\newcommand{\intlim}{\int\limits}
\newcommand{\ol}{\overline}

\renewcommand{\Im}{\text{Im}}
\renewcommand{\Re}{\text{Re}}

\renewcommand{\Re}{\operatorname{Re}}
\renewcommand{\Im}{\operatorname{Im}}
 
\usepackage{color}

\title[Pseudo-differential operator on  $\mathrm{SL}(2,\R)$ of $K$-type]{Pseudo-differential operators  on radial sections of line bundles over the Poincaré upper half plane}

\author[Rana, Ruzhansky]{Tapendu Rana, Michael Ruzhansky}	
\address{Tapendu Rana  \endgraf Department of Mathematics: Analysis, Logic and Discrete Mathematics,	\endgraf Ghent University, 	\endgraf Krijgslaan 281, Building S8, B 9000 Ghent, Belgium.} \email{tapendurana@gmail.com, tapendu.rana@ugent.be}
\address{Michael Ruzhansky \endgraf Department of Mathematics: Analysis, Logic and Discrete Mathematics,	\endgraf Ghent University, 	\endgraf Krijgslaan 281, Building S8, B 9000 Ghent, Belgium. 
\endgraf and
\endgraf School of Mathematical Sciences
\endgraf Queen Mary University of London
\endgraf United Kingdom}
\email{michael.ruzhansky@ugent.be} 

\date{}

\begin{document}
\maketitle
\begin{abstract}
In this article, we explore the boundedness properties of pseudo-differential operators on radial sections of line bundles over the Poincaré upper half plane, even when dealing with symbols of limited regularity. We first prove the boundedness of these operators when the symbol is smooth. To achieve this, we establish a connection between the operator norm of the local part of our pseudo-differential operators and the corresponding Euclidean pseudo-differential operators. Additionally, we introduce a class of rough symbols that lack any regularity conditions in the space variable and investigate the boundedness properties of the associated pseudo-differential operators. As a crucial part of our proof, we provide asymptotic estimates and functional identities for certain matrix coefficients of the principal and discrete series representations of the group $\mathrm{SL(2,\R)}$.
\end{abstract}
\tableofcontents
\section{{Introduction}}
The study of pseudo-differential operators can be traced back to the 1960s, with pioneering works by Hörmander \cite{Hormander_1965, Hormander_Ann_1966} and Kohn-Nirenberg \cite{Kohn_Nirenberg_1965}. Motivated by the deep connection between pseudo-differential operators and elliptic and hypoelliptic equations, they extensively studied the boundedness properties of pseudo-differential operators in classical spaces, which played a pivotal role in determining the regularity of solutions to related equations. By introducing the concept of symbol class and utilizing Fourier analysis, Hörmander provided a powerful framework for analyzing these operators on Euclidean space. Consequently, the study of pseudo-differential operators on $\mathbb{R}^d$, as well as in various other spaces, flourished, becoming an important and active research area in modern harmonic analysis and partial differential equations. In this article, we aim to explore the boundedness properties of pseudo-differential operators on radial sections of line bundles over the Poincaré upper half plane $\mathrm{SL}(2,\R)/\mathrm{SO(2)}$, and establish their connection to the classical Euclidean space. We begin by providing some essential background information to facilitate a mathematically rigorous discussion.
  
A pseudo-differential operator on $\R^d$ associated with a symbol $a(x,\xi)$ is an operator defined using the Fourier inversion formula:
\begin{equation*}
a(x, D)f(x)=\int_{\R^d} e^{2\pi i x \cdot \xi} a(x, \xi)\mathcal F f(\xi)\,d\xi,
\end{equation*}
where $f $ is a function in $C_c^{\infty}(\R^d)$, $\mathcal{F}f $ is the Fourier transform of $f $, and the symbol $a(x,\xi)$ belongs to an appropriate class of functions that encodes the behavior of the operator. One of the most widely used classes of symbols is $\mathcal{S}^{m}_{\varrho,\delta}$ introduced by H\"ormander \cite{Hormander_class}. This class consists of all $a \in C^{\infty} (\R^d \times \R^d)$ with
\begin{align*}
\left| \partial^{\beta}_{x}\partial^{\alpha}_{\xi} a(x,\xi) \right| \leq C_{\alpha,\beta} (1+|\xi|)^{m-\varrho|\alpha|+\delta |\beta|},
\end{align*}
for all multi-indices $\alpha, \beta$, where $m \in \R$ and $0\leq \varrho, \delta\leq 1$. H\"ormander's work established the boundedness of operators with symbols belonging to $\mathcal{S}^{d(\varrho-1)/2}_{\varrho,\delta}$ on $L^p(\R^d)$ for $1<p<\infty$, where $0\leq \delta<\varrho<1$; see \cite{Hormander_class}. In 1993, E. M. Stein further investigated the weak type and $L^p$-boundedness problems of pseudo-differential operators on $\R^d$. Specifically, he considered cases where the symbol belongs to the class $\mathcal{S}^{0}_{1,0}$; see \cite[Chap VI, Theorem 1]{Stein_93}. Moreover, by closely following the calculations,  the aforementioned result can be extended to the following version in the one-dimensional case:

 \begin{theorem}\label{thm:pdo_for_family_of_symbols}
 	 	Let  $ a(x,D)$  be the pseudo-differential operator associated with the symbol $ a$, which satisfies the following conditions for all $\alpha, \beta \in \{0,1,2\}$:
 	 	\begin{equation}\label{eqn:condn_on_symbol_family}
 	 		\left| {\partial^\beta_x}{\partial _{\xi} ^\alpha} a (x,\xi)\right| \leq  {C_{\alpha,\beta} }\, {( 1+ |\xi| )^{-\alpha}}. 
 	 	\end{equation}  
         Then, for $ 1<p<\infty $,  $a(x, D)$  extends to  a bounded operator from $ L^p(\R) $ to itself.
 	 \end{theorem}
   \begin{remark} If the symbol $a(x, \xi)$ is independent of the $x$ variable, say $a(x, \xi)=m(\xi)$, then the associated operator is called a multiplier operator.   If the multiplier $m:\R^d \rightarrow \C$ satisfies the following H\"ormander-Mikhlin differential inequalities:
\begin{equation*} 
\left|\frac{d^\alpha}{d\xi^\alpha} m(\xi)\right|\leq A_\alpha |\xi|^{-|\alpha|},
\end{equation*}
for all $\xi\in\R^d\setminus \{0\}$ and for all multi indices $\alpha$ with $0\leq |\alpha|\leq [d/2]+1 $, then  the corresponding multiplier operator is a bounded operator on $L^p(\R^d)$, $1<p<\infty$. We remark that for a pseudo-differential operator,  one cannot weaken the regularity assumption on the symbol $ a(x,\xi) $,  having singularity near $ \xi=0 $. Particularly if the symbol $a  $ satisfies the following simpler-looking condition
 	 	\begin{equation*} 
 	 		\left| {\partial^\beta_x}{\partial _{\xi} ^\alpha} a (x,\xi)\right| \leq  {C_{\alpha,\beta} }\, {|\xi| ^{-|\alpha|}},
 	 	\end{equation*}
  	then the corresponding operator may not be bounded; see \cite[p. 267]{Stein_93}.
\end{remark}

It is well known that pseudo-differential operators with symbols belonging to the class $\mathcal{S}^{0}_{1,0}$ are generalized Calderón-Zygmund operators, meaning that their kernels satisfy Hörmander’s condition, which imposes smoothness conditions on the space variable of the symbol. This naturally raises the question of investigating the boundedness problem of pseudo-differential operators with limited regularity in the space variable of the symbol. In this direction, several authors, including H. Kumano-go \cite{Kumango70}, H\"ormander \cite{HormanderL271}, Nagase \cite{Nagase77}, and R. Coifman and Y. Meyer \cite{Coifman78} have studied the $L^p$-boundedness problem with limited regularity.  More recently, in 2007, Kenig and Staubach \cite{carlos_07} introduced the class of $\Psi$-pseudo-differential operators, where the symbols have no regularity assumptions in the space variable. This opens up new possibilities for exploring the boundedness properties of pseudo-differential operators without the usual smoothness conditions on the symbol.

       Let $ {\mathcal{S}^{m}_{\varrho,\infty}}$ be the class of symbols consisting of all $a(x,\cdot) \in C^{\infty} (\R^d_\xi)$ such that $x\mapsto a(x,\xi)$ is a measurable function in $x$ and 
       \begin{align*} 
       \| \partial^{\alpha}_{\xi}a(x,\xi)\|_{L^\infty_{x}} \leq C_\alpha (1+|\xi|)^{m-\varrho |\alpha|}    
       \end{align*}
    for  all multi indices $\alpha$, where $m\in \R, ~\varrho\leq 1$, and  $C_{\alpha}$'s are constants. In the paper \cite{carlos_07}, the authors established the following result concerning the $L^{p}$-boundedness problem of pseudo-differential operators with non-regular symbols. 
   \begin{theorem}[{{\cite[Theorem 2.7]{carlos_07}}}]\label{thm_carlos}
       Let $a  \in  {\mathcal{S}^{m}_{\varrho,\infty}}, ~ 0\leq \varrho \leq 1$. Suppose that  $m<n(\varrho-1)/2$, then $a(x,D)$ is a bounded operator from $L^p(\R^d)$ to itself for all $p \in [2,\infty]$. For $\varrho =1$ and $m <0$ the range of $p$ for which the operator is $L^p$ bounded is $p\in [1,\infty]$. 
   \end{theorem}

Over the past few decades, the theory of pseudo-differential operators on various Lie groups has gained significant popularity and has become a rich field of study with vast literature. For instance, the second author and his collaborators explored a noncommutative analog of the Kohn–Nirenberg quantization \cite{Kohn_Nirenberg_1965} for operators on compact Lie groups \cite{Ruzhansky_torus_10, Ruzhansky_comp_13, Ruzhansky_PDO_SYM, CR_Mem} and introduced the corresponding H\"ormander symbol class \cite{MR3217484, Ruzhansky_glo_14, MR3723800}.  Moreover, research on the $L^p$-boundedness of pseudo-differential operators in the context of compact Lie groups has been undertaken; see \cite{MR3369343, MR3610249, MR3936641}. Bernicot and Frey \cite{MR3180890} have delved into the study of pseudo-differential operators on homogeneous spaces. Additionally, the theory of pseudo-differential operators has been extended to graded Lie groups, with the authors in \cite{MR3362017, FR, MR4322546} introducing the symbol class based on a positive Rockland operator. Further research on pseudo-differential operators in the context of nilpotent Lie groups can be found in \cite{MR3167566, MR3969446}.

In recent times, there has been a growing interest in the theory of pseudo-differential operators on discrete spaces, primarily due to its connections with quantum ergodicity problems and the discretization of continuous problems; see \cite{Botchway_Ruzhansky_2020, Masson_14, Rana_Rano}. Moreover, other noteworthy works on pseudo-differential operators in non-Euclidean settings can be found in \cite{MR3680539, MR3953844, MR4078700, MR3741847,  MR3923338, MR4343620,  MR4565437}. 

The study of pseudo-differential operators has predominantly occurred in the doubling setting, which relies on suitable covering lemmas. However, when dealing with the Poincaré upper half plane or more general noncompact type symmetric spaces, these spaces exhibit exponential volume growth, leading to the absence of analogues for the Calderon-Zygmund decomposition or useful covering lemmas. Despite these challenges, Clerc and Stein addressed the multiplier problem in their seminal work \cite{Clerc_Stein} on general noncompact type symmetric spaces. In the context of the rank one symmetric space $\X$, the authors observed that for a given multiplier $m$, if the associated multiplier operator $T_m$ is to be bounded on $L^p(\X)$ for some $p\in (1,\infty)\setminus \{2\}$, then $m$ must necessarily extend to a bounded holomorphic function in the interior of the strip $S_p$ with
 \begin{align}\label{defn_Sp_intro}
	S_{p}:=\{\lambda \in\C: |\Im \lambda|\leq \gamma_p |\rho|\},\quad \text{where} \quad \gamma_p:=\left|\frac{2}{p}-{1}\right|,\quad\text{for }p\in[1,\infty),
\end{align}
and $\rho$ is half the sum of all positive roots with multiplicity. Later, in 1990, Anker, in his remarkable work  \cite{Anker}, proved the following analogue of the H\"ormander-Mikhlin multiplier theorem on noncompact type symmetric spaces of arbitrary rank. However, to avoid introducing additional notation, we present the results here in the context of rank one symmetric spaces.
\begin{theorem}[{{\cite[Theorem 1]{Anker}}}] \label{Anker}
	Let $\X$ be a rank one symmetric space of noncompact type, $1< p<\infty$, $v=| \frac{1}{p}-\frac{1}{2} |$ and $N=[v \dim \X] +1$. Assume that $m:\R\rightarrow \C$ extends to an even holomorphic function on $S_p^\circ$, $\frac{\partial^\alpha}{\partial \lambda^\alpha} m    
	\,  (\alpha=0, 1, \cdots, N)$ extend continuously to $S_p$  and satisfy
	\bes
	\sup_{\lambda\in S_p} (1 + |\lambda|)^{\alpha}\left|\frac{\partial^\alpha}{\partial \lambda^\alpha} m(\lambda) \right|<\infty,
	\ees
	for all $\alpha=0, 1, \cdots, N$.
	Then the associated multiplier operator $T_m$ 	is a bounded operator on $L^p(\X)$  for all $p\in (1,\infty)$.
 \end{theorem}

We note that Anker also proved the weak type $(1,1)$-boundedness of the multiplier operator by imposing certain regularity assumptions on the boundary. However, Anker suggested that it might be possible to relax these assumptions and allow $m$ to have a singularity at the boundary points $\pm i\gamma_p |\rho|$ while still being an $L^p$ multiplier on $\X$.
Later, Ionescu \cite{Ionescu_2002,Ionescu_2003} made further improvements to the theorem by replacing the continuity requirement of the multiplier $m$ on the boundary with a condition related to the singularity at $\pm i\gamma_p |\rho|$. This condition is the best-known sufficient condition of the Hörmander-Mikhlin type on $\X$.

\begin{theorem}[{{\cite[Theorem 8]{Ionescu_2002}}}]
	Let $\X$ be a rank one symmetric space of noncompact type and  $1< p<\infty$. Assume that $m:\R\rightarrow \C$ extends to an even holomorphic function on $S_p^\circ$ and satisfies
	 \bes
	\sup_{\lambda\in S_p^{\circ}} (\left|\lambda\pm i\gamma_p |\rho| \right|)^{\alpha}\left|\frac{\partial^\alpha}{\partial \lambda^\alpha} m(\lambda) \right|<\infty,
\ees
for all $\alpha=0, 1, \cdots ,\left[\frac{\dim \X+1}{2}\right]+2$. Then, the associated multiplier operator $T_m$ is a bounded operator on $L^p(\X)$.
\end{theorem}  
  
Motivated by the works of Ionescu \cite{Ionescu_2002,Ionescu_2003}, the authors Ricci and Wr\'{o}bel \cite{Ricci} studied the $L^p$-boundedness ($1<p<\infty$) of multiplier operators on the radial sections of line bundles over the Poincaré upper half plane $\mathrm{SL}(2,\R)/\mathrm{SO(2)}$. More precisely, they focused on the group $G=\mathrm{SL}(2,\R)$, and instead of considering $K=\mathrm{SO(2) }$-biinvariant functions, they explored the problem of the $L^p$-boundedness for the multiplier operator defined on functions on $G$ satisfying the property:
\begin{align}\label{n_ntype_int}
f(k_\theta g k_\vartheta) = e^{in(\theta+\vartheta)} f(g),
\end{align}
 for all $g \in G$, $k_\theta, k_{\vartheta}\in K$, where $n\in \Z$ is fixed and \begin{align*}
  k_\theta := \begin{pmatrix}
				\cos \theta  & \sin \theta   \\
				-\sin \theta & \cos \theta  
			\end{pmatrix}.
   \end{align*} 
The functions satisfying \eqref{n_ntype_int} are called $(n,n)$ type functions on $G$. In the special case where $n=0$, their result \cite[Theorem 5.3]{Ricci} aligns with the multiplier theorems on symmetric spaces obtained by Anker \cite{Anker} and Stanton and Tomas \cite{Stanton_and_Tomas}. Moreover, the authors \cite{Ricci} extended the result of Clerc and Stein \cite[Theorem 1]{Clerc_Stein} in this setting by providing a necessary condition on the multiplier $m$ for the associated multiplier operator to be $L^p$-bounded for $p\in (1,\infty)\setminus \{2\}$. For more details, please see Theorem \ref{thm_RW}.
   
Given the progress in studying multiplier operators,  the present article continues the research trend by introducing pseudo-differential operators on radial sections of line bundles over the Poincaré upper half plane and investigating their boundedness properties. 

In this study, we establish a sufficient condition on the symbol that ensures the $L^p$-boundedness of the corresponding pseudo-differential operator for $(n,n)$ type functions on $G$; see Theorem \ref{thm:pdo_for_n,n_type}. This result not only extends the findings in \cite{Ricci} for the multiplier case but also presents an analogous version of Theorem \ref{thm:pdo_for_family_of_symbols} in the context of Euclidean space.

Additionally, we present a weak-type version of the generalized transference principle by Coifmann-Weiss. By employing this principle, we establish a connection between the operator norm of the local part of our pseudo-differential operators and the corresponding Euclidean pseudo-differential operators.

Furthermore, we delve into the boundedness problem of pseudo-differential operators associated with non-smooth symbols in the space variables in our setting; see Theorem \ref{thm_nonreg}. This result is a natural analogue of the boundedness of rough pseudo-differential operators established by Kenig and Staubach in Euclidean spaces Theorem \ref{thm_carlos}. To further investigate this area, inspired by the works of Kenig and Staubach \cite{carlos_07}, we introduce the class $\mathcal{S}^m_{\varrho,\infty}(S_p)$ of rough symbols with no regularity conditions in the space variable and analyze the boundedness properties of the associated pseudo-differential operators for $(n,n)$ type functions on $G$.

Before we proceed into the specifics of the paper, let us compare our analysis of pseudo-differential operators with that of multiplier operators for $(n,n)$ type functions on $G$.

In the context of multipliers, the associated operator can be represented as a convolution operator, which allows one to make use of various convolution inequalities like the \textit{Kunze-Stein} phenomenon and Herz majorizing principle available on the group $G$. However, dealing with pseudo-differential operators does not provide the same convenience of employing these mentioned methods, marking a crucial distinction between the two cases. Additionally, in this group setting, there is an extra discrete component of the Plancherel measure, leading us to decompose our pseudo-differential operator $\Psi_{\sigma}$ into discrete and continuous parts; see \eqref{eqn_decom_of_psi_sigma}. Addressing the discrete part poses a new challenge, where we use an asymptotic estimate of the spherical functions ($\psi_{ik}^{n,n}$), departing from the \textit{Kunze-Stein} phenomenon used in the multiplier case \cite{Ricci}.  In fact, we establish a characterization of the matrix coefficients of discrete series representation lying in weak $L^p$ spaces; see Lemma \ref{lem_psi_k_Lp}.   This characterization represents the weak-type version of Mili\v{c}i\'c's result \cite[Corollary, p.84]{mililic_1977} in the context of $\mathrm{SL}(2,\R)$, which we employed to establish the boundedness of the discrete part of $\Psi_{\sigma}$.

Moving to the continuous part of $\Psi_{\sigma}$, we first use a functional identity of the spherical functions ($\phi_{\tau,\lambda}^{n,n}$) to rewrite the operator as a kernel integral operator on $G$. However, considering the exponential growth of the measure on the group, we further decompose the continuous part of $\Psi_\sigma$ into two parts, namely the local part $(\Psi^{\text{loc}}_{\sigma})$ and the global part $(\Psi^{\text{glo}}_{\sigma})$; see Section \ref{eqn_decom_of_psi_sigma}.

Proving the boundedness of the local part $(\Psi^{\text{loc}}_{\sigma})$ constitutes a major challenge. In the case of multipliers, the availability of a transference principle for convolution operators leads to their boundedness. On the other hand, for pseudo-differential operators, we utilize a generalized Coifman-Weiss transference principle for singular integral operators to establish a relation between the operator norm of the local part of our pseudo-differential operators and the corresponding Euclidean pseudo-differential operators. This relation helps us derive the derivative condition on the space variable of the symbol.

The situation is entirely different for the global part $(\Psi^{\text{glo}}_{\sigma})$, where the analysis of the Poincaré upper half-plane becomes crucial. While multiplier theory in \cite{Ricci} relies on the Herz majorizing principle to prove the boundedness, we need to adopt a different approach. We use the global expansion of the spherical function and the holomorphicity condition on the dual variable of the symbol to obtain a quantitative estimate of the kernel away from the origin. Finally, by employing a property of the Abel transform and a duality argument, we establish the $L^p$-boundedness of $\Psi^{\text{glo}}_{\sigma}$.

A critical element in proving the boundedness of $\Psi_\sigma$ for a rough symbol $\sigma$ (Theorem \ref{thm_nonreg}) is the established relation between the local part of a pseudo-differential operator on radial sections of line bundles over the Poincaré upper half-plane and the corresponding operator in the Euclidean space. Specifically, using the generalized Coifmann-Weiss transference principle, we reduce the $L^p$-boundedness of the local part $\Psi_\sigma^{\text{loc}}$ to the $L^p$-boundedness of certain Euclidean pseudo-differential operators whose symbols depend on $\sigma$ (as shown in Section \ref{Analysis on the local part}). This enables us to apply the result of Kenig and Staubach (Theorem \ref{thm_nonreg}) in our setting and, consequently, establish the boundedness of the operator $\Psi_\sigma^{\text{loc}}$. The boundedness of $\Psi^{\text{dis}}_{\sigma}$ and $\Psi^{\text{glo}}_{\sigma}$ follows similarly, as these operators do not require any smoothness condition on the space variable of the symbol $\sigma$.

  We conclude this section by providing an outline of this article.

In the next section, we present the necessary background on the group $\mathrm{SL}(2,\R)$ and establish some results relevant to our study. Then, in Section \ref{main_result}, we introduce pseudo-differential operators for $(n,n)$-type functions on $G$ and present our main results. In the same section, we introduce the class $\mathcal{S}^m_{\varrho,\infty}(S_p)$.

In Section \ref{sec_transference_principle_weak}, we delve into the generalized Coifman-Weiss transference principles. Following that, in Section \ref{sec_phi_m,n}, we prove some functional identities and derive asymptotic estimates of the spherical functions, which are crucial for our study.

Next, in Section \ref{decomposition of PSi}, we represent $\Psi_{\sigma}$ as singular integral operators. The subsequent sections, namely Sections \ref{sec_dis_part}, \ref{Analysis on the local part}, and \ref{sec_glo_par}, are dedicated to stating and proving the boundedness results of $\Psi^{\text{dis}}_{\sigma}$, $\Psi^{\text{loc}}_{\sigma}$, and $\Psi^{\text{glo}}_{\sigma}$, respectively. The proofs in these sections will lead to the establishment of Theorem \ref{thm:pdo_for_n,n_type}.

Finally, in Section \ref{sec_nonreg}, we prove Theorem \ref{thm_nonreg}, concluding our investigation.
\section{Preliminaries}\label{Preliminaries}
\subsection{Generalities}
The letters $\N$, $\Z$, $\R$, and $\C$ will respectively denote the set of all natural numbers, the ring of integers, and the fields of real and complex numbers. We denote the set of all non-negative integers and nonzero integers by $ \Z_+ $ and $\Z^*$, respectively. For $ z \in \C $, we use the notations $\Re z$ and $\Im z$ for real and imaginary parts of $z$, respectively. We shall follow the standard practice of using the letters $C$, $C_1$, $C_2$, etc., for positive constants, whose value may change from one line to another. Occasionally, the constants will be suffixed to show their dependencies on important parameters. We will use $X\lesssim Y$ or $Y\gtrsim X$ to denote the estimate $X\leq CY$ for some absolute constant $C>0$. We shall also use the notation $X\simeq Y$ for $X\lesssim Y$ and $Y\lesssim X$.  For any Lebesgue exponent $ p \in [1,\infty)$, let $p^{\prime}$ denote the conjugate exponent $p/(p-1)$. Throughout this article,  we denote  
\begin{align}\label{defn_Sp}
	S_{p}=\{z\in\C: |\Im z|\leq \gamma_p\}\quad\text{and}\quad \gamma_p=\left|\frac{2}{p}-1\right|,\quad\text{for }p\in[1,\infty).
\end{align}
From the above definitions it is evident that $\gamma_{p}=\gamma_{p^{\prime}}$ and $S_{p}=S_{p^{\prime}}$ for all $p\in[1,\infty)$. We shall henceforth write $S^{\circ}_{p}$ and $\partial S_{p}$ to denote the usual topological interior and the boundary of $S_{p}$ respectively. 
\subsection{Lorentz spaces}Let $(X,m) $ be a $ \sigma $-finite measure space. For $f: X \rightarrow \C$  a measurable function on $ X $, the distribution function $ d_f $ defined on $ [0,\infty) $  is given by $$ d_f(\alpha) = m (\{ x \in X \mid |f(x)| > \alpha \}) .$$
  For  $ p\in [1,\infty), q \in [1,\infty]$, the Lorentz spaces $ L^{p,q}(X) $ consist of all measurable functions $f  $ on $X$ for which  $\| f \|_{p,q} $ is finite, where  $\| f \|_{p,q} $ is the Lorentz space norm defined as follows, as in  \cite[Prop.1.4.9]{Grafakos}  
  \begin{equation}\label{defn_lor}
 	\|f\|_{p,q} = \begin{cases}
 		p^{\frac{1}{q}} \left( \int_0^\infty \alpha^q d_f(\alpha) ^{\frac{q}{p}} \frac{d\alpha}{\alpha}\right)^\frac{1}{q}  \text{  when } q < \infty,\\
 		\sup\limits_{\alpha > 0} \,   \alpha^{{p}} d_f(\alpha)   \hspace{1.9 cm}\text{ when } q =\infty.
 	\end{cases}
 \end{equation} 
\subsection{The group $ \mathrm{SL}(2,\R) $}\label{sec:The_group_SL(2,R)} From this section onwards, $G$ will always denote the group $\mathrm{SL}(2, \R)$. The  Iwasawa decomposition for $G$  gives a diffeomorphism between $K \times A \times N$ and $G$, where $K=SO(2),$
	       
        \begin{equation*}
            \begin{aligned}
	            A&=\left\{ a_t := \begin{pmatrix}
				e^t  & 0   \\
				0 & e^{-t}  
			\end{pmatrix} :  t \in \R \right\}, \quad \text{and} \quad N=\left\{ n_v := \begin{pmatrix}
				1  & v  \\
				0 & 1  
			\end{pmatrix} :   v \in \R \right\} .
              \end{aligned}
              \end{equation*}
			That is, by the Iwasawa decomposition any $x \in G$ can  be  written uniquely  as $x =k_\theta  a_t n_v $; using this, we write $$K(x) = k_\theta,\quad H(x) = t, \quad \text{and} \quad N(x)= v.$$ In fact, if $ x=\begin{pmatrix}
				a  & b  \\
				c & d  
			\end{pmatrix}\in \mathrm{SL}(2,\R)  $, then $\theta, t$ and $v$ are given by
   \begin{equation}\label{formula_for_SL2R}
       e^{2t} =a^2+c^2,\, e^{i\theta}= \frac{a-ic}{\sqrt{a^2+c^2}}, \text{ and } v= \frac{ab+cd}{\sqrt{a^2+c^2}}.
   \end{equation}We also have another Iwasawa  decomposition 
        \bes 
                     G=\ol{N} A K, 
        \ees 
        where $$ \ol{N}=\left\{ \overline{v} := \begin{pmatrix}
				1  & 0  \\
				v & 1  
			\end{pmatrix} :   v \in \R \right\} .$$
        Next, let $A^+=\{ a_t \mid t>0 \}$; then the Cartan decomposition for $G$  gives $$G={K \overline{A^+}K} .$$ Using this we define $g^{+} $ as the $\R^+$ component in the Cartan decomposition, of the element  $g\in G$, that is we denote $a_{g^+}$ as the unique element such that $$g =k_1 a_{g^+} k_2, \quad \text{for some } k_1,k_2 \in K.$$
       A function $f:G \rightarrow \C$ is said to be of $(n,n)$ type if 
        \begin{equation} \label{K_type_n}
				f(k_{\theta} x k_{\vartheta}) = e_n(k_\theta)f(x) e_n(k_\vartheta) 
			\end{equation}
   for all $k_{\theta}, k_{\vartheta} \in K$ and $x\in G$, where $e_n(k_\theta) := e^{in \theta}$. If $n=0$, we refer to such functions as $K$-biinvariant functions.

   By the Cartan decomposition we can extend a function $ f $ defined on $ \overline{A^+} $ to an $ (n,n) $ type function on $ G $ by 
   \begin{equation}\label{extension a function defined on closure of A+ to G}
				f( k_\theta a_t k_\vartheta ) = e_n(k_\theta) f(a_t) e_n(k_\vartheta), \quad \text{ for all } k_\theta , k_\vartheta \in K.
	\end{equation}  
 It is known that the following groups $G$, $\ol {N}$, and $K$ are unimodular, and we will denote the  (left or equivalently right) Haar measures of these groups as $dx$, $d \ol{v}$ and  $dk$, where $\int_K dk=1$. Then we have the following integral formulae corresponding to the Iwasawa and  Cartan decompositions, which hold for any integrable function $f$:
			\begin{equation}\label{eqn:integral-decom-nbar-a}
				\intlim_{G}f(x) dx =\int_{\ol N} \int_{\R} \int_{K} f(\ol  va_t k)e^{2 t} d\ol v dt dk,
			\end{equation}
			and		
			\begin{equation}\label{definition of integration}
				\int_G f(x) dx = \int_K \int_{\R^+} \int_K f(k_1 a_t k_2) \Delta(t) dk_1 dt dk_2,
			\end{equation}
			where $\Delta(t)= 2 \sinh 2t$.
   We recall the following integration formula on $K$ from \cite[Lemma 5.19, p.197]{Helgason_GGA}.
\begin{lemma}\label{k(xk)_dif}
    Let $x \in G$. The mapping $T_x:k \mapsto K(xk)$ is a diffeomorphism of $K$ onto itself and 
    \begin{equation}
        \int_K F(K(xk)) \, dk =\int_K F(k) e^{-2 H(x^{-1}k)} \, dk, \quad F\in C(K).
    \end{equation}
\end{lemma}
   We require a connection between the Iwasawa decomposition of $G = \overline{N}AK$ and the Cartan decomposition. Str\"omberg \cite{Stromberg} previously employed a similar relation for noncompact-type symmetric spaces. In this context, Ionescu established the following result:
			\begin{lemma}$($\cite[Lemma 3]{Ionescu_2002}$)$\label{lemma:expression_of_v-a _r}
				If $ \ol v \in \ol N $ and $ r\geq 0 $, then 
				\begin{equation}
					[\ol v a_r]^+ = r +H(\ol v)+ E(\ol v, r), 
				\end{equation}
				where \begin{equation}
					0\leq E(\ol v, r)\leq 2 e^{-2r}.
				\end{equation}
			\end{lemma}
From the formula \eqref{formula_for_SL2R}, we have
\bes
e^{2 H(\ol v)} = \left( 1+ v^2\right), \quad \text{for all } v\in \R,
\ees
and so 
\be\label{eqn:H(v)-positive}
		 H(\ol v) \geq 0,  \,\, \text{ for all } \ol v \in \ol N.
\ee
Also, it is easy to see for any $ \epsilon_0 >0 $ that 
\begin{equation}\label{eqn:L1_norm_P(v-) (1+epsilon)_is_finite}
				\int_{\ol N}  e^{- (1+\epsilon_0)H(\ol v)} d\ol v =C_{\epsilon_0 }<\infty.
			\end{equation}
It is well known that the abelian group $ A $ acts as a dilation on $ \ol N  $, by the mapping
\begin{eqnarray*}
	\ol n\mapsto a \ol n a^{-1} \in \ol N.
\end{eqnarray*}
Moreover, if  $\delta_r$ is the dilation of $\ol N$, defined by
\begin{equation}\label{dil_delta}
\delta_r(\ol n):=  a_r \ol n a_{-r},
 \end{equation} 
then the following is true
\begin{equation}\label{eqn:dialation_formula_for_N-}
	\int_{\ol N} \mathfrak{h}(\delta_r(\ol n)) \,d\ol n =  e^{2 r} \int_{\ol N} \mathfrak{h}( \ol n)  \,d\ol n,
\end{equation}
for any integrable function $ \mathfrak{h} $ on $ \ol N $. We equip the set $\overline{N}A$ with the binary operation induced by this conjugation $\delta_{r}$ and refer to it as the semi-direct product $\overline{N} \ltimes A$. With the Iwasawa decomposition, as shown in \eqref{eqn:integral-decom-nbar-a}, we can establish an identification between $K$-right invariant functions on $G$ and functions on $\overline{N}A$, and it is known that the corresponding $L^{p}$-norms coincide. Additionally, we will make use of the fact that the Abel transform  
\begin{equation}\label{Abel_tran}
    \mathcal{A}\phi(r):= e^{r}\int_{\ol N} \phi(\overline{n}a_r) d\overline{n} 
\end{equation}
takes a $K$-biinvariant function $\phi$ to an even function on $\R$ (see \cite[p. 381]{Helgason_GA}).

\subsection{Spherical analysis of $(n,n)$ type functions}
		We are going to define Fourier transforms of suitable $(n,n)$ functions using representations of $G$ (see Chapter 9 of \cite{Barker}).  	Let $ \what M $ denote the equivalence classes of irreducible representations of $ M= \{  \pm I \} $, where $ I $ is the identity matrix. Then \[ \what M = \{ \tau_ +, \tau_-\}, \]
			where $ \tau_ +(\pm I)=1 $ and $ \tau_-(\pm I) =\pm 1 $. For each $\tau \in \what M$,  $\Z^{\tau} $ stands for the set of even integers for $\tau =\tau_+$, the set of odd integers for $\tau =\tau_-$ and $ -\tau $ will denote the opposite parity of $ \tau $.  
			\noindent We define \begin{align}\label{definition of Z(k)}
				\Gamma_n =\begin{cases} \lbrace
					k \mid 0 <k < n \text { and } k \in \Z^{-\tau} \rbrace \text{ if } n > 0 \\
					\lbrace k \mid n <k < 0\text { and } k \in \Z^{-\tau} \rbrace \text{ if } n < 0.\end{cases}
			\end{align}
		For $\lambda \in \C $ and $ \tau \in \what M $,  let $(\pi_{\tau, \lambda }, H_\tau)$ be the principal series representation of $ G $ given by
		\begin{equation}\label{definition of principal series representation}
			{\left(\pi_{\tau, \lambda}(x) e_n\right)(k)=e^{(i\lambda - 1)H(xk)} e_n\left(K(x^{-1}k^{-1})\right)^{-1}}
		\end{equation}
		for all $ x \in G, k\in K $, where {$H_\tau$}  is the subspace of $L^2(K)$ generated by the orthonormal set  $\{ e_n \mid n \in \Z^\tau \}$. This representation is unitary if and only if $\lambda \in \R$.   For $ \lambda =0  $, the representation $\pi_{\tau_-, 0}$  has two irreducible subrepresentations, the so called  {mock discrete series}. We will denote them by $D_+$ and $D_-$. The representation spaces of  {$D_+$ and $D_-$ contain  $e_n \in L^2(K)$ respectively for positive odd $n$'s and negative odd $n$'s}.   For each  $k\in\mathbb Z^\ast$ (set of nonzero integers), there is a discrete series representation $\pi_{ik}$, which occurs as a subrepresentation of $\pi_{\tau, i|k|}$, $k\in \Z\setminus \Z^{\tau}$   (see \cite[p.19]{Barker}).  We define for    $ k \in \Z^*  $
		\begin{equation*}
			\Z(k) = \begin{cases}\{ m \in \Z^{\tau} : m \geq k+1 \} \text{ if }  k\geq 1 \\
				\{ m \in \Z^{\tau} : m \leq k-1 \} \text { if }  k\leq -1.
			\end{cases}
		\end{equation*}
		For $n\in \Z^\sigma$,  the canonical matrix coefficient for the principal series
		\begin{equation}\label{definition of phi sigma lamda m, n}
			\phi_{\tau,\lambda}^{n,n}(x) := \left\langle \pi_{\tau, \lambda}(x) e_n, e_n \right\rangle  =\int_K e^{-(i\lambda+1)H(xk)} e_n(k^{-1}) \overline{e_n(K(xk)^{-1})} dk,  
		\end{equation}
		are functions of $ (n,n) $ type. For $k \in \Z^*$ and $n\in \Z(k)$ the canonical matrix coefficient of the discrete series representation is
            \begin{equation*} {\psi^{n,n}_{ik}(x):=\langle \pi_{ik}(x) e_n^k, e_n^k\rangle_k}, 
		\end{equation*}
		where ${e_n^k}$ are the renormalized basis and $\langle \,,\rangle_k$ is the renormalized inner product for $\pi_{ik}$, see \cite[p. 20]{Barker} for more details. These functions $\phi_{\tau,\lambda}^{n,n}$ and $\psi^{n,n}_{ik}$  are  also known as spherical functions of  $(n,n)$ type. We will denote $ H_k $ as the Hilbert space generated by $ \{ e_m^k : m \in \Z(k)\} $.   It is known  \cite[Prop. 7.3]{Barker} that 
         \begin{equation}\label{relation between psi n,n and phi {n,n}}
			\psi_{ik}^{n,n} = \phi_{\tau, i|k|}^{n,n},
		\end{equation} where $\tau \in \what{M}$ is determined by $ k \in \Z^{-\tau} $. Additionally, it is worth noting that $\phi_{\tau_+, \lambda}^{0,0}$ corresponds to the elementary spherical function commonly denoted as $\phi_\lambda$.

          Let $\Omega$ denote the Casimir element on $G$ (see \cite[(2.6)]{Barker}), which acts as a biinvariant differential operator on $G$. As stated in \cite[(4.7)]{Barker}, the smooth eigenfunctions $\phi_{\tau, \lambda}^{n,n}$ of $\Omega$ satisfy the differential equation:
    \begin{equation}\label{ef_phi_n,n}
    \Omega f = -\frac{\lambda^2+1}{4} f.
    \end{equation}
 
    Furthermore, the spherical functions ${\phi_{\tau,\lambda}^{n,n}}$ have the following well-known properties:
    \begin{enumerate}
    \item We have from \cite[Section 2.3, 4.1]{Rana}  $${\phi_{\tau,\lambda}^{n,n}}(a_t)={\phi_{\tau,\lambda}^{-n,-n}}(a_t)={\phi_{\tau,\lambda}^{n,n}}(a_{-t}), \quad \text{for all $t\in \R$.}$$
    \item For any fixed $x \in G$, $\lambda \mapsto {\phi_{\tau,\lambda}^{n,n}}(x)$ is an entire function.
    \item By utilizing the Cartan decomposition of $G$, it is easy to see that
    \begin{equation}\label{phi_n,n-lam}
     \phi_{\tau,\lambda}^{-n,-n}(x^{-1}) = {\phi_{\tau, \lambda }^{n,n}}(x) \quad \text{and} \quad {\phi_{\tau,\lambda}^{n,n}}={\phi_{\tau,-\lambda}^{n,n}},
    \end{equation}
    for all $x\in G, \lambda \in \R$. 
    \end{enumerate}

  We remark that we use a different parameterization of the representations and spherical functions from Barker \cite{Barker}. According to our definition in \eqref{definition of principal series representation}, the unitary dual of $\mathrm{SL}(2,\R)$ is $\R$, while according to Barker's convention, the unitary dual of $\mathrm{SL}(2,\R)$ is $i\R$. As a result, our $\pi_{\lambda}$ corresponds to his $\pi_{-i\lambda}$, and similarly, $\phi_{\tau,\lambda}^{n,n}$ and $\psi_{ik}^{n,n}$ are reparametrized accordingly.   This choice of parametrization offers a clearer analogy with the general semisimple Lie groups case, making our analysis more transparent.
     
   Let $n\in \Z^\tau.$ Then for a smooth compactly supported $(n,n)$ type function $f$  on $G$, the principal series Fourier transform of $f$ is defined by  \begin{equation*}
			\what{f_H}(\lambda) = \int_G f(x) {\phi_{\tau,\lambda}^{n,n}}(x^{-1}) dx,
		\end{equation*}
   for all $\lambda \in \R$, and  the discrete series Fourier transform is defined by 
        \begin{equation*}
			\what{f_B}(ik)  = \int_G f(x) {\psi_{ik}^{n,n}}(x^{-1}) dx,
		\end{equation*}
for all $k \in \Gamma_n$. Then   the inversion formula is given by \cite[Theorem 10.4]{Barker}:
\begin{equation}\label{inversion_n,n}
    f(x)= \frac{1}{4\pi^2} \int_{\R} \what {f}_H(\lambda) \phi^{n,n}_{\tau,\lambda}(x) \left|c^{n,n}_\tau(\lambda)\right|^{-2} d\lambda + \frac{1}{2\pi} \sum_{k \in \Gamma_n} \what{f}_B (ik) \psi^{n,n}_{ik}(x) |k|, 
\end{equation}
where $\tau \in \what M$ is determined by $n \in \Z^\tau$, and $c^{n,n}_\tau(\lambda)$ is given by \cite[(6.2)]{Barker}
{\begin{align}\label{defn_c_n,n}
  c^{n,n}_\tau(\lambda) =\frac{1}{\sqrt{\pi}}  \frac{\Gamma\left(\frac{i\lambda}{2}\right) \Gamma\left(\frac{1+i\lambda}{2}\right)}{\Gamma\left(\frac{1+i\lambda-|n|}{2}\right) \Gamma\left(\frac{1+i\lambda+|n|}{2}\right)},
\end{align}
for all $\lambda \in \C$. The functions $c^{n,n}_\tau$ are regarded as meromorphic functions on the complex plane. In the following, we identify the singular points and discuss the estimates of the function $c^{n,n}_{\tau}(-\lambda)^{-1}$.
\begin{lemma}\label{c^n,n-}
Let $\tau \in \widehat{M}$ and $n \in \mathbb{Z}^{\tau}$. The function $\lambda \rightarrow c^{n,n}_\tau(-\lambda)^{-1}$ is meromorphic on the complex plane and exhibits the following properties:
\begin{itemize}
    \item[(i)] The zeros of $c^{n,n}_\tau(-\cdot)^{-1}$ are simple and occur at $\lambda \in i\Z^{\tau_-}$ such that $\Im \lambda\leq 0$.
    \item[(ii)]  The function $c^{n,n}_\tau(-\cdot)^{-1}$  has simple poles at $\lambda \in i\Z^{-\tau}$ such that $\Im \lambda<-|n|$ or $0<\Im \lambda<|n|$.
     \item[(iii)] For $p>1$ and a fixed $N\in \N$, following estimate holds
			\begin{equation}\label{est: cminusla}
				\left|	\frac{d^\alpha}{d\lambda^\alpha} c^{n,n}_\tau(-\lambda)^{-1}\right | \leq C_\alpha(1+|\lambda|)^{1/2-\alpha},
			\end{equation}
			for all integers $ \alpha = 0,1,\ldots N$, and for all $ \lambda \in \C $ with  $ 0\leq \Im \lambda \leq \gamma_p $.
   \item[(iv)] Furthermore, if  $p=1$ and $n\in \Z^{\tau_-}$, then \eqref{est: cminusla} holds true. However, if $n\in \Z^{\tau_{+}}$, then   \eqref{est: cminusla} is valid   for all $\lambda  \in \C\setminus \textbf{B}(i)$ with $0\leq \Im \lambda\leq 1$, where $\textbf{B}(i)$ is any compact neighbourhood of $i$ in the complex plane.
\end{itemize}
\end{lemma}			
\begin{proof} We have (i) and (ii) according to \cite[Proposition 6.1]{Barker}. Now, considering $p>1$, we observe from \eqref{defn_c_n,n} that we may assume $n\geq 0$. Consequently, we can write, as shown in \cite[13.13]{Barker}:
\begin{align*}
	  c^{n,n}_\tau(-\lambda)^{-1}&= -\frac{\lambda -i(n-1)}{\lambda+ i(n+1)} \,  c^{n-2,n-2}_\tau(-\lambda)^{-1} .
\end{align*} 
Hence, for $n\geq 2$, using the above formula, we obtain 
\begin{align*}\label{cn_cm}
	  c^{n,n}_\tau(-\lambda)^{-1}= \left( \prod_{ k\in \Gamma_n} p_k(\lambda) \right) c^{m,m}_\tau (-\lambda)^{-1},
\end{align*} 
where $ p_k(\lambda) =  \frac{\lambda- ik}{\lambda + ik}   $, and $ m =0 $ if $ \tau =\tau_+ $, $ m=1 $ if $ \tau=\tau_- $.   
Using the Leibniz rule, we can easily see that $p_k(\lambda)$ satisfy the following inequalities for all $\lambda$, with $0\leq \Im \lambda \leq \gamma_p$:  $$  \left|	\frac{d^\alpha}{d\lambda^\alpha}p_k (\lambda)\right | \leq  C_\alpha(1+|\lambda|)^{-\alpha}$$
 for all integers $ \alpha = 0,1,\ldots N$. With this inequality in mind, it is sufficient to prove \eqref{est: cminusla} for $c^{0,0}_\tau(-\lambda)^{-1}$ and $c^{1,1}_\tau(-\lambda)^{-1}$. Finally, the required estimates of $c^{0,0}_\tau(-\lambda)^{-1}$ and its derivatives
 \begin{align*}
     \left|	\frac{d^\alpha}{d\lambda^\alpha} c^{0,0}_\tau(-\lambda)^{-1}\right | \leq C_\alpha(1+|\lambda|)^{1/2-\alpha},
 \end{align*}
 follow from \cite[(4.2)]{Ionescu_2002} (see also \cite[(A.2)]{Ionescu_2000}), the estimates of $c^{1,1}_{\tau}(-\lambda)^{-1}$ follows similarly using the Stirling formula \cite[Chapter 4]{Titchmarsh}. This, in turn, completes the proof of (iii).  
 
To prove (iv), we note that for $p=1$, the same proof is applicable when $n\in \mathbb{Z}^{\tau_-}$ since, in this situation, the meromorphic function $p_1(\lambda)$ does not occur, and thus, there is no singularity at $\lambda=i$. However, when $n\in \mathbb{Z}^{\tau_+}$, we  exclude a neighborhood of $i$ in the complex plane to address the singularity of $p_1(\lambda)$ at $\lambda=i$, and obtain \eqref{est: cminusla}.
\end{proof}
\begin{lemma}[{{\cite[Lemma 4.2]{Stanton_and_Tomas}}}]\label{lemma:est_of_|c(-lambda)|^-2} Let $N$  be any  fixed natural number. Then we have \begin{equation}\label{estimate of |c(-lambda)|^-2}
					\left|	\frac{d^\alpha}{d\lambda^\alpha} |c^{n,n}_\tau(\lambda)|^{-2}\right | \leq C_\alpha(1+|\lambda|)^{1-\alpha},
				\end{equation} 
				for all $ \lambda \in  \R $ and $ \alpha \in [0,N] $. 
			\end{lemma}
		\begin{proof}
			  The lemma above follows  from the  explicit formula of $|c^{n,n}_\tau(\lambda)|^{-2}$, 
			 \begin{align}
			 	|c^{n,n}_\tau(\lambda)|^{-2} = \begin{cases}
			 		\frac{\lambda \pi} {2} \tanh (\lambda \pi /2) , \text{ if } \tau =\tau_+\\
			 			\frac{-\lambda \pi}{2} \coth (\lambda \pi /2) , \text{ if } \tau =\tau_-,
			 	\end{cases}
			 \end{align}
		 which can be found in \cite[(10.1)]{Barker}.
		\end{proof}

   \subsection{H\"ormander norm of symbols}\label{hormander}
Consider an open subset $U$ of the complex plane $\mathbb{C}$. We use $\mathcal{H}^{\infty}(U)$ to denote the space of bounded holomorphic functions in $U$, which is equipped with the supremum norm. Now, let $m$ be a bounded holomorphic function defined on $U$, which is continuous on the closure $\overline{U}$, including its derivatives up to the $k$th order. In \cite{FR}, the authors introduced the Mikhlin-H\"ormander norm at infinity of order $k$ on $\overline{U}$, which is given by the following expression:
\begin{equation}
    \| m\|_{\mathcal{MH}( \overline{U},k)}=\max\limits_{\alpha\in \{0,\ldots, k\}} \sup\limits_{\lambda \in \overline{U}} (1+|\lambda|)^{\alpha} \left|  \frac{\partial^\alpha}{\partial \lambda^{\alpha}} m( \lambda)  \right|.
\end{equation}
Now, let us consider a smooth function $\sigma: \mathbb{R} \times U \rightarrow \mathbb{C}$ such that, for each fixed $s\in \mathbb{R}$, the map $\lambda \mapsto \sigma(s, \lambda)$ defines a bounded holomorphic function on $U$ and remains continuous on its closure $\overline{U}$. In this context, we define the H\"ormander norm of such symbols $\sigma$ at infinity of order $(j,k)$ on $\mathbb{R} \times \overline{U}$ as follows:
 \begin{equation}\label{eqn_defn_Hormander_U}
     \| \sigma\|_{\mathcal{H}{(\overline{U},j,k)}}=\max\limits_{\substack{\alpha\in \{0,\dots, k\}\\ \beta\in \{0,\dots, j\}}}\, \sup\limits_{\substack{\lambda \in \overline{U},\, s\in \R}} (1+|\lambda|)^{\alpha} \left|  \frac{\partial^\beta}{\partial s^{\beta}}\frac{\partial^\alpha}{\partial \lambda^{\alpha}} \sigma(s, \lambda)  \right|.
 \end{equation}
 Slightly abusing this notation, we also write:
  \begin{equation}\label{eqn_defn_Hormander_R}
     \| \sigma\|_{\mathcal{H}{(\R,j,k)}}=\max\limits_{\substack{\alpha\in \{0,\dots, k\}\\ \beta\in \{0,\dots, j\}}}\, \sup\limits_{\substack{\lambda \in \R,\, s\in \R}} (1+|\lambda|)^{\alpha} \left|  \frac{\partial^\beta}{\partial s^{\beta}}\frac{\partial^\alpha}{\partial \lambda^{\alpha}} \sigma(s, \lambda)  \right|.
 \end{equation}
 We say $\sigma$ belongs to $\mathcal{H}{(\overline{U},j,k)}$, if $\sigma :\R \times U\rightarrow \C$ is a smooth function and 
$ \| \sigma\|_{\mathcal{H}{(\overline{U},j,k)}}<\infty$.
\section{Main results}\label{main_result}

In this section, we will present our main results concerning the pseudo-differential operators for $(n,n)$ type functions on $G=\mathrm{SL}(2,\R)$. Before doing so, we  first recall the multiplier results of Ricci and Wr\'{o}bel in this setting.

Throughout this article, we assume that $n\in \Z$ is a fixed integer. Let $m: \mathbb{R}\cup i\Gamma_n \rightarrow \C$ be a bounded measurable function. Then the associated Fourier multiplier operator $T_m$ for $(n,n)$ type functions on $G$ is defined through the inversion formula (see \eqref{inversion_n,n}) as follows:
\begin{align}\label{mult_(n,n)}
T_m f(x) &= \frac{1}{4\pi^2} \int_{\R} m(\lambda) \what {f_H}(\lambda) \phi_{\tau,   \lambda}^{n,n} (x) |c^{n,n}_{\tau}(\lambda)|^{-2} d\lambda 
          +\frac{1}{2\pi}\sum_{k\in \Gamma_n} m(ik) \what{ f_B}(ik) \psi_{ik}^{n,n}(x) |k|.
\end{align}
 In their work \cite{Ricci}, the authors observe that if $T_m$ is bounded on $L^p(G)_{n,n}$ for some $p\in (1,\infty) \setminus \{2\}$, then the multiplier $m$, which was initially defined on $\R \times i\Gamma_n$, must necessarily extend to a bounded even holomorphic function in $S_{p}^{\circ}$. Moreover, they demonstrated that 
 (\cite[Proposition 5.2]{Ricci})
 \begin{align*}
     \|m\|_{\mathcal{H}^{\infty}(S_p^{\circ})}\leq \| T_m\|_{L^p\rightarrow L^p}.
 \end{align*}
With the notations introduced in Section \ref{hormander}, the authors proved the following analogue of the H\"ormander-Mikhlin multiplier theorem.
\begin{theorem}[{{\cite[Theorem 5.3]{Ricci}}}]\label{thm_RW} Fix $1<p<\infty$, $p\not =2$. Assume that the multiplier $m : S_p\cup i\Gamma_n \rightarrow \C $ is a function satisfying the following properties:
\begin{itemize}
    \item[(i)] The multiplier $m$ is a bounded even holomorphic function in $S_{p}^{\circ}$.
    \item[(ii)]  $\| m\|_{\mathcal{MH}(S_p,2)}<\infty$.
\end{itemize} Then the corresponding Fourier multiplier operator $T_m$ is bounded on $L^p(G)_{n,n}$. Furthermore, there exists a constant $C_{p,n}>0$ such that
\begin{align*}
    \| T_m f\|_{L^p(G)_{n,n}} \leq C_{p,n} \left(  \| m\|_{\mathcal{MH}(S_p,2)}<\infty  + \sum_{k \in \Gamma_n}  |k| \,  |m(ik)| \right) \|f\|_{L^p},
\end{align*}
for all $f \in L^p(G)_{n,n}$.

\end{theorem}

Inspired by the results on multipliers discussed above, we naturally become interested in exploring the consequences of replacing multipliers with a more general symbol. Specifically, we aim to determine the conditions under which the associated pseudo-differential operator remains bounded. In the Euclidean space, these conditions often involve considering the regularity of the symbol $ \sigma(x,\lambda) $ with respect to both $x$ and $\lambda$, its growth properties, and its behavior at infinity. Thus finding the appropriate conditions on $ \sigma(x,\lambda) $ becomes a fundamental aspect of our analysis.  In this regard, we formally define the pseudo-differential operators in our setting to present our results. 

Let $\sigma : G\times \R\cup i\Gamma_n \rightarrow \C$ be a suitable function. We define  the associated pseudo-differential operator  $\Psi_{\sigma}$ for $(n,n)$ type functions on $G$  through the inversion formula (see \eqref{inversion_n,n}) as follows:
\begin{equation}\label{pdo_n,n_intro}
    \Psi_{\sigma} f(x) = \frac{1}{4\pi^2} \int_{\R} \sigma (x,\lambda) \what {f_H}(\lambda) \phi_{\tau,   \lambda}^{n,n} (x) |c^{n,n}_{\tau}(\lambda)|^{-2} d\lambda  + \frac{1}{2\pi} \sum_{k\in \Gamma_n} \sigma(x,ik) \what{ f_B}(ik) \psi_{ik}^{n,n}(x) |k|,
\end{equation}
 where $\tau $ is determined by $n\in \Z^{\tau}$.
We now present one of our main results on pseudo-differential operators alluded to in the introduction.
 \begin{theorem}\label{thm:pdo_for_n,n_type}
		Let  $ p\in (1,\infty)\setminus\{2\}$.  Suppose that  $ \sigma: G \times  S_{p}\cup i\Gamma_n \rightarrow \C  $ is a function satisfying the
        following properties: 
\begin{itemize}
    \item[(i)] For each $\lambda\in S_p\cup\, i\Gamma_n$, $ x \mapsto \sigma(x,\lambda) $  is a $K$-biinvariant function on $ G $ and $\|\sigma\|_{L^{\infty}(G\times i \Gamma_n)}$ is finite.
   \item[(ii)] For each $x \in G$, $ \lambda \mapsto \sigma(x,\lambda) $ is an even holomorphic function on the strip $S_p^\circ$.  
  \item[(iii)]  For each $\overline{v} \in \overline{N}$, the function $(s,\lambda) \mapsto \sigma_{\overline{v}}(s,\lambda):=\sigma(\overline{v} a_s, \lambda)$ belongs to ${\mathcal{H}{(S_p,2,2)}}$ and  \begin{equation*}
       \sup_{\ol v\in \ol {N}}\| \sigma_{\overline{v}}\|_{\mathcal{H}{(S_p,2,2)}} <\infty,
  \end{equation*}
        where $\| \sigma\|_{\mathcal{H}{(S_p,2,2)}}$ is defined  as in \eqref{eqn_defn_Hormander_U}.
  \end{itemize}
  Then the operator $ \Psi_\sigma $ extends to a bounded operator on $ L^p{(G)}_{n,n} $ to itself for all $ p\in (1,\infty)\setminus\{2\}$. Moreover, there exists a constant
            $C_{p,n}>0$ such that
        \begin{align*}
            \|\Psi_{\sigma} f\|_{L^p(G)}\leq C_{p,n} \left( \sup\limits_{\overline{v}\in \overline{N}}   \| \sigma_{\overline{v}}\|_{\mathcal{H}{(S_p,2,2)}} + \|\sigma \|_{L^{\infty}(G \times i\Gamma_n)} \right)  \|f\|_{L^p(G)},
        \end{align*}
    for all $f \in L^p(G)_{n,n}.$ 
    \end{theorem}
\begin{remark}
    \begin{enumerate}        
        \item  We recall that in the case of multipliers, the holomorphic extension property of the multiplier is necessary for the multiplier operator to be bounded on $L^p(G)_{n,n}$. Taking this into account, along with the results on pseudo-differential operators in Euclidean spaces, it is natural to assume that the symbol satisfies the holomorphicity condition stated in Theorem \ref{thm:pdo_for_n,n_type} to establish the boundedness of the associated pseudo-differential operator.
        \item  Next, we compare Theorem \ref{thm:pdo_for_n,n_type} with the corresponding result on rank one symmetric spaces of noncompact type. In \cite[Theorem 1.6]{PR_PDO_22}, the authors established the $L^{p}$-boundedness (for $p\in(1,\infty)\setminus\{2\}$) of the pseudo-differential operators on symmetric spaces by assuming, among other things, the following condition:
        \begin{equation}\label{eqn:hypothesis_of_sigma_sym}
        \left| \frac{\partial^\beta}{{\partial s}^\beta}\frac{\partial^\alpha}{\partial \lambda ^\alpha} \sigma (g a_s,\lambda)\right| \leq {C_{\alpha,\beta}}{(1+|\lambda| )^{-\alpha}},
        \end{equation}
            for all $ g \in G, s\in \R$ and $ \lambda \in S_p $, where the order of the mixed partial derivatives is up to a prescribed order.
            In our case, due to the $K$-biinvariant hypothesis on the symbol, we can observe that the corresponding condition \eqref{eqn:hypothesis_of_sigma_sym} for the symbol in the symmetric space simplifies to hypothesis (iii) of Theorem \ref{thm:pdo_for_n,n_type}. However, there is a difference in the behavior of our integral kernel $\mathcal{K}$ of the operator $\Psi_{\sigma}$ compared to the one in symmetric spaces. Our kernel $\mathcal{K} $ is no longer a $K$-biinvariant function with respect to the second variable, which means that the argument used in \cite{PR_PDO_22} cannot be directly applied in our setting.
        
        \item We only need regularity condition on the space variable of the symbol $\sigma(x,\lambda)$ when $\lambda$ lies on the real line. There is no need for any regularity condition on the space variable $x$ for all $\lambda \in S_p$. Please see Remark \ref{rem_glo}.

        \item  As the spherical functions are of $(n,n)$ type, we observe from \eqref{pdo_n,n_intro} that for any given $f \in C_c^{\infty}(G)_{n,n}$, if $\sigma(\cdot,\lambda)$ is any fixed $K$-type function then  $\Psi_\sigma(f)$ cannot be of $(n,n)$ type unless $\sigma(\cdot,\lambda)$ is a $K$-biinvariant function. This demonstrates that the $K$-biinvariant nature of the symbol is necessary to establish the boundedness of $\Psi_\sigma$ in $L^p(G)_{n,n}$.
    \end{enumerate}
\end{remark}

We now shift our focus to the boundedness problem of the pseudo-differential operator $\Psi_{\sigma}$ associated with non-smooth symbols in the space variables. Building upon the works of Kenig and Staubach \cite{carlos_07}, we introduce the symbol class $\mathcal{S}^m_{\varrho,\infty}(S_p)$ for rough symbols in our context.

The symbol class $\mathcal{S}^m_{\varrho,\infty}(S_p)$, where $1<p<\infty$, $m\in \R$, and $0\leq\varrho \leq 1$, includes all functions $\sigma : G \times S_p \cup i\Gamma_n \rightarrow \C$, that satisfy the following properties:
\begin{enumerate}
    \item For each $\lambda\in S_p\cup i\Gamma_n$, the function $x\mapsto \sigma(x,\lambda)$ is $K$-biinvariant, measurable on $G$ and $\|\sigma\|_{L^{\infty}(G\times i \Gamma_n)}<\infty$.
    \item For each $x\in G$, the function $\lambda \mapsto \sigma(x,\lambda)$ is an even holomorphic function on $ S_p^\circ$ and satisfies the differential inequalities:
\begin{align}\label{nonreg_ineq}
\left\| \frac{\partial^{\alpha}}{\partial \lambda ^\alpha} \sigma(x,\lambda) \right\|_{L^{\infty}_x}\leq C_{\alpha} (1+|\lambda|)^{m-\varrho\alpha},
\end{align}
for all $\lambda \in S_p$, $\alpha \in \N$, where $C_{\alpha}>0$ are constants. 
\end{enumerate}
Additionally, we define the symbol class $\mathcal{S}^m_{\varrho,\infty}(S_1)$, which consists of functions $\sigma : G \times S_1 \rightarrow \C$ that satisfy properties (1) and (2) for $p=1$, and further have the property $\frac{\partial^\alpha}{\partial \lambda^{\alpha}}\sigma(x,\lambda)|_{\lambda=i}=0,$  for all $\alpha\in \{0,1,2\}.$

    \begin{exmp} 
       Let $1<p<\infty$.  For $\nu \in \C$, let us define
        $$\sigma_{\nu}(g, \lambda) = \left(\frac{\lambda^2+1}{4}\right)^{i[g]^+ - i\nu}, \quad g\in G, \lambda \in S_p,$$
    where we recall that $[g]^+$ denotes the $\R^+$ component of $g$ in the Cartan decomposition (see Section \ref{sec:The_group_SL(2,R)}). By a simple computation, it follows that $\sigma_{\nu} $ satisfies \eqref{nonreg_ineq}. Now, to check the other condition, we observe that $g \mapsto [g]^{+}$ is a $K$-biinvariant function. Thus, we have  $\sigma_{\nu} \in  \mathcal{S}^{ 2\Im \nu}_{1,\infty}(S_p)$, for all $p \in (1,\infty)\setminus \{2\}$.  More generally, one can consider the following class of symbols of the form 
 $$\sigma_{\nu}(g, \lambda) = \left(\frac{\lambda^2+1}{4}\right)^{i\eta(g) - i\nu}, \quad g\in G, \,\lambda \in S_p,$$
where   $\eta$ is any real-valued $K$-biinvariant function on $G$.
    \end{exmp}
    
We now present a boundedness result for the associated pseudo-differential operator $\Psi_\sigma$ when $\sigma \in \mathcal{S}^m_{\varrho,\infty}(S_p)$, which serves as an analogue of Theorem \ref{thm_carlos} established by Kenig and Staubach \cite{carlos_07} in the Euclidean spaces.
    
    \begin{theorem}\label{thm_nonreg}
        Let $p\in [1,2)\cup (2,\infty)$. Suppose that $m<0$ and $\sigma= \sigma (x,\lambda) \in  \mathcal{S}^m_{1,\infty}(S_p)$. Then we have the following:
        \begin{enumerate}
            \item The operator $\Psi_{\sigma}$ is a bounded operator from $L^p(G)_{n,n}$ to itself for all $p\in (1,2)\cup (2,\infty)$.
            \item  When $p=1$ and $n\in \Z^{\tau_-}\cup \{0\}$, then $\Psi_{\sigma}$  is a bounded operator from $L^1(G)_{n,n}$ to itself. If $n\in \Z^{\tau_+}\setminus\{0\}$, $\Psi_{\sigma}$   is a weak type $(1,1)$-bounded operator from $L^1(G)_{n,n}$ to $L^{1,\infty}(G)_{n,n}$.
            \end{enumerate}
            \end{theorem}
 \begin{remark}   We note that similar to Theorem \ref{thm_carlos}, the number of derivatives needed in the dual variable is not infinite. In fact, by following the proof of Theorem \ref{thm_nonreg}, one can demonstrate that the required number of derivatives is finite, although it does depend on the value of $m$ as in \eqref{nonreg_ineq}.
    \end{remark}
 
\section{Generalized Coifman-Weiss transference principles}\label{sec_transference_principle_weak}
 Let $ \{ X, \mathfrak{m} \} $ be a measure space. An operator $ \mathcal{B} $ on $ L^p(X,\mathfrak{m}) $ is said to be of weak-type ($ p,p $) if it maps $ \phi\in L^p(X,\mathfrak{m}) $ into a measurable function defined on a measure space $ (Y, \nu) $ in such a way that for each $ s > 0 $,
 \begin{align*}
 	 \nu\{ y \in Y : |(\mathcal{B} \, \phi)(y)| >s \}\leq [ \mathcal{C}\| \phi \|_{L^p(X)}/s]^p,
 \end{align*}
 where $ \mathcal{C} $ is independent of $ \phi\in L^p(X) $.
 
  Let  $ G $ be a locally compact group satisfying the following property: Given a compact subset $ B $ of $ G $ and $ \epsilon>0 $, there exists an open neighborhood $ V $ of the identity $ e $ having finite measure such that
 \begin{align}
 	 \frac{\mu (B^{-1}V) }{\mu (V)} \leq 1+\epsilon,
 \end{align}
where $ \mu $ is, say, left Haar measure on the group $G$. Let us suppose, further, that  $ \mathcal{R} $  is a representation consisting of measure preserving transformations of the space $X$.  Since $\mathcal{R}_u$ is measure-preserving 
 \begin{align}
 	\int_{X} |  f (\mathcal{R}_u x)|^p d\mathfrak{m}(x) = \int_{X} |f(x)|^p d\mathfrak{m}(x)  
 \end{align}
 for all $u \in G$.
The transformation we will consider is of the form 
\begin{align}\label{eqn_kernel_type}
	(Tf)(x)= \int_G k(x, \mathcal{R}_{u} x, u) f(\mathcal{R}_u x) d\mu(u),
\end{align}
where $ k(x,y,u) $ is a measurable function on $X\times X \times G  $ which is $0 $ if $u  $ does not belong to a compact set $ B\subset G $. Moreover, we assume that for each $x\in X$, the kernel 
 \begin{equation*}
 	{k}_x (v,u) := {k}(\mathcal{R}_v x, \mathcal{R}_{u^{-1}}\mathcal{R}_v x, u)={k}(\mathcal{R}_v x, \mathcal{R}_{u^{-1}v}x, u)
 \end{equation*}
 satisfies 
 \begin{equation}\label{eqn:req_est_for_pi}
 	\left( \int_G \left|  \int_G {k}_x(v,u) h (u^{-1}v ) d\mu(u) \right|^p d\mu(v)  \right)^{\frac{1}{p}} \leq \mathcal{C} \left(\int_G  \left| h(u)\right|^p d\mu(u)\right)^{\frac{1}{p}},	
 \end{equation}
 for all  $ h\in L^p(G) $, where $ \mathcal{C} $ is independent of $ x \in X $.  Then the authors in \cite[p. 292, (2.7) ]{coifman_73} proved that the operator $ T $  is   bounded from $ L^p({X}) $ to itself with norm not exceeding $ \mathcal{C} $  that is:
 \begin{equation}\label{eqn:Lp_est_for_Pi}
 	\left(\int_{ {X}}  \left| {T} f(x)\right|^p dx\right)^{\frac{1}{p}}	  \leq \mathcal{C} \left(\int_{{X} } \left| f(x)\right|^p dx\right)^{\frac{1}{p}}, \quad 
 \end{equation}
 \text{for all $f \in L^p(X).$}  We present the following weak-type version of the transference principle mentioned above. This version also extends the result of \cite[Theorem 2.6]{coifman_77} and is of independent interest.
\begin{theorem}\label{thm_transference_weak}
	Let $ p\in [1,\infty) $. Assume that $ k(x,y,u) $ is a measurable function on $X\times X \times G  $ as in \eqref{eqn_kernel_type}, and also satisfies the following for all $ s>0, $
	\begin{align}\label{est_hypothesis_on_k}
		 \mu \left\{ v \in G  :  \left| \int_G {k}_x (v,u) h( u^{-1}v) d\mu(u) \right|  >s \right\} \leq \frac{\mathcal{C}^p}{s^p} \|h\|_{L^p(G)}^p,
	\end{align} 
where $ \mathcal{C} $ is not depending on $``x" $ and $ h\in L^p(G) $. Then the operator $ T $ defined in \eqref{eqn_kernel_type} is of weak type $ (p,p) $. Moreover, the following is true,
\begin{align}
	 \mathfrak{m} \left\{  x \in X : |(Tf)(x)| > s\right\} \leq \frac{\mathcal{C}^p} {s^p}\|f\|^p_{L^p(X)},
\end{align}
for all $ f \in L^p(X) $.
\end{theorem}
\begin{proof}
	 Let  us define \begin{align*}
	 	   \xi (s) & = \left\{ x \in X : \left|\int_G k(x, \mathcal{R}_{u^{-1}} x, u) f(\mathcal{R}_{u^{-1}} x) du\right| >s  \right\}, \\
	 	    \xi_v (s) & = \left\{ x \in X : \left|\int_G k(\mathcal{R}_v x, \mathcal{R}_{u^{-1} v} x, u) f(\mathcal{R}_{u^{-1} v} x) du\right| >s  \right\}, \\
	 	    \mathcal{F}(s) & = \left\{(v,x) \in G \times X : \left|\int_G k(\mathcal{R}_v x, \mathcal{R}_{u^{-1} v} x, u) f(\mathcal{R}_{u^{-1} v} x) du\right| >s  \right\}.
	 \end{align*}  
Observe that $\xi (s)=\xi_e(s) $ and moreover,  since $ \mathcal{R}_v $ is measure preserving,  we have 
 \begin{align}\label{eqn_measure_preserve_xi}
 	\mathfrak{m} (\xi(s)) = \mathfrak{m}(\mathcal{R}_v \xi(s))= \mathfrak{m} (\xi_v(s)).
 \end{align}
 Let $ \chi(v,x)  $ be the characteristic function on $ \mathcal{F}(s) $ and $ \psi $  be the characteristic function of $ B^{-1}V $ (thus $\psi(u^{-1}v)=1$ when $u \in B$ and $v\in V$). We note that if we fix $v$, then $ \chi(v,x)  $ is the characteristic function on $ \xi_v (s) $.  Now, integrating both sides of the equation above and using the Fubini theorem, we get
 \begin{align*}
 	 \mathfrak{m} (\xi(s))& =  \frac{1}{\mu (V)} \int_V  \mathfrak{m} (\xi_v(s)) \,d\mu(v)\\
 	 & =  \frac{1}{\mu (V)} \int_V \int_{X} \chi(v,x) \, d\mathfrak{m}(x)\,d\mu(v) \\
 	 & =  \frac{1}{\mu (V)}  \int_{X} \mu  \left\{  v \in G\cap V : \left|\int_G k(\mathcal{R}_v x, \mathcal{R}_{u^{-1} v} x, u) f(\mathcal{R}_{u^{-1} v} x) d\mu(u)\right| >s  \right\}   d\mathfrak{m}(x)\\
 	  & =  \frac{1}{\mu (V)}  \int_{X} \mu  \left\{  v \in G\cap V : \left|\int_G k(\mathcal{R}_v x, \mathcal{R}_{u^{-1} v} x, u) f(\mathcal{R}_{u^{-1} v} x) \psi(u^{-1} v) d\mu(u)\right| >s  \right\} \\ &\hspace{10cm} \cdot  d\mathfrak{m}(x)\\
& \leq   \frac{1}{\mu (V)}  \int_{X} \mu  \left\{  v \in G : \left|\int_G k(\mathcal{R}_v x, \mathcal{R}_{u^{-1} v} x, u) f(\mathcal{R}_{u^{-1} v} x) \psi(u^{-1} v) d\mu(u)\right| >s  \right\}   d\mathfrak{m}(x)\\
& \leq  \frac{1}{\mu (V)} \left( \int_{X}  \frac{\mathcal{C}^p}{s^p}  \int_G |f(\mathcal{R}_{w} x)|^p \psi (w) d\mu (w) \right) d\mathfrak{m}(x) \quad \text{(using \eqref{est_hypothesis_on_k})}\\
& \leq \frac{1}{\mu (V)}  \frac{\mathcal{C}^p}{s^p}   \int_G  \| f\|_{L^p(X)}^p \psi (w) d\mu (w)   \quad  (\text{since $\mathcal{R}_{w}$ is measure-preserving)} \\
& \leq \frac{\mu (B^{-1} V)}{\mu (V)} \frac{\mathcal{C}^p}{s^p}   \| f\|_{L^p(X)}^p\\
&\leq (1+\epsilon)  \frac{\mathcal{C}^p}{s^p}   \| f\|_{L^p(X)}^p,
 \end{align*}
  where we used \eqref{est_hypothesis_on_k} by \text{taking $h(u^{-1}v) = f(\mathcal{R}_{u^{-1} v} x) \psi(u^{-1} v) $}.
Since $\epsilon>0$ is arbitrary, this concludes the proof of Theorem \ref{thm_transference_weak}.
\end{proof}
	\section{Properties of spherical functions} \label{sec_phi_m,n} 
 In this section, we will derive some functional identities and asymptotic estimates of the spherical functions. 
  \subsection{Functional identities for the spherical functions}
    We start with the following formula.
  \begin{proposition} \label{phi_m,n_id}
    We have   
\begin{align*}\label{eqn:identity-phi}
         \phi_{\tau,\lambda}^{n,n}(y^{-1}x) 
         &= \int_K e^{-(i\lambda+1)  H(y ^{-1} k)} e^{(i\lambda-1)H(x^{-1}k)}  e_n(K({x^{-1}k})^{-1}) \overline{e_n(K(y ^{-1} k)^{-1})}   \,dk,
         \end{align*}
for all $x, y\in G$ and $\lambda\in\C$.
  \end{proposition}
  \begin{proof}
      We recall from the Iwasawa decomposition and use it for $xk= K(xk) \exp(H(xk)) n_1$. We can write
     \begin{equation}\label{H(yxk)}
      \begin{aligned}
          y^{-1}x k&= y ^{-1} K(xk) exp(H(xk)) n_1\\
           & = K(y ^{-1} K(xk)) \exp(H(y ^{-1} K(xk)))n_2 \exp(H(xk)) n_1.
      \end{aligned}
     \end{equation}
       Since $A$ normalizes $N$, we get from \eqref{H(yxk)},
    \begin{equation*}
  y^{-1}x k = K(y ^{-1} K(xk)) \exp\left(H(y ^{-1} K(xk))+H(xk)\right) n_3,
    \end{equation*} 
    which in turn implies
    \begin{equation}\label{H(yxk)=}
        H( y^{-1}x k) = H(y ^{-1} K(xk))+ H(xk),
    \end{equation}
    and 
    \begin{equation}\label{K(yxk)=}
        K( y^{-1}x k) =   K(y ^{-1} K(xk)) .
    \end{equation}
    Plugging \eqref{H(yxk)=} and \eqref{K(yxk)=} in \eqref{definition of phi sigma lamda m, n}, we get
    \begin{align*}
        \phi_{\tau,\lambda}^{n,n}(y^{-1}x)&= \int_K e^{-(i\lambda+1)H(y^{-1}xk)} e_n(k^{-1}) \,\overline{e_n(K(y^{-1}xk)^{-1})}\, dk\\
        &= \int_Ke^{-(i\lambda+1) \left( H(y ^{-1} K(xk)+H(xk)\right)}  e_n(k^{-1})\, \overline{e_n(K(y ^{-1} K(xk))^{-1})} \,dk.
    \end{align*}
    But again by \eqref{H(yxk)=} and \eqref{K(yxk)=} (putting $y=x$), we get
    \begin{align}
        H(xk) = - H(x^{-1}K(xk)) \quad \text{and} \quad k= K(x^{-1}K(xk)),
    \end{align}
    whence we obtain
    \begin{align*}
         \phi_{\tau,\lambda}^{n,n}(y^{-1}x)&\\=& \int_K e^{-(i\lambda+1) \left( H(y ^{-1} K(xk)-H(x^{-1}K(xk))\right)}  e_n(K({x^{-1}K(xk))}^{-1}) \,\overline{e_n(K(y ^{-1} K(xk))^{-1})} \,dk.
    \end{align*}
    Now we apply Lemma \ref{k(xk)_dif} to get
      \begin{align*}
         \phi_{\tau,\lambda}^{n,n}(y^{-1}x)&= \int_K e^{-(i\lambda+1) \left( H(y ^{-1} k) -H(x^{-1}k)\right)}  e_n(K({x^{-1}k})^{-1}) \,\overline{e_n(K(y ^{-1} k)^{-1})}  e^{-2H(x^{-1}k)}\,dk\\
         &= \int_K e^{-(i\lambda+1)  H(y ^{-1} k)} e^{(i\lambda-1)H(x^{-1}k)}  e_n(K({x^{-1}k})^{-1})\, \overline{e_n(K(y ^{-1} k)^{-1})}   \,dk,
         \end{align*}
         which concludes the proof of our lemma.
  \end{proof}
 \begin{lemma}\label{lem_fnal_id}
      The spherical function $ \phi_{\tau,\lambda}^{n,n}$ satisfies the following identity
      \begin{equation}\label{phimn_id3}
          \int_K  \phi_{\tau,\lambda}^{n,n}(ykx) e_n(k^{-1}) \, dk = \phi_{\tau,{\lambda}}^{n,n}(y^{-1}) \phi_{\tau,\lambda}^{n,n}(x),
      \end{equation}
  for all $\lambda \in \R$ and $x,y \in G$.
  \end{lemma}
  \begin{proof}
      We can write from Proposition \ref{phi_m,n_id},
      \begin{align*}
           \phi_{\tau,\lambda}^{n,n}(ykx)= \int_K e^{-(i\lambda+1)  H(yk k_1)} e^{(i\lambda-1)H(x^{-1}k_1)}  e_n(K({x^{-1}k_1})^{-1}) \overline{e_n(K(yk k_1)^{-1})}   \,dk_1.
      \end{align*}
      Using the formula above, the left-hand side of \eqref{phimn_id3} transforms to
      \begin{align*}
            \int_K   \left(  \int_K e^{-(i\lambda+1)  H(yk k_1)} e^{(i\lambda-1)H(x^{-1}k_1)}  e_n(K({x^{-1}k_1})^{-1}) \overline{e_n(K(yk k_1)^{-1})}   \,dk_1\right)  e_n(k^{-1})\, dk.
      \end{align*}
      We now apply Fubini's theorem, followed by the change of variable $ k k_1 \rightarrow k$ in the expression above, to get 
        \begin{align*}
             &\int_K   \left(  \int_K e^{-(i\lambda+1)  H(yk)}  e_n(k^{-1}) \overline{e_n(K(yk )^{-1})}  \,dk\right) e^{(i\lambda-1)H(x^{-1}k_1)}  e_n\left(K({x^{-1}k_1})^{-1}\right) e_n(k_1)\, dk_1\\
             & = \phi_{\tau,\lambda}^{n,n}(y) \int_K    e^{(i\lambda-1)H(x^{-1}k_1)}  e_n\left(K({x^{-1}k_1})^{-1}\right) e_n(k_1)\, dk_1\\
               & = \phi_{\tau,\lambda}^{n,n}(y)   \phi_{\tau,-{\lambda}}^{-n,-n}(x^{-1})\\
               &=  \phi_{\tau,\lambda}^{n,n}(y)  \phi_{\tau,\lambda}^{n,n}(x)
        \end{align*}
        where in the last step, we used \eqref{phi_n,n-lam}. This completes the proof of the lemma
  \end{proof}

\subsection{Asymptotic estimates of spherical functions} In this section, we will discuss the asymptotic estimate of the spherical functions. We begin by presenting the local and global expansions of $\phi^{n,n}_{\tau,\lambda}$ in terms of well-known special functions. It is important to note that the spherical functions exhibit different behaviors near and away from the identity, and this distinction will be evident in the expansions.
 
 Let $J_\mu(z)$ be the Bessel functions of the first kind, and let
           \begin{equation}\label{defn_J_mu} 
		 \mathcal J_\mu(z)=\frac{J_\mu(|z|)}{|z|^\mu} \Gamma(\mu + \frac{1}{2}) \Gamma(\frac{1}{2}) 2^{\mu-1}.
		 \end{equation}
		 Then we have the following asymptotic expansion of the spherical function $ \phi_{\tau,\lambda}^{n,n} $ near identity.
			\begin{lemma}\label{thm:localexpansion-phi}
				For  $ 0\leq t\leq 1  $, the spherical function  $ \phi_{\tau,\lambda}^{n,n} $ can be decomposed as 
				\begin{equation*}
					 \phi_{\tau,\lambda}^{n,n} (a_t) = \left(\frac{t}{\Delta(t)}\right)^{1/2}\sum_{j=0}^2 t^{2j} b_j^n (t) \mathcal{J}_j(\lambda t) + E_n(\lambda ,t ), \quad \lambda \geq 0, 
				\end{equation*}
			where $ b_0^n(t) \equiv b_0  $ is a constant independent of $ n  $, while $ |b_j^n(t)| \leq C_n  $, $ j=1,2 ,$ and 
			\begin{equation*}
				 \int_{1}^{\infty }|E_n(\lambda,t )| \lambda d\lambda \leq C_n, \quad \text{ uniformly in } 0\leq t\leq 1.
			\end{equation*}
			\end{lemma}			
			\begin{proof}
			    We recall that $\phi_{\tau,\lambda}^{n,n}$ is a solution of the following differential  equation 
\begin{equation*}
	 \Omega f =-\frac{ (\lambda^2 +1)}{4}f,
\end{equation*}
where $ f $ is a $ (n,n) $ type function on $ G. $  Since $ \Omega $ preserves
the types of a function and an $ (n, n) $ type function is uniquely
determined by its restriction to $ (0, \infty) $, we obtain an operator
$ \Pi_{n,n}(\Omega) $ on $ (0, \infty) $ which satisfies
\begin{equation*}
\Omega f(a_t)= \Pi_{n,n}(\Omega)(f\mid_{A^+})(a_t).
\end{equation*}
The operator $\Pi_{n,n}(\Omega)$, known as the $ (n,n) $-th radial component of the Casimir, is a second-order differential operator, and its expression is given by  \cite[Theorem 13.1]{Barker}:
\begin{equation*} 
\Pi_{n,n}(\Omega) f(a_t) = \frac{1}{4}\frac{d^2 }{dt^2} f(a_t) + \frac{1}{2} \coth 2t \frac{d}{dt}   f(a_t) + \frac{1}{4}  \frac{n^2}{\cosh^2 t} f(a_t), \quad t>0.
\end{equation*}
We substitute $f(t)=  (\cosh t)^n g(t)$ in
\begin{equation} \label{Equation of eigen function of Pi {n,n}}
\Pi_{n,n}(\Omega) f = -\frac {(\lambda^2+1) } {4}f 
\end{equation}
to obtain the following,
\begin{equation}\label{ODE of the function g}
\frac{d^2g }{dt^2} + ((2n+1)\tanh t + \coth t ) \frac{dg}{dt}   + ( {\lambda}^2+ {(n+1)}^2)g = 0.
\end{equation}
After performing the change of variable $z:= -\sinh^2 t $, the ODE \eqref{ODE of the function g} simplifies to the hypergeometric differential equation 
\begin{equation*}
z(1-z) \frac{d^2g}{dz^2}  + (c-(a+b+1)z) \frac{dg}{dz} - \frac{1}{4} ab g=0 \label{Hypergeometric equation reduced},
\end{equation*}
with parameters $a=\frac{n+1}{2}- \frac{i\lambda}{2}, b=\frac{n+1}{2}+ \frac{i\lambda}{2}, c=1 $. Hence, using the expression from \cite[(2.2)]{MR0774055}, we find that the Jacobi function $\varphi_{\lambda}^{(0,n)}$ is the unique solution satisfying regularity conditions and equaling $1$ at $z=0$. By employing the integral representation of the Jacobi function from \cite[(2.21)]{MR0374832}, we deduce from \eqref{Equation of eigen function of Pi {n,n}} the following: 
\begin{align*}
    \phi_{\tau,\lambda}^{n,n} (a_t) &=\frac{2^{\frac{3}{2}}}{\pi} \int_{0}^t \cos (\lambda s) \left( \cosh 2t-\cosh 2s \right) ^{-\frac{1}{2}} {}_2 F_1 \left(n, -n;1/2;\frac{\cosh t -\cosh s}{2 \cosh t} \right) \,ds.
\end{align*}   

Subsequently, by performing similar calculations as presented in \cite[(4.12)]{Ricci}, we can arrive at our lemma.
\end{proof}
The following lemma is a counterpart of Ionescu's result \cite[Propositin A.2 (c)]{Ionescu_2000} in our context. It provides an estimate for the spherical function away from the identity, which will be utilized in the large-scale analysis of $\Psi_\sigma$.		
			\begin{lemma} \label{propn:global_expansion_of_phi_lambda}
				Suppose that $ t\geq 1/10 $ and $ N\in\N $. Then $ \phi_{\tau,\lambda}^{n,n}(a_t) $ can be written in the following form
				\begin{equation}\label{eqn:phi-lambda-assymp-expansion}
					\phi_{\tau,\lambda}^{n,n}(a_t)= e^{- t} \left( e^{i\lambda t}c^{n,n}_\tau(\lambda) (1+a(\lambda,t)) +e^{-i\lambda t} c^{n,n}_\tau(-\lambda) (1+a(-\lambda,t)) \right),
				\end{equation}
				where the function $ a(\lambda,t) $ satisfies the following inequalities,
				\begin{equation}\label{eqn:est_for_a(lambda,t)}
					\left|  \frac{\partial^\alpha}{\partial \lambda^\alpha}   a(\lambda,t) \right| \leq C_n (1+| \lambda | )^{-\alpha},
				\end{equation}
				for all integers $ \alpha\in [0,N] $, and for all $ \lambda  $ in the region $ 0\leq \Im \lambda \leq  1+1/10. $
			\end{lemma}
		\begin{proof}
			From \cite[(13.1)]{Barker}, we have for all $ t>0 $,
			\begin{equation*}
				\phi_{\tau, \lambda}^{n,n} (a_t) = e^{-t} \left[ e^{-i\lambda t} c_{\tau}^{n,n} (\lambda) (1+a(\lambda,t)) + e^{i\lambda t} c_{\tau}^{n,n} (-\lambda) (1+a(-\lambda,t))   \right],
 			\end{equation*}
 		where \begin{eqnarray*}
a(\lambda,t) =\sum_{k=1}^{\infty} a^{n,n}_{k}(\lambda)e^{-2kt},
 		\end{eqnarray*}
 	and the functions $ a^{n,n}_{k} $ satisfy the following recursion relation  for $ \lambda\in \C\setm i\Z, $
 	\begin{equation}\label{eqn:recusrsive_reln_of_a_k}
 		 a^{n,n}_{k}(\lambda) = -\frac{1}{k(k-i \lambda)} \left( \sum_{j=1,odd}^k a^{n,n}_{k-j} (\lambda) j  n^2 - \sum_{j=2,even}^k a^{n,n}_{k-j}  (\lambda )  (1-i\lambda+2k-2j+{j}n^2)  \right)
 	\end{equation}
 with $ a_{0}^{n,n} \equiv1 $.
 We define \begin{align*}
 	  \alpha_{j}^k (\lambda)=\begin{cases}
 	  	-\dfrac{jn^2}{k(k-i\lambda)}, \quad  \text{ when $ j $ is odd  and $ 1\leq j\leq k $} \\
 	  	\dfrac{1}{k} \left( 1+ \dfrac{ 1+ k-2j +jn^2}{k-i\lambda}\right), \quad \text{ when $ j $ is even}.
 	  \end{cases}
 \end{align*}
			  Our  aim is to show that for all $ \beta\in [0,N] $, there exists constants $ C_n $ and $ A_\beta $ such that \begin{equation}\label{eqn:est of a_k}
			 	\left| \frac{\partial^\beta}{\partial\lambda^\beta}   a^{n,n}_{k}(\lambda) \right| \leq C_n k^{A_\beta} (1+|\Re \lambda| )^{-\beta}
			 \end{equation}
		 for all $ k\geq 1 $  and  all $ \lambda\in \C $, with $ 0\leq \Im \lambda\leq 1+1/10.  $ Using the following inequality above
		 \begin{equation*}
            |k-i\lambda|^2  = |k+\Im \lambda -i \Re \lambda|^2 =(k+\Im \lambda)^2+ \Re \lambda ^2 \geq \max\{ k, |\Re \lambda|\}^2,
		 \end{equation*}
	  we can directly say that 
			 \begin{eqnarray}\label{eqn:est_of_alpha_j^k}
			 	\left| \frac{\partial^\beta}{\partial\lambda^\beta}   \alpha_j^k(\lambda) \right| \leq \frac{C_n}{k|k-i\lambda|^\beta} \leq \frac{C_n}{k(1+|\Re \lambda|)^\beta}
			 \end{eqnarray}
		 for all integers $ k\geq 1 , j\leq k-1 $, $ \beta\in [0,N] $ and  for all $ \lambda\in \C $, with $ 0\leq \Im \lambda\leq 1+1/10.  $
		 
		 \noindent We now prove \eqref{eqn:est of a_k} for $ \beta=0 $ by induction over $ k\geq1 $. For $ k=1 $,
		 \begin{align*}
		 	 \left| a^{n,n}_{1} (\lambda) \right|= |\alpha^1_1(\lambda)| \leq \frac{n^2}{|1-i\lambda|} \leq C_n \text{ (depending on $ n $)}.
		 \end{align*}
			 Next, we assume \eqref{eqn:est of a_k} holds for all $ 1\leq j\leq k-1 $. Then using \eqref{eqn:recusrsive_reln_of_a_k}, \eqref{eqn:est_of_alpha_j^k}  and \eqref{eqn:est of a_k} we get,
			 \begin{align*}
			 	 \left|  a^{n,n}_k(\lambda) \right|  & \leq  \left( \sum_{j=1,odd}^k |a^{n,n}_{k-j} (\lambda)| |\alpha_{j}^k|+ \sum_{j=2,even}^k |a^{n,n}_{k-j}  (\lambda )|   |\alpha_{j}^k| \right)\\
			 	 &  \leq \frac{C_n}{k} \left( \sum_{j=1,odd}^k |a^{n,n}_{k-j} (\lambda)| + \sum_{j=2,even}^k |a^{n,n}_{k-j}  (\lambda )|  \right)\\
			 	 & \leq \frac{C_n}{k} \sum_{j=1}^{k-1}    \frac{ C_n j^{A_0}}{A_0}\\
			 	 &  \leq  \frac{C_n}{k}   \frac{ C_n k^{A_0+1}}{A_0} =  C_n k^{A_0} \frac{C_n}{A_0},
			 \end{align*}
		 where we used the fact for all $ k\geq 2 $, $ A\geq 4 $,
		 \begin{equation*}
		 	 1+ \sum_{j=1}^{k-1}  j^{A}\leq \frac{k^{A+1}}{A}.
		 \end{equation*}
	 Now if we choose $ A_0 =C_n$, then by induction, we have proved \eqref{eqn:est of a_k} for $ \beta=0 $ and for all $ k\geq 1 $.
	 
	 Next, to prove \eqref{eqn:est of a_k} for arbitrary integer $ \beta \leq N $, we assume by induction that we found suitable powers $ A_\beta $, such that \eqref{eqn:est of a_k} holds for all $0\leq  \alpha\leq \beta-1 $ and for all $ k\geq 1 $. We can also assume $ A_0\leq A_1\leq \cdots \leq A_{\beta-1} $. Again we apply induction over $ k $ for fixed $ \beta $.  Its obvious that $ \left|\frac{\partial^\beta }{\partial \lambda^\beta}a_1^{n,n}(\lambda) \right| \leq C_n (1+|\Re \lambda|) ^{-\beta} $. We assume \eqref{eqn:est of a_k} holds for $ \beta $ and for all $ j\in \{ 1,2.\cdots k-1\} $.  Then again \eqref{eqn:recusrsive_reln_of_a_k}, \eqref{eqn:est_of_alpha_j^k}  and \eqref{eqn:est of a_k} imply
	 
	 \begin{align*}
	 	\left|\frac{\partial^\beta }{\partial \lambda^\beta}a_k^{n,n}(\lambda) \right|&  \leq  2^{\beta}    \left( \sum_{j=1}^{k-1}  \sum_{\alpha=0}^\beta \left| \frac{\partial^{\beta-\alpha} }{\partial \lambda^{\beta-\alpha}}a^{n,n}_{k-j} (\lambda)\right| \left| \frac{\partial^\alpha }{\partial \lambda^\alpha}\alpha_{j}^k\right|\right)\\
	 	&\leq  2^{\beta}    \sum_{\alpha=0}^\beta \left( \sum_{j=1}^{k-1}  \frac{C_n}{k(1+| \Re \lambda| )^{\beta-\alpha} }\left| \frac{\partial^\alpha }{\partial \lambda^\alpha}\alpha_{j}^k\right|\right)\\
	 	& \leq  2^{\beta}    \sum_{\alpha=0}^\beta \frac{C_n}{k(1+| \Re \lambda| )^{\beta-\alpha} } \left( \sum_{j=1}^{k-1}\frac{C_n\max\{ j,1\}^{A_\alpha}}{(1+|\Re \lambda|)^\alpha }\right)\\
	 	&\leq   2^{\beta}    \sum_{\alpha=0}^\beta \frac{C_n}{(1+| \Re \lambda| )^{\beta} } \frac{C_n k^{A_\beta}}{A_\beta}\\
	 	& \leq C_n k^{A_\beta} (1+|\Re \lambda|) ^{-\beta} \frac{C_n2^\beta (\beta+1)}{A_\beta}.
	 \end{align*}
 By taking  $ A_\beta \geq \max \{ A_{\beta-1}, C_n 2^\beta (\beta+1)\},$ the proof follows by induction. In fact one can set $ A_\beta =C_n 2^\beta(\beta+1) $ for all integers $ \beta\in [0,N] $.
		\end{proof}

We now focus on the estimate of the canonical discrete series matrix coefficients $\psi_{ik}^{n,n}$. We recall the following uniform growth properties of $\psi_{ik}^{n,n}$  from \cite[Theorem 8.1]{Barker}.
\begin{theorem}\label{thm_est_of_psi_k}
    Fix $l \in \mathbb{N}$. There exist constants $C, r_1, r_2, r_3 \geq 0$ such that
    \begin{equation}\label{eqn_est_of_psi_k}
        \left| \psi_{ik}^{n,n}(x)  \right| \leq C (1+ |n|)^{r_1}(1+|k|)^{r_2}(1+x^+)^{r_3}\phi_{\tau^+,0}^{0,0}(x)^{(1+l)}
    \end{equation}
    for all $k \in \Z^*$ for which $|k|\geq l$, and for all $n\in \Z(k)$, for all $t \geq 0$.
\end{theorem}
We remark that the estimate above is a consequence of a more general result by Trombi and Varadarajan, where they found a necessary condition on a discrete series representation having the $K$-finite matrix coefficient with a certain rate of decay; see \cite[Theorem 8.1]{trombi_1972}. Later, in 1977, Mili\v{c}i\'c proved that their condition was sufficient, too, which, in turn 
provided a precise characterization of discrete series representations whose $K$-finite matrix coefficients lie in $L^p(G')$ for $1\leq p<2$, where $G'$ is a connected semisimple Lie group with finite center; see \cite[Theorem, p.60]{mililic_1977} for more details. In the setting of $\mathrm{SL(2,\R)}$ group, the result can be stated as follows:
\begin{corollary}[{{\cite[Corollary, p.84]{mililic_1977}}}] \label{cor_millic}
Let $0<p\leq 2$ and $\gamma_p = (2/p)-1$. If $(\pi_{ik}, H_k)$ is a discrete series representation corresponding to $k\in \Z^*$, and $\psi$ is a $K$-finite matrix coefficient of $\pi_{ik}$, then the following conditions are equivalent:
\begin{enumerate}
    \item $|k|> \gamma_p$,
    \item $\psi \in L^p(G)$.
\end{enumerate}
\end{corollary}

\begin{remark}\label{lem_est_of_psi_k}
It also follows from the estimate in \cite[(3.2)]{Barker} for $\phi_{\tau^+,0}^{0,0}$, which exhibits exponential decay, and \eqref{eqn_est_of_psi_k}, that the matrix coefficients of the discrete series representation belong to $L^q(G)$ for all $q>1$. More precisely, for $q>1, k \in \Z^*$, and $n \in \Z(k)$,  we have
    \begin{align}\label{lem_Lqest_of_psi_k}
        \|\psi_{ik}^{n,n}\|_{L^q(G)} \leq C_q  (1+ |n|)^{r_1}(1+|k|)^{r_2}.
    \end{align}
\end{remark}
The previous result illustrates that, given a $k\in \Z^*$ and $p \in (0,2]$, how we can examine whether a $K$-finite matrix coefficient of $\pi_{ik}$ belongs to $L^p(G)$ or not. However, it does not say anything about the case $ |k|=\gamma_p$. So, to accommodate the case $|k|= \gamma_p$, we provide the following weak-type version of the result \cite[Corollary, p. 84]{mililic_1977} for the $(n, n)$-th matrix coefficient of $\pi_{ik}$.
 \begin{lemma}\label{lem_psi_k_Lp}
    Let $0<p\leq 2$ and $\gamma_p = (2/p)-1$. If $(\pi_{ik}, H_k)$ is a discrete series representation corresponding to $k\in \Z^*$, and  $n\in \Z(k)$, then the following conditions are equivalent:
\begin{enumerate}
    \item $|k|\geq \gamma_p$,
    \item $\psi_{ik}^{n,n} \in L^{p,\infty}(G)$.
\end{enumerate}
Moreover, if $|k|= \gamma_p$, then $\psi_{ik}^{n,n} \in L^{p,q}(G)$, if and only if $q=\infty$.
\end{lemma}
Before we give the proof of the lemma above, let us first provide an asymptotic estimate for the discrete series coefficient.
\begin{lemma} 
    Given $k\in \Z^*$, $n\in \Z(k)$, we have
     \begin{equation}\label{psi_k_est}
         |\psi_{ik}^{n,n}(a_t)|\simeq e^{-(1+|k|)t}, \quad t\geq 0 .
     \end{equation}
 \end{lemma}
 \begin{proof}
   Suppose $k \in \Z^*$ and $n\in \Z(k)$, then from \cite[(12.2), Chap. 12]{Barker} we have 
     \begin{equation}\label{psi_asy}
         \psi_{ik}^{n,n} (a_t)  = 2 e^{-(1+|k|)t } \sum_{l=0}^{\infty} d_l^{n,n}(k) \,e^{-2lt}, \quad t>0.
     \end{equation}
     Using the fact that  $\psi_{ik}^{n,n} = \phi_{\tau, i|k|}^{n,n}$ (see \ref{relation between psi n,n and phi {n,n}}), Barker showed that the  coefficient $d_l^{n,n}(k)$ as in \eqref{psi_asy}  satisfies the following relation (see \cite[Proposition 14.2]{Barker})
     \begin{align*}
         d_l^{n,n}(k)=\begin{cases}
             c^{n,n}_\tau (k) a^{n,n}_l (k) \quad \text{if }k>0,\\
             c^{n,n}_\tau (-k) a^{n,n}_l (-k) \quad \text{if }k<0,
         \end{cases}
     \end{align*}
     where $ a^{n,n}_l (k)$'s are defined in \eqref{eqn:recusrsive_reln_of_a_k}. We note that the poles of $\lambda\mapsto c^{n,n}_\tau (\lambda)$ occur only at $\lambda \in i\Z^{\sigma},\, \Im \lambda \geq 0$, where $\sigma$ is determined by $n\in \Z^{\sigma}$ (see Lemma \ref{c^n,n-}). Therefore, using \cite[Proposition 13.5]{Barker}, we can say that for a given $k\in \Z^*$, $n\in \Z(k)$, and for any $r>0$ there exists a constant $C_k>0$ such that for all $l\in \N$
     \begin{align*}
         |d_l^{n,n}(k)|\leq C_k e^{rl}.
     \end{align*}
     Utilizing the inequality above for   $r<1$, it follows from \eqref{psi_asy} that, 
    \begin{align*}
         \lim_{t\rightarrow \infty} \left|e^{(1+|k|)t} \psi_{ik}^{n,n}(a_t)- 2d_0^{n,n}(k)\right| =0,
    \end{align*}
      whence we obtain
     \begin{equation*} 
         |\psi_{ik}^{n,n}(a_t)|\simeq e^{-(1+|k|)t}, \quad \text{for all } t\geq 0.
     \end{equation*}
 \end{proof}
\begin{proof}[\text{Proof of Lemma \ref{lem_psi_k_Lp}}]
     Let us define for $k\in \N$, $$f_k(t) =e^{-(k+1)t} \quad \text{for } t\geq 0,$$ and extend it as a $(n,n)$-type function on $G$ as shown in \eqref{extension a function defined on closure of A+ to G}.   Then  in view of \eqref{psi_k_est}, the asymptotic estimate of   $\psi_{ik}^{n,n}$, it is sufficient to prove that $f_k \in L^{p,\infty}(G//K)$ if and only if $k\geq \gamma_p$.  To establish this, we will first determine the distribution function of $f_k$. Taking into account that $m$ is the Haar measure on $G$ in the polar decomposition and observing that $f_k$ is a $K$-biinvariant function, we arrive at the following expressions:
\begin{align*} 
d_{f_k} (\alpha)& =m \left\{t \in [0,\infty) :  e^{-(k+1)t}> \alpha \right\} =  m \left\{t \in [0,\infty) :  t< \frac{1}{k+1} \log\frac{1}{\alpha}\right\}.
	\end{align*}
 Clearly, if $\alpha\geq 1$, $d_{f_k} (\alpha) =0$, hence we only need to consider $\alpha\in (0,1)$. Let $0<\alpha <e^{-(k+1)/2}$, then $\frac{1}{k+1} \log\frac{1}{\alpha}>1/2$. Thus in this range of  $\alpha$, using the asymptotic behaviour of $\Delta(t)$ we can write
 \begin{align}\label{df_k}
     d_{f_k} (\alpha) \simeq \int^{\frac{1}{2}}_0 t \,dt+  \int_{\frac{1}{2}}^{\frac{1}{k+1}\log\frac{1}{\alpha}} e^{2t}\, dt \simeq {\alpha^{-\frac{2}{k+1}}}.
 \end{align}
 When $e^{-(k+1)/2}<\alpha<1 $,   we can show similarly
 \begin{align*}
     d_{f_k} (\alpha) \simeq  \log {\left(\frac{1}{\alpha}\right)}^{\frac{2}{k+1}}.
 \end{align*}
Thus, from the definition of Lorentz space \eqref{defn_lor}, it follows that
\begin{align*}
  \| f\|_{L^{p,\infty}} =\sup_{\alpha>0}\, \alpha^{p} \, d_{f_k} (\alpha) \simeq \sup_{0<\alpha<1} {\alpha^{p-\frac{2}{k+1}}},
\end{align*}
which in turn implies $\|f_k\|_{L^{p,\infty}}<\infty$ if and only if $p-\frac{2}{k+1}\geq 0$, that is, $k\geq \gamma_p$. Finally, we  will complete the proof of the lemma by contradiction. Let us assume that for a given $k=\gamma_p$, $f_k \in L^{p,q}(G)$ for some $q\in (0,\infty)$.  Then we have from \eqref{df_k} and the definition of Lorentz space \eqref{defn_lor},
\begin{align*}
  \| f\|_{L^{p,q}}^q \simeq \int_0^{e^{-(k+1)/2}} \alpha^{q-1} {\alpha^{-\frac{2q}{p(k+1)}}}   d\alpha \simeq \int_0^{e^{-(k+1)/2}} \alpha^{q-1} {\alpha^{-q}} d\alpha=\infty,
\end{align*}
which contradicts our assumption $f_k \in L^{p,q}(G)$, for $q\in (0,\infty)$, concluding the lemma. \end{proof}
 \begin{remark}
By following a similar calculation as presented in Lemma \ref{lem_psi_k_Lp},  we can establish that any matrix coefficient $\psi_{ik}^{m,n}$  of $\pi_{ik}$, where $m,n\in \Z(k)$ (see \cite[(5.1)]{Barker} for definition), satisfies the following asymptotic estimate:
\begin{align}\label{est_psi_mc}
    |\psi_{ik}^{m,n}(a_t)|\simeq   e^{-(1+|k|) t }, \quad t\geq 0.
\end{align}
Since any $K$-finite matrix coefficient $\psi$ of $\pi_{ik}$ can be represented as a finite linear combination of $\{\psi_{ik}^{m,n}: m,n\in \Z(k) \}$, the matrix coefficients of $\pi_{ik}$ (see \cite[Corollary 5.5]{Barker}), we can combine the arguments from the proof of Lemma \ref{lem_psi_k_Lp} with the estimate in \eqref{est_psi_mc}, to establish the following theorem. This theorem offers a precise characterization of discrete series representations whose $K$-finite matrix coefficients belong to the space $L^{p,\infty} (G)$ for $0<p\leq 2$. This characterization serves as a weak-type analog of Mili\v{c}i\'c's result \cite[Corollary, p.84]{mililic_1977} in the context of $\mathrm{SL}(2,\R)$; see also \cite[Theorem, 5.3]{Barker}. For $0<p\leq 2$, let $E_{p,\infty} (G)$ denote the set of equivalence classes of irreducible representations of $G$ whose $K$-finite matrix coefficient belongs to $L^{p,\infty}(G)$.
\end{remark}
 \begin{theorem} 
    Let $0<p\leq 2$ and $\gamma_p = (2/p)-1$. If $(\pi_{ik}, H_k)$ is a discrete series representation corresponding to $k\in \Z^*$, then the following conditions are equivalent:
\begin{enumerate}
    \item $|k|\geq \gamma_p$,
    \item  $(\pi_k,H_k)$ is in $E_{p,\infty} (G)$.
\end{enumerate}
\end{theorem}

 \section{Singular integral realization of $\Psi_{\sigma}$}\label{decomposition of PSi}

In studying the boundedness of the pseudo-differential operator $\Psi_\sigma$ on symmetric spaces, it is customary to analyze the local and global components of the operator separately. However, in this group setting, there is an additional discrete component of the Plancherel measure. To address this, we will decompose the operator $\Psi_\sigma$ into continuous and discrete parts.  To deal with the continuous part, we will first represent it as a singular integral kernel operator.

Let $f$ be a smooth, compactly supported $(n,n)$ type function on $G$. Then we recall the definition of $\Psi_{\sigma}$ from \eqref{pdo_n,n_intro}
\begin{align}\label{eqn_decom_of_psi_sigma}
\Psi_{\sigma} f(x) &= \frac{1}{4\pi^2} \int_{\R} \sigma (x,\lambda) \what {f_H}(\lambda) \phi_{\tau,   \lambda}^{n,n} (x) |c^{n,n}_{\tau}(\lambda)|^{-2} d\lambda  + \frac{1}{2\pi} \sum_{k\in \Gamma_n} \sigma(x,ik) \what{ f_B}(ik) \psi_{ik}^{n,n}(x) |k| \notag \\
        &=: \Psi_{\sigma}^{\text{con}} f(x) + \Psi_{\sigma}^{\text{dis}} f(x).
\end{align}   
By substituting the expression of $ \what {f_H}(\lambda)$ in the above expression of $ \Psi_{\sigma}^{\text{con}} f$ , we get
\begin{align*}
     \Psi_{\sigma}^{\text{con}} f(x) &=  \frac{1}{4\pi^2} \int_{\R} \sigma (x,\lambda) \left( \int_G f(y) \phi_{\tau,\lambda}^{n,n}(y^{-1}) dy  \right)  \phi_{\tau,   \lambda}^{n,n} (x) |c^{n,n}_{\tau}(\lambda)|^{-2} d\lambda \\
     &= \frac{1}{4\pi^2} \int_{\R} \sigma (x,\lambda) \left( \int_G f(y) \phi_{\tau,\lambda}^{n,n}(y^{-1}) \phi_{\tau,   \lambda}^{n,n} (x) dy  \right)   |c^{n,n}_{\tau}(\lambda)|^{-2} d\lambda.
\end{align*}
By applying the formula for the spherical function $\phi_{\tau,\lambda}^{n,n}$ from Lemma \ref{lem_fnal_id}, the expression above simplifies to:
 \begin{align*}
      \Psi_{\sigma}^{\text{con}} f(x)   &= \frac{1}{4\pi^2} \int_{\R} \sigma (x,\lambda) \left( \int_G f(y) \int_K \phi_{\tau,\lambda}^{n,n}(y^{-1}kx) e_n(k^{-1})\, dk  \,dy  \right)   |c^{n,n}_{\tau}(\lambda)|^{-2} d\lambda.
 \end{align*}
Using Fubini's theorem, followed by the change of variable $y\rightarrow ky$, and taking into account that $f$ is a $(n,n)$ type function on $G$ and $\int_K dk = 1$, we can further simplify this expression as follows:
\begin{align*}
    \Psi_{\sigma}^{\text{con}} f(x)  
        &= \int_{G} f(y) \mathcal{K}^{\text{con}}(x,y^{-1}x) dy,
\end{align*}
where 
\begin{equation}\label{defn_K_con}
\mathcal{K}^{\text{con}}(x,y) :=  \frac{1}{4\pi^2}  \int_{\R }\sigma (x,\lambda) \phi_{\tau, \lambda}^{n,n}(y ) |c^{n,n}_{\tau}(\lambda)|^{-2} d\lambda.
\end{equation}
Similarly, for $\Psi_{\sigma}^{\text{dis}} f $, we can write
\begin{align*}
    \Psi_{\sigma}^{\text{dis}} f(x)= \int_G f(y)  \mathcal{K}^{\text{dis}}(x,y^{-1}x) dy, 
\end{align*}
where \begin{eqnarray*}
    \mathcal{K}^{\text{dis}}(x,y) := \frac{1}{2\pi} \sum_{k\in \Gamma_n} \sigma(x,ik)   \psi_{ik}^{n,n}(y) |k|.
\end{eqnarray*}

In Section \ref{sec_phi_m,n}, we observed that the spherical function $\phi_{\tau,\lambda}^{n,n}$ behaves differently near the identity element compared to its behavior away from it. To account for this distinction and handle the exponential volume growth of $G$, we will decompose $\Psi_{\sigma}^{\text{con}}$ into a sum of local and global parts. To achieve this decomposition, we introduce a smooth, even function $\eta^{\circ}: \R \rightarrow [0,1]$, which is supported on $[-1,1]$ and satisfies $\eta^{\circ}(t) = 1$ if $|t| \leq 1/2$. We define another function $\eta(t) = 1 - \eta^{\circ}(t)$. By utilizing the Cartan decomposition, we extend the functions $\eta^{\circ}$ and $\eta$ to become $K$-biinvariant functions on $G$, by
\begin{equation}\label{defnb_etas}
	\eta^{\circ}(x)=\eta^{\circ}(x^+) \quad \text{and} \quad\eta(x)=\eta(x^+), \quad  \text{ for all } x\in G. 
\end{equation} 
Then we can write,
\begin{equation}\label{eqn_decom_psi_cont}
	\Psi_{\sigma}^{\text{con}} f(x) = \Psi_{\sigma}^{\text{loc}} f(x) + \Psi_{\sigma}^{\text{glo}} f(x),
\end{equation}
where 
    \begin{equation}\label{eqn:defn_of_T_sigma_local}
        \begin{aligned}
    	\Psi_{\sigma}^{\text{loc}}f(x)&:=	\int_{G}  f(y)  \eta^{\circ}(y^{-1} x) \mathcal{K}^{\text{con}}(x,y^{-1} x) dy,\\
            & = \int_{G}  f(y)   \mathcal{K}^{\text{loc}}(x,y^{-1} x) dy,\\
        \end{aligned}
    \end {equation} 
and
    \begin{equation}\label{eqn:defn_of_T_sigma_global}
        \begin{aligned}
    	\Psi_{\sigma}^{\text{glo}}f(x)&:=	\int_{G}  f(y)  \eta(y^{-1} x) \mathcal{K}^{\text{con}}(x,y^{-1} x) dy,\\
            & = \int_{G}  f(y)   \mathcal{K}^{\text{glo}}(x,y^{-1} x) dy.
        \end{aligned}
    \end {equation}

The decomposition of $\Psi_{\sigma}$, as given by \eqref{eqn_decom_of_psi_sigma} and \eqref{eqn_decom_psi_cont}, simplifies the task of proving the boundedness of $\Psi_{\sigma}$. It reduces the problem to demonstrating the boundedness of the individual operators $\Psi_{\sigma}^{\text{dis}}$, $\Psi_{\sigma}^{\text{loc}}$, and $\Psi_{\sigma}^{\text{glo}}$, each of which requires a distinct approach to establish their boundedness. In the remainder of this article, we will examine and discuss the boundedness properties of these operators, $\Psi_{\sigma}^{\text{dis}}$, $\Psi_{\sigma}^{\text{loc}}$, and $\Psi_{\sigma}^{\text{glo}}$, separately.
 
\section{Discrete analysis of  $\Psi_\sigma$}\label{sec_dis_part}
    In this section, we will demonstrate the boundedness of $\Psi^{\text{dis}}_{\sigma}$. While the authors \cite[Lemma 4.4]{Ricci} employed the convolution estimate from the \textit{Kunze-Stein} phenomenon to establish the $L^p$-boundedness of multipliers, our approach utilizes the uniform growth properties of the canonical discrete series matrix coefficients $ \psi_{ik}^{n,n} $ near infinity (Theorem \ref{thm_est_of_psi_k}). This property of $\psi_{ik}^{n,n}$  enables us to show that for all $k\in \Z^*$ and $n\in \Z(k)$,  $ \psi_{ik}^{n,n} \in L^p(G) $ for any $p\in (1,\infty)$ (see \eqref{lem_Lqest_of_psi_k}), leading to the $L^p$-boundedness of the $\Psi^{\text{dis}}_{\sigma}$ operator.

However, the situation is different for $p=1$. Corollary \ref{cor_millic} indicates that a similar technique will not work in this case, as $ \psi_{ik}^{n,n} $ does not belong to $L^1(G)$ when $k =1$ (then $n>0$ is even). Nevertheless, we utilize Lemma \ref{lem_psi_k_Lp} to establish the weak type $(1,1)$-boundedness of $\Psi^{\text{dis}}_{\sigma}$.
\begin{theorem}
     Let $1\leq p<\infty$ and $\Psi_{\sigma}^{\text{dis}}$ be the operator defined   in \eqref{eqn_decom_of_psi_sigma}. Assume that $$\|\sigma\|_{L^{\infty}(G\times i \Gamma_n)} =\sup_{x\in G,\, k\in \Gamma_n} |\sigma(x,ik)|<\infty.$$ Then the following are true
    \begin{enumerate}
            \item The operator $\Psi_{\sigma}^{\text{dis}}$ is a bounded operator from $L^p(G)_{n,n}$ to itself for all $p\in (1,\infty)$.
            \item  When $p=1$ and $n\in \Z^{\tau_-}\cup \{0\}$, then $\Psi_{\sigma}^{\text{dis}}$  is a bounded operator from $L^1(G)_{n,n}$ to itself. If $n\in \Z^{\tau_+}\setminus\{0\}$, $\Psi_{\sigma}^{\text{dis}}$   is a weak type $(1,1)$-bounded operator from $L^1(G)_{n,n}$ to $L^{1,\infty}(G)_{n,n}$.
            \end{enumerate}
\end{theorem}
\begin{proof}
First, let us proceed with the case $1<p<\infty$. Let $f\in L^{p}(G)_{n,n}$, then using H\"{o}lder's inequality and \eqref{lem_Lqest_of_psi_k}, we obtain
\begin{align*}
    \left| \what{ f_B}(ik)\right| = \left| \int_G f(x) {\psi_{ik}^{n,n}}(x^{-1}) dx \right| \leq \|f\|_{L^p(G)} \|\psi_{ik}^{n,n}\|_{L^{p'}(G)}, 
\end{align*}
for all $k \in \Gamma_n$. Therefore, using the inequality above, we can write the following.
\begin{align*}
     \|\Psi_{\sigma}^{\text{dis}} f\|_{L^p(G)}^p &\leq C_{p,n} \sum_{k\in \Gamma_n} |\what{ f_B}(ik)|^p \int_{G} |\sigma(x,ik)  \psi_{ik}^{n,n}(x) k|^p dx \\
                    &\leq C_{p,n} \|\sigma\|_{L^{\infty}(G\times i \Gamma_n)}^p \sum_{k\in \Gamma_n} \|f\|_{L^p(G)}^p \|\psi_{ik}^{n,n}\|_{L^{p'}(G)}^p \int_{G} | \psi_{ik}^{n,n}(x)|^p dx\\
                    &\leq C_{p,n} \|\sigma\|_{L^{\infty}(G\times i \Gamma_n)} ^p \|f\|_{L^p(G)}^p \sum_{k\in \Gamma_n} \|\psi_{ik}^{n,n}\|_{L^{p'}(G)}^p \|\psi_{ik}^{n,n}\|_{L^{p}(G)}^p\\
                    &\leq C_{p,n} \|\sigma\|_{L^{\infty}(G\times i \Gamma_n)}^p \|f\|_{L^p(G)}^p. 
\end{align*}
This establishes the $L^p(G)_{n,n}$-boundedness of $\Psi^{\text{dis}}_{\sigma}$ for $1 < p < \infty$. In the case of $p = 1$ and $n \in \Z^{\tau-} \cup \{0\}$, the proof for the $L^1(G)_{n,n}$-boundedness of $\Psi_{\sigma}^{\text{dis}}$ follows the same procedure as in the previous case. Lastly, for $p = 1$ and $n \in \Z^{\tau_+} \setminus \{0\}$, by utilizing H\"{o}lder's inequality and Lemma \ref{lem_psi_k_Lp}, we obtain the following for $f \in L^1(G)_{n,n}$:
\begin{align*}
     \|\Psi_{\sigma}^{\text{dis}} f\|_{L^{1,\infty}(G)} &\leq C_{n} \sum_{k\in \Gamma_n} |\what{ f_B}(ik)|  \|\sigma(\cdot,ik)  \psi_{ik}^{n,n}(\cdot) k\|_{L^{1,\infty}(G)}  \\
                    &\leq C_{n} \|\sigma\|_{L^{\infty}(G\times i \Gamma_n)}  \sum_{k\in \Gamma_n} \|f\|_{L^1(G)} \|\psi_{ik}^{n,n}\|_{L^{\infty}(G)} \| \psi_{ik}^{n,n}\|_{L^{1,\infty}(G)} \\
                    &\leq C_{n}  \|\sigma\|_{L^{\infty}(G\times i \Gamma_n)}  \|f\|_{L^1(G)} \sum_{k\in \Gamma_n} \|\psi_{ik}^{n,n}\|_{L^{\infty}(G)} \|\psi_{ik}^{n,n}\|_{L^{1,\infty}(G)}\\
                    &\leq C_{n}  \|\sigma\|_{L^{\infty}(G\times i \Gamma_n)}  \|f\|_{L^1(G)},
\end{align*}
whence we obtain the weak type $(1,1)$-boundedness of $\Psi^{\text{dis}}_{\sigma}$.
\end{proof}

\section{Local analysis of $\Psi_{\sigma}$}\label{Analysis on the local part}

In the case of the multiplier operator, the local part can be expressed as a convolution with a compactly supported function. By using a Coifmann-Weiss transference principle for convolution, the authors in \cite{Ionescu_2000,Ricci} related the multiplier operator to the Euclidean multiplier. Then, the boundedness of the multiplier operator follows from the Mikhlin multiplier theorem on $\R$. However, a fundamental difference arises between the multiplier case and our current situation with the pseudo-differential operator. In the case of the multiplier operator, writing it as a convolution operator plays a crucial role, but the presence of an extra variable $x$ in the symbol $\sigma(x, \lambda)$ of the pseudo-differential operator prevents us from using the theory of multipliers.

We establish the boundedness of $\Psi_{\sigma}^{\text{loc}}$ by employing a generalized Coifman-Weiss transference principle in Section \ref{sec_transference_principle_weak} for the kernel integral operator. This principle allows us to establish a connection between the $L^p$-boundedness of $\Psi_{\sigma}^{\text{loc}}$   and Euclidean pseudo-differential operators. Our analysis of the local part of $\Psi_\sigma$ leads to the following result.
	 
    \begin{theorem}\label{thm:Lp_boundedness_of_T_sigma_local} Let $\Psi_{\sigma}^{\text{loc}}$ be the operator defined   in \eqref{eqn:defn_of_T_sigma_local}. Assume that $\sigma : G \times \R \mapsto \C$ is a function satisfying the following properties:
    \begin{enumerate}
        \item For each $\lambda \in \R,$  $x\mapsto \sigma(x,\lambda)$ is a $K$-biinvariant function on $G$.
        \item For each $x \in G,\, \lambda \mapsto \sigma(x,\lambda)$ is an even function on $\R$.
        \item For each $\overline{v} \in \overline{N}$, the function $(s,\lambda) \mapsto \sigma_{\overline{v}}(s,\lambda):=\sigma(\overline{v} a_s, \lambda)\in {\mathcal{H}{(\R,2,2)}} $ and  
        \begin{equation}\label{hyp_sigma_local}
            \sup\limits_{\overline{v}\in \overline{N}}   \| \sigma_{\overline{v}}\|_{\mathcal{H}{(\R,2,2)}} <\infty,
        \end{equation}
        where $\| \sigma\|_{\mathcal{H}{(\R,2,2)}}$ is defined in \eqref{eqn_defn_Hormander_R}.
     
    \end{enumerate}
Then, for $1<p<\infty$, $\Psi_{\sigma}^{\text{loc}}$ is bounded from $L^{p}(G)_{n,n}$ to itself. Moreover, there exists a constant $C_{p,n}>0$ such that
\begin{align}\label{result_loc}
    \|\Psi_{\sigma}^{\text{loc}} f\|_{L^{p}(G)}\leq C_{p,n} \left( \sup\limits_{\overline{v}\in \overline{N}}   \| \sigma_{\overline{v}}\|_{\mathcal{H}{(\R,2,2)}} \right) \|f\|_{L^{p}(G)},
    \end{align}
for all $f\in L^{p}(G)_{n,n}$.  
  \end{theorem} 
  \begin{remark}
      We note that the norm of $\sigma_{\ol v}$ in \eqref{result_loc} is taken in $\mathcal{H}{(\R,2,2)}$ not in $\mathcal{H}{(S_p,2,2)}$.
  \end{remark}
  Before delving into the proof of Theorem \ref{thm:Lp_boundedness_of_T_sigma_local}, let us first obtain a quantitative estimate of $\| \Psi_\sigma^{\text{loc}} f \|_{L^p(G)}$ for $f \in C_c^{\infty} (G)_{n,n}$ and $p \in [1,\infty)$. We will use this estimate to establish the $L^p$-boundedness of the operator  $\Psi_{\sigma}^{\text{loc}}$. Let us recall the definition \eqref{eqn:defn_of_T_sigma_local} of $\Psi_{\sigma}^{\text{loc}}$ and proceed with a change of variable, leading to the following expression:
  \begin{align*}
      \| \Psi_\sigma^{\text{loc}} f \|_{L^p(G)}= \left( \int_{G} \left|  \int_{G} f(y) \mathcal{K}^{\text{loc}}(x, y^{-1}x) dy  \right|^{p}  dx \right)^{\frac{1}{p}} =  \left( \int_{G} \left|  \int_{G} f(xy) \mathcal{K}^{\text{loc}}(x, y^{-1}) dy  \right|^{p}dx  \right)^{\frac{1}{p}}.
  \end{align*}	
 By writing the integral formula \eqref{definition of integration} corresponding to the Cartan decomposition, the right-hand side of the equation above transforms into 
 \begin{align*}
     &\left( \int_{G} \left| \int_{K} \int_{\R} f(xka_t) \mathcal{K}^{\text{loc}}(x,  (ka_t)^{-1})  |\Delta(t)| dk dt    \right|^p dx \right)^{\frac{1}{p}}\\
     &\leq \int_{K} \left( \int_{G} \left|  \int_{\R} f(xka_t)  \mathcal{K}^{\text{loc}}(x,  a_t)  |\Delta(t)|  dt    \right|^p dx\right)^{\frac{1}{p}}dk,
 \end{align*}
 where in the last step, we employed  Minkowski’s integral inequality and used the fact that $x \mapsto \mathcal{K}^{\text{loc}}(x,y)$ is an $(n,n)$ type function on $G$.  Next, by the change of variable $xk \rightarrow x$, and utilizing the fact that $K$ is a compact subgroup of $G$ with $\int_K dk = 1$, we obtain the following:
 \begin{align}\label{Psi_Lp_est}
     \| \Psi_\sigma^{\text{loc}} f \|_{L^p(G)} \leq  	& \left( \int_{G}  \left|  \int_{\R} f(\mathcal {R}_{t} x) \,     \mathcal{K}^{\text{loc}}(x, a_t) |\Delta(t)|  dt    \right|^p dx\right)^{\frac{1}{p}},
 \end{align}
 where  $ \{\mathcal{R}_t: t \in \R\} $  are the representations consisting of measure-preserving transformations of the space $G$, defined by
	\bes
 \mathcal{R}_t x :=xa_t, \quad   x \in G.
	\ees 
 In particular, the representations $ \{\mathcal{R}_t: t \in \R\} $ preserve the $L^p$ norm of any function on $G$, that is for any $f \in L^p(G)$, we have
 \begin{align*}
      \|f(\mathcal{R}_t x )\|_{L^p(G)}= \|f\|_{L^p(G)}, \quad \text{for all } t\in \R.
 \end{align*}
 Next, we will apply the generalized Coifman-Weiss transference principle to the integral operator on the right-hand side of \eqref{Psi_Lp_est} in order to establish the $L^p(G)_{n,n}$-boundedness of the operator $\Psi_{\sigma}^{\text{loc}}$.

\subsection{Application of generalized Coifman-Weiss transference principle} 
The proof of Theorem \ref{thm:Lp_boundedness_of_T_sigma_local} follows as a consequence of the generalized Coifman-Weiss transference principle in Section \ref{sec_transference_principle_weak} and the following corollary. This corollary plays a crucial role in applying the Coifman-Weiss transference principle in our specific context, as it provides the explicit estimate required to use   \eqref{eqn:req_est_for_pi} in our setting. Let us define the Coifman-Weiss kernel $\mathcal{K}^{\text{CW}}: G \times G\times \R \rightarrow \C$ by
\begin{align}\label{def_CW_kernel}
\mathcal{K}^{\text{CW}}(x,y,t) =\mathcal{K}^{\text{con}}(x,y^{-1}x) \eta^\circ (t) |\Delta(t)|.
\end{align} 
\begin{corollary}\label{cor_trans_cw}
    For $ x\in G $, let us recall   $\mathcal R_s x=xa_s$, $ s\in \R $. Assume that there exists a constant $\mathcal{C}>0$ such that for $1<p<\infty$,
		\begin{equation}\label{eqn:Coifman-Weiss_condn_satisified_by_K}
			\left(	\int_{\R} \left| \int_{\R} \mathcal{K}^{\text{CW}}(\mathcal{R}_s x, \mathcal{R}_{-t} \mathcal{R}_s x , t)   h(s-t) dt\right|^{p} ds\right) ^{\frac{1}{p}}\leq \mathcal{C} \| h\|_{L^p(\R)} 
		\end{equation}
		 for all $ h\in L^p(\R) $. Then we have 
    $$\|\Psi_{\sigma}^{\text{loc}} f\|_{L^{p}(G)}\leq \mathcal{C}  \|f\|_{L^{p}(G)},$$
    for all $\gamma>0$, $f \in L^p(G)_{n,n}$, where $\mathcal{C}$ is the same constant as in \eqref{eqn:Coifman-Weiss_condn_satisified_by_K}.
\end{corollary}
\begin{proof}
   Let us define the operator \begin{equation}\label{cw_our}
       (Tf)(x)= \int_{\R} \mathcal{K}^{\text{CW}}(x, \mathcal{R}_t x,t  )    f(\mathcal{R}_t x) \, dt, \quad \text{for } f\in C_c^{\infty}(G)_{n,n}.
   \end{equation}
   As the function $t\mapsto \mathcal{K}^{\text{CW}}(x, \mathcal{R}_tx,t)$ is compactly supported, we can compare \eqref{cw_our} with the Coifmann-Weiss transference operator in \eqref{eqn_kernel_type}. Furthermore, equation \eqref{eqn:Coifman-Weiss_condn_satisified_by_K} indicates that the kernel of the operator $T$  satisfies \eqref{eqn:req_est_for_pi} and the hypothesis of Theorem \ref{thm_transference_weak}. Thus, in view of  \eqref{Psi_Lp_est}  and by applying the generalized Coifman-Weiss transference principle \eqref{eqn:req_est_for_pi}, we establish our corollary.
\end{proof}
\subsection{Kernel of the Coifman-Weiss transference operator}\label{kernel_CW} The corollary above simplifies the proof of Theorem \ref{thm:Lp_boundedness_of_T_sigma_local}, since we now only need to focus on establishing \eqref{eqn:Coifman-Weiss_condn_satisified_by_K} with 
\begin{align*}
    \mathcal{C}=    \sup\limits_{\overline{v}\in \overline{N}}   \| \sigma_{\overline{v}}\|_{\mathcal{H}{(\R,2,2)}} .
\end{align*} 
The remaining part of this section will be dedicated to proving the inequality \eqref{eqn:Coifman-Weiss_condn_satisified_by_K}, assuming the hypotheses of Theorem \ref{thm:Lp_boundedness_of_T_sigma_local}.
 We observe from the definition \eqref{def_CW_kernel} of  the Coifman-Weiss kernel $\mathcal{K}^{\text{CW}}$ that 
 \begin{equation*} 
 \begin{aligned}
		\mathcal{K}^{\text{CW}}(\mathcal{R}_s x, \mathcal{R}_{-t} \mathcal{R}_s x , t) &=   \mathcal{K}^{\text{con}}(xa_s,  a_{t}) \eta^{\circ}(t) \Delta(t), \quad t\geq 0.
\end{aligned}
\end{equation*}
By recalling the expression \eqref{defn_K_con} of $\mathcal{K}^{\text{con}}$, we get
 \begin{equation}\label{eqn_defn_ofK_1}
 \begin{aligned}
		\mathcal{K}^{\text{CW}}(\mathcal{R}_s x, \mathcal{R}_{-t} \mathcal{R}_s x , t) = \eta^\circ(a_t) \Delta(t)  \int_{\R}\sigma(x a_s,\lambda ) \phi_{\tau,\lambda}^{n,n}(a_t) |c^{n,n}_{\tau}(\lambda)|^{-2} d\lambda.
\end{aligned}
\end{equation}

 It is not immediately evident from the given hypotheses that $\mathcal{K}^{\text{CW}}$ exists other than in a distributional sense. Throughout the rest of the paper, we will consistently assume that the symbol satisfying estimates as in Theorem \ref{thm:pdo_for_n,n_type} are, in fact, rapidly decreasing, although our estimates will not depend on the rate of decrease. Explicitly, we   assume $ \sigma(\cdot, \lambda) $ is multiplied with a factor of the form $ e^{-\epsilon \lambda^2} ,$ where $ 0<\epsilon \leq 1 $.  This assumption will enable us to define pointwise functions and perform various formal manipulations, such as integration by parts. Our estimates are uniform in $ \epsilon $. Once we prove suitable uniform estimates, standard limiting arguments allow us to pass to the general case. 
 
We will prove \eqref{eqn:Coifman-Weiss_condn_satisified_by_K} in the following steps:
	
 \textbf{Step 1: Expansion of the Kernel $ \mathcal{K}^{\text{CW}} $.}  We will utilize the asymptotic expansion of $ \phi^{n,n}_{\tau,\lambda} $ near the identity to express $ \mathcal{K}^{\text{CW}} $   in terms of simpler or well-behaved functions.

	Let  $ \varPhi $ is a smooth even function on $ \R $, with $ 0\leq \varPhi\leq 1 $; $ \varPhi=1 $ when $ |\lambda|>2$; $ \varPhi(\lambda)=0 $ when $ |\lambda|<1 $. By extending the approach from \cite[Proposition 4.1]{Stanton_and_Tomas} and applying Lemma \ref{thm:localexpansion-phi}, we can derive an equivalent result to \cite[Proposition 4.3]{PR_PDO_22} within our framework. This will enable us to decompose $\mathcal{K}^{\text{CW}}$ as follows:
	\begin{equation}\label{eqn:separation_of_kernel}
		\mathcal{K}^{\text{CW}}(\mathcal{R}_s x, \mathcal{R}_{-t} \mathcal{R}_s x , t)    =   \mathcal{K}_0(xa_s,t) + \mathcal{K}_{\text{err}}(xa_s,t),
	\end{equation} 
	satisfying the following:
	\begin{enumerate}\item[(i)]There is some constant $ C>0 $ and $\kappa \in L^1(\R) $, such that
		\begin{equation}\label{eqn:condn_zeta_local_part}
			| \mathcal{K}_{\text{err}}(x, t)| \leq C\,\left(\sup\limits_{\overline{v}\in \overline{N}}   \| \sigma_{\overline{v}}\|_{\mathcal{H}{(\R,0,2)}}\right)  \mathcal{\kappa}(t), \text{ for all   $ x \in G$} ,
		\end{equation} \text{and}
		\item [(ii)] $ 
			\mathcal{K}_0(xa_s,t)  =  C \eta^\circ(a_t )\,  \Delta(t) \, \left( \frac{t}{\Delta(t)}\right)^{1/2} \int_{0}^\infty \varPhi(\lambda)  \sigma(xa_s, \lambda ) \mathcal{J}_{0} (\lambda t) |c^{n,n}_{\tau}(\lambda)|^{-2} d\lambda,
		 $
	\end{enumerate}
 where we recall $ \mathcal{J}_{0}$ is the  generalized Bessel function defined as in \eqref{defn_J_mu}.

	 \textbf{Step 2: Connection with Euclidean pseudo-differential operator.} 
  From \eqref{eqn:condn_zeta_local_part}, we notice that we can bound $|\mathcal{K}_{\text{err}}(xa_s, \cdot)|$ with an integrable function $\kappa$ in $\R$ that remains independent of `$x$' and `$s$'. Consequently, it's evident that $\mathcal{K}_{\text{err}}$ satisfies \eqref{eqn:Coifman-Weiss_condn_satisified_by_K}. To establish the assumptions of Corollary \ref{cor_trans_cw}, our task is to prove that $\mathcal{K}_0$ also fulfills \eqref{eqn:Coifman-Weiss_condn_satisified_by_K}. To achieve this, we will demonstrate that the functions $\{\mathcal{K}_0(xa_s, t): x\in G \}$ can act as kernels for Euclidean pseudo-differential operators corresponding to a family of symbols $\{ a_x(s, \xi): x\in G \} $. These symbols will satisfy the conditions outlined in Theorem \ref{thm:pdo_for_family_of_symbols}. By doing so, we can employ Theorem \ref{thm:pdo_for_family_of_symbols} to conclude that $\mathcal{K}_0$ satisfies \eqref{eqn:Coifman-Weiss_condn_satisified_by_K}. In summary, our objective is to prove the following:
	\begin{equation}\label{eqn:condn_on_K_0}
		\left| (1+|\xi|)^{\alpha} \partial_s^\beta \partial_\xi^\alpha \int_{-\infty}^{\infty} e^{-2\pi i\xi t } \mathcal{K}_0(xa_s,t)dt \right| \leq C \left( \sup\limits_{x\in G}   \| \sigma_{{x}}\|_{\mathcal{H}{(\R,2,2)}} \right)    ,
	\end{equation}
	for all $x\in G$, $s,\xi \in \R$, and $\alpha, \beta \in \{0,1,2\}$, where the constant $ C $  is independent of $ x\in G $. Now for a given $x_0\in G$,   utilizing the Iwasawa decomposition and the fact that the abelian group $ A $ acts as a dilation on $ \ol N $, we can find  $k_0\in K, \overline{v}\in \overline{N}$, and $ s_0 \in \R$ such that $x_0 = k_0\overline{v} a_{s_0} $. Since by hypothesis of Theorem \ref{thm:Lp_boundedness_of_T_sigma_local}, $x\mapsto \sigma(x,\lambda)$ is a $K$-biinvariant function on $G$, we  observe that it suffices to prove \eqref{eqn:condn_on_K_0} for all $x=\overline{v}\in \overline{N}$.
	
 \textbf{Step 3: Estimate of the Kernel $\mathcal{K}_0$.}  We recall the definition of $ \mathcal{K}_0 $ from \eqref{eqn:separation_of_kernel}:
	\begin{equation}
		\mathcal{K}_0(\overline{v}a_s,t)  =  C \eta^\circ(a_t )\, \Delta(t) \,  \left( \frac{t}{\Delta(t)}\right)^{1/2} \int_{0}^\infty \varPhi(\lambda)  \sigma(\overline{v}a_s,\lambda ) \mathcal{J}_{0} (\lambda t) |c^{n,n}_{\tau}(\lambda)|^{-2} \,d\lambda,
	\end{equation}
  where we have the following formula from \cite[Eq. 10.9.12, p.224]{NIST_2010} 
  \begin{align*}
		\mathcal{J}_0 (\lambda t) =\frac{2}{\pi} \int_{\lambda}^{\infty} (\xi^2-\lambda^2)^{-1/2} \sin \xi t \, d\xi.
	\end{align*}
	Let us define $$ \Sigma(x,\lambda) =  (\partial_\lambda \cdot (1/\lambda)) ^{(d-2)/2} \left(  \varPhi(\lambda) \sigma(x,\lambda) |c^{n,n}_{\tau}(\lambda)|^{-2}  \right),$$
 for all $x \in G$ and $\lambda \in \R$. Then using integration by parts we have,
	\begin{align*}
		\mathcal{K}_0(\overline{v} a_s, t) &=C \eta^\circ(a_t) \left[ \frac{\Delta(t)}{t} \right]^{\frac{1}{2}}\, t \int_{\R}  \Sigma(\overline{v}a_s,\lambda)  \mathcal{J}_0 (\lambda t) d\lambda \\
		& = C \eta^\circ(a_t) \left[ \frac{\Delta(t)}{t} \right]^{\frac{1}{2}}\, t \int_{\R}   \sin \xi t \int_0^\xi   \Sigma(\overline{v}a_s,\lambda)  (\xi^2-\lambda^2)^{-1/2}\, d\lambda \,d\xi \\
		& =  C \eta^\circ(a_t) \left[ \frac{\Delta(t)}{t} \right]^{\frac{1}{2}}\,  \int_{\R}   \cos \xi t  \frac{d}{d \xi }\int_0^\xi   \Sigma(\overline{v}a_s,\lambda)  (\xi^2-\lambda^2)^{-1/2} \,d\lambda\, d\xi.
	\end{align*}
	
	\noindent Let us define
	\begin{equation}\label{eqn:defn_of_g(z,mu)}
		\begin{aligned}
			g(x,\xi) &:= \int_{0}^\xi (\xi^2-\lambda^2)^{-1/2}   \Sigma(x,\lambda) \,d\lambda,\quad x\in G,  \xi\geq 0\\
			&= \frac{1}{2}   \int_{-1}^{1}  (1-\lambda^2)^{-1/2}  \Sigma(x,\lambda \xi)\, d\lambda  \quad x\in G,  \quad \text{when }\xi>0,
		\end{aligned}
	\end{equation}
	and 
	\begin{eqnarray*}
		h(x,\xi) = \left( \frac{\partial}{\partial\xi} g \right) (x,|\xi|) , \quad \xi \in \R.
	\end{eqnarray*}
	We obtain
	\begin{align*}
		\mathcal{K}_0(\overline{v}a_s,t)&=  -C \eta^\circ(a_t) \left[ \frac{\Delta(t)}{t} \right]^{1/2}\,  \int_0^\infty     \frac{d}{d\xi} \, g (\overline{v}a_s,\xi) \cos \xi t  \, d\xi\\
		&=   -C \eta^\circ(a_t)\left[ \frac{\Delta(t)}{t} \right]^{1/2}\,  \int_{-\infty}^\infty   h(\overline{v}a_s,\xi) e^{i\xi t} d\xi.
	\end{align*}
	Let $ \nu(t) =-{C \eta^\circ(a_t)} \left[ \frac{\Delta(t)}{t} \right]^{1/2}, \,  t \in \R$. Then \be\label{eqn:ftK0}
	\mathcal{F}(\mathcal{K}_0(\overline{v}a_s,\cdot) )(\xi) =\left( \mathcal{F}(\nu) \ast_{\R} h(\overline{v}a_s,\cdot)\right)(\xi).
	\ee
	First, we claim that  \begin{equation*}\label{eqn:est_required_for_h}
		\sup_{\xi \in \R} \left( \left| \partial_s^\beta h(\overline{v}a_s, \xi)\right| +\left|  (1+\xi)\partial_s^\beta \partial_\xi h(\overline{v}a_s, \xi)\right|\right) \leq C \left( \sup\limits_{\overline{v}\in \overline{N}}   \| \sigma_{\overline{v}}\|_{\mathcal{H}{(\R,2,2)}} \right) \quad \text{ for all } \beta \in \{0,1,2\},
	\end{equation*}
	where the constant $ C$ is independent of $\overline{v} \in \overline{N} $.  From  \eqref{eqn:defn_of_g(z,mu)} we have
	\begin{align*}
		h(\overline{v}a_s, \xi) = \frac{1}{2}   \int_{-1}^{1}  (1-\lambda^2)^{-1/2} \lambda \frac{\partial}{\partial (\lambda \xi)} \Sigma(\overline{v},\lambda \xi)\, d\lambda \quad \text{ for $ \xi >0 $}.
	\end{align*} 
	Then using  Lemma \ref{lemma:est_of_|c(-lambda)|^-2} and the hypothesis in \eqref{hyp_sigma_local} we get
	\begin{equation*}
		\left| \frac{\partial^\beta}{\partial s^\beta}\frac{\partial^\alpha}{\partial \lambda^\alpha}  \Sigma(\overline{v}a_s,\lambda)\right|  \leq C  \left( \sup\limits_{\overline{v}\in \overline{N}}   \| \sigma_{\overline{v}}\|_{\mathcal{H}{(\R,2,2)}} \right) (1+|\lambda|)^{1-\alpha} \quad  \text{for all } \alpha,\beta\in \{0,1,2\}.
	\end{equation*}
	Hence \begin{align*}
		|\frac{\partial^\beta}{\partial s^\beta} h(\overline{v}a_s,\xi)| &\leq C \left( \sup\limits_{\overline{v}\in \overline{N}}   \| \sigma_{\overline{v}}\|_{\mathcal{H}{(\R,2,2)}} \right) \int_{0}^1 (1-\lambda^2)^{-1/2} \lambda \,d\lambda\\
		& \leq C\left( \sup\limits_{\overline{v}\in \overline{N}}   \| \sigma_{\overline{v}}\|_{\mathcal{H}{(\R,2,2)}} \right).
	\end{align*}
	Also,  \begin{align*}
		|(1+\xi)\frac{\partial^\beta}{\partial s^\beta}  \frac{\partial}{\partial \xi}  h(\overline{v}a_s,\xi)| &\leq C \left( \sup\limits_{\overline{v}\in \overline{N}}   \| \sigma_{\overline{v}}\|_{\mathcal{H}{(\R,2,2)}} \right)  (1+|\xi|) \int_{0}^1 (1-\lambda^2)^{-1/2} \frac{\lambda^2}{(1+|\lambda \xi| )}  d\lambda \\
		&\leq C \left( \sup\limits_{\overline{v}\in \overline{N}}   \| \sigma_{\overline{v}}\|_{\mathcal{H}{(\R,2,2)}} \right) \left( \int_{0}^1 (1-\lambda^2)^{-1/2}  {\lambda^2}  d\lambda \right. \\& \hspace{6cm}\left. +  |\xi| \int_{0}^1 (1-\lambda^2)^{-1/2} \frac{\lambda}{| \xi| }  d\lambda \right) \\
		&\leq C\left( \sup\limits_{\overline{v}\in \overline{N}}   \| \sigma_{\overline{v}}\|_{\mathcal{H}{(\R,2,2)}} \right).
	\end{align*}
	Similarly, we can prove that 
	\bes
	\left|  (1+\xi)^\alpha \partial_s^\beta \partial_\xi^\alpha h(\overline{v}a_s, \xi)\right|\leq C \left( \sup\limits_{\overline{v}\in \overline{N}}   \| \sigma_{\overline{v}}\|_{\mathcal{H}{(\R,2,2)}} \right) \, \quad  \text{for all } \alpha, \beta \in \{0,1,2\}.
	\ees
	  By applying the inequality above and taking into account that $\mathcal{F}(\nu)$ is a Schwartz function, we can deduce from (\ref{eqn:ftK0}) that $\mathcal{K}_0$ satisfies (\ref{eqn:condn_on_K_0}). Consequently,  This also concludes the proof of our Theorem \ref{thm:Lp_boundedness_of_T_sigma_local}.
\section{Large-scale analysis of  $\Psi_{\sigma}$}\label{sec_glo_par}
In the previous section, we observed how the kernel $ \mathcal{K}^{\text{loc}} $ exhibits similar behavior to the kernel of an Euclidean pseudo-differential operator. By using the transference method, we successfully obtained a bound for the $L^p$ operator norm of the local part of $\Psi_{\sigma}$. However, the analysis of the global part undergoes significant changes due to the exponential volume growth of the group $G$ and the entirely different local and global behavior of $\phi_{\tau,\lambda}^{n,n}$.

In multiplier theory, the authors in \cite{Ricci} utilized the Herz majorizing principle theorem for the convolution operator. For cases involving noncompact type symmetric spaces, Ionescu \cite{Ionescu_2002} further employed the same principle to estimate the $ L^p $ norm of the multiplier operator.  In this section, we will establish the $L^p$ bound of $\Psi_{\sigma}^{\text{glo}}$. To achieve this, we will use the expansion of spherical functions $\phi_{\tau,\lambda}^{n,n}$ away from the identity, and we will see that the global analysis of $\Psi_{\sigma}$ has no Euclidean analogue. Our approach in this section follows the general outline of \cite{Ionescu_2000}. Finally, Theorem \ref{thm:pdo_for_n,n_type} will be completed as a consequence of the following theorem. Before stating the main result in this section, let us recall the definition \eqref{eqn:defn_of_T_sigma_global} of $\Psi^{\text{glo}}_{\sigma}$,
 \begin{equation}\label{defn_Psi_sigma_rep}
	  \Psi^{ \text{glo}}_{\sigma}f(x) =\int_{G} f(y) \mathcal{K}^{\text{glo}}(x,y^{-1}x) \, dy , \quad x\in G,
\end{equation}
where 
\begin{equation}\label{Kglo_re}
     \mathcal{K}^{\text{glo}}(x,y^{-1}x) = \eta (y^{-1}x) \int_{\R} \sigma (x,\lambda) \phi_{\tau,\lambda}^{n,n}(y^{-1}x ) |c^{n,n}_{\tau}(\lambda)|^{-2} d\lambda.
\end{equation}
We will complete the proof of Theorem  \ref{thm:pdo_for_n,n_type}  by establishing the following theorem.
 \begin{theorem}\label{thm:Lp_boundedness_of_T_sigma_global}
	Let $p\in [1,2)\cup (2,\infty)$. Suppose that $\sigma:G\times S_{p}^{\circ}\rightarrow\mathbb{C}$ is a function satisfying the following properties:
\begin{enumerate}
\item[(i)] For each $\lambda \in S_p^{\circ}$, $x\mapsto\sigma(x,\lambda)$ is a $K$-biinvariant function on $G$.
\item[(ii)] For each $x$ in $G$,  $ \lambda \mapsto \sigma(x,\lambda)$ is an even holomorphic function on $ S_p^\circ$ and 
\begin{equation} \label{sigma_glo_con}
     \sup_{x\in G}\| \sigma(x,\cdot)\|_{\mathcal{MH}(S_p,2)}<\infty.
\end{equation}
\item[(iii)] Additionally, for $p=1$ and $n\in \Z^{\tau_+}\setminus\{ 0\}$,  $\frac{\partial^\alpha}{\partial \lambda^{\alpha}}\sigma(x,\lambda)|_{\lambda=i}=0,$  for all $\alpha\in \{0,1,2\}$.
\end{enumerate}
Then the operator $\Psi^{ \text{glo}}_{\sigma}$ defined as in \eqref{defn_Psi_sigma_rep}, is a bounded operator from $L^{p}(G)_{n,n}$ to itself. Moreover, there exists a constant $C_{p,n}>0$ such that
$$\| \Psi^{ \text{glo}}_{\sigma}f\|_{L^{p}(G)}\leq C_{p,n} \left( \sup_{x\in G}\| \sigma(x,\cdot)\|_{\mathcal{MH}(S_p,2)} \right) \|f\|_{L^p(G)}, $$
\end{theorem}
for all $ f\in L^{p}(G)_{n,n} $. 
\begin{remark}\label{rem_glo}
\begin{enumerate}
\item  We would like to mention that for $p=1$ case when $n=0$ or $n \in \Z^{\tau_-}$, we obtain an improvement on the $L^1$ case. More precisely, for $n=0$ or $n \in \Z^{\tau_-}$, it suffices to assume conditions (i) and (ii) of the theorem above to establish the $L^1$-boundedness of $\Psi_\sigma^{\text{glo}}$.  
 
 \item  When examining the proof of the boundedness of $\Psi_\sigma^{\text{glo}}$ and $\Psi_\sigma^{\text{dis}}$, it becomes apparent that these operators do not rely on any regularity conditions concerning the space variable of the symbol $\sigma$. The requirement for a derivative condition on the space variable is solely necessary to establish the boundedness of the local part $\Psi_\sigma^{\text{loc}}$.
In fact, to address the boundedness of $\Psi_\sigma^{\text{loc}}$, we needed $\sup\limits_{\overline{v}\in \overline{N}} \| \sigma_{\overline{v}}\|_{\mathcal{H}{(\R,2,2)}} <\infty$ (see \eqref{result_loc}). Hence, the regularity condition on the space variable of the symbol $\sigma(x,\lambda)$ is essential, but only when $\lambda$ is along the real line, not for the whole strip $S_p$.
 \end{enumerate}
\end{remark}

\subsection{{Estimate of the global kernel $\mathcal{K}^{\text{glo}}$.}}
As in the Euclidean setting, the estimate of $\mathcal{K}^{\text{glo}}$ plays a crucial role in establishing the boundedness of the pseudo-differential operator. However, unlike in Euclidean space, we will observe the exponential decay instead of polynomial decay in our estimates. This exponential decay is essential to handle the exponential volume growth of the group $\mathrm{SL}(2,\R)$. In the following lemma, we will explore how the holomorphicity of the symbol is responsible for the decay property mentioned above.

\begin{lemma}  Suppose that $ \sigma $ satisfies the hypothesis  of Theorem \ref{thm:Lp_boundedness_of_T_sigma_global} and   let $\mathcal{K}^{\text{glo}}$ be as in \eqref{Kglo_re}. Then the kernel  $\mathcal{K}^{\text{glo}}$ satisfies the following estimates
   \begin{enumerate} 
   \item  For $1\leq p<2 $,  one has
          \begin{equation}\label{est_of_K_glo}
       |\mathcal{K}^{\text{glo}}(x,y^{-1}x)|\leq C_{p,n} \left( \sup_{x\in G}\| \sigma(x,\cdot)\|_{\mathcal{MH}(S_p,2)} \right)
              \frac{\eta (y^{-1}x)}{ {(1+[y^{-1}x]^+)}^{2}} e^{-\frac{2}{p}( \,[y^{-1}x]^+)}.
    \end{equation}
    \item For $2<p<\infty$, one has
    \begin{equation}\label{est_of_K_glo_p>2}
       |\mathcal{K}^{\text{glo}}(x,y^{-1}x)|\leq C_{p,n} \left( \sup_{x\in G}\| \sigma(x,\cdot)\|_{\mathcal{MH}(S_p,2)} \right)   \, \eta (y^{-1}x) \, \frac{e^{-\frac{2}{p'}( \,[y^{-1}x]^+)}}{ {(1+[y^{-1}x]^+)}^{2}},
    \end{equation}
    \end{enumerate}
    where  $C_{p,n}>0$ is a constant depending only on the symbol and $n\in \Z^{\sigma}$. 
\end{lemma}
\begin{proof}
To begin, we will prove the lemma for $1<p<2$. We observe from    \eqref{Kglo_re} that   the function $y\mapsto \mathcal{K}^{\text{glo}}(\cdot,y)$ is supported away from identity, as $\eta$ is a $K$-biinvariant function supported on the set $$\{ k_1 a_t k_2 : t\geq 1/2, k_1,k_2 \in K\}.$$  
This allows us to apply the Harish-Chandra series expansion of the spherical function $\phi^{n,n}_{\tau,\lambda}$ from Lemma \ref{propn:global_expansion_of_phi_lambda}. Substituting the explicit expression of $\phi^{n,n}_{\tau,\lambda}(y^{-1}x)$ from \eqref{eqn:phi-lambda-assymp-expansion} and using the fact that $ |c^{n,n}_{\tau}(\lambda)|^{2}= c^{n,n}_{\tau}(\lambda) c^{n,n}_{\tau}(-\lambda)$ for all $\lambda \in \R$, we obtain
\begin{align*}
     \mathcal{K}^{\text{glo}}(x,y^{-1}x) =&  \eta(y^{-1}x) e^{- \,[y^{-1}x]^+} \left(\int_{\R } \sigma (x,\lambda) (1+a(\lambda, [y^{-1}x]^+) ) c^{n,n}_{\tau}(-\lambda)^{-1}  e^{i\lambda([y^{-1}x]^+) } d\lambda \right.\\
     & \hspace{1cm}\left. + \int_{\R} \sigma (x,\lambda) (1+a(-\lambda, [y^{-1}x]^+)) c^{n,n}_{\tau}(\lambda)^{-1}  e^{i\lambda([y^{-1}x]^+) } d\lambda \right).    
\end{align*}
After the change of variable $\lambda\mapsto -\lambda$ in the second integral and using the fact that $\lambda \mapsto \sigma(x,\lambda) $ is an even function, we get
\begin{align*}
    \mathcal{K}^{\text{glo}}(x,y^{-1}x) = \, & C \eta(y^{-1}x) e^{- \,[y^{-1}x]^+} \int_{\R} \sigma (x,\lambda)  c^{n,n}_{\tau}(-\lambda)^{-1}  e^{i\lambda([y^{-1}x]^+) } d\lambda \\
    &+ C \eta(y^{-1}x) e^{- \,[y^{-1}x]^+} \int_{\R} \sigma (x,\lambda) a(\lambda, [y^{-1}x]^+)  c^{n,n}_{\tau}(-\lambda)^{-1}  e^{i\lambda([y^{-1}x]^+) } d\lambda \\
    =& \mathcal{K}^{\text{glo}}_0(x,y^{-1}x)+ \mathcal{K}^{\text{glo}}_{\text{error}}(x,y^{-1}x).
\end{align*}
First, we will demonstrate that  $\mathcal{K}^{\text{glo}}_0(x,y^{-1}x)$ satisfies the estimate \eqref{est_of_K_glo}. We observe that the above integrand is a holomorphic function on $S^{\circ}_{p}$. Therefore, by applying Cauchy's integral theorem, we move the integration with respect to $\lambda$ from $\R$ to $\R -i(\gamma_p- \gamma_p/2[y^{-1}x]^+) $, obtaining
\begin{align}\label{K_glo_0}
    \mathcal{K}^{\text{glo}}_0(x,y^{-1}x)&= C \, \eta(y^{-1}x) e^{- 2/p( \,[y^{-1}x]^+)}  \int_{\R} {\nu (x,\lambda+i(\gamma_p- \gamma_p/[y^{-1}x]^+) )} e^{i\lambda([y^{-1}x]^+) } d\lambda, 
\end{align}
where $\nu (x,\lambda) $ is defined as follows \[ \nu (x,\lambda) = {\sigma (x,\lambda) )}{c^{n,n}_{\tau}  (-\lambda)^{-1}}  .\]
 Consequently, integrating by parts the inner integral of \eqref{K_glo_0}, we deduce that
\begin{align*}
    \mathcal{K}^{\text{glo}}_0(x,y^{-1}x)& =C \, \eta(y^{-1}x)\frac{e^{-2/p( \,[y^{-1}x]^+)}}{([y^{-1}x]^+)^2}  \int_{\R}  \frac{\partial^2}{\partial \lambda^2}{\nu (x,\lambda+i(\gamma_p- \gamma_p/[y^{-1}x]^+) )} e^{i\lambda([y^{-1}x]^+) } d\lambda.
\end{align*}
Taking modulus on both sides, using hypothesis \eqref{sigma_glo_con}, and the estimate \eqref{est: cminusla} of $c^{n,n}_{\tau}(-\lambda)^{-1}$ we get,
\begin{align*}
    \left| \mathcal{K}^{\text{glo}}_0(x,y^{-1}x)\right|& =C_{p,n} \left( \sup_{x\in G}\| \sigma(x,\cdot)\|_{\mathcal{MH}(S_p,2)} \right)   \, \eta(y^{-1}x)\frac{e^{-2/p( \,[y^{-1}x]^+)}}{([y^{-1}x]^+)^2}  \int_{\R} \frac{d\lambda}{(1+|\lambda|)^{3/2}}\\
    &\leq C_{p,n} \left( \sup_{x\in G}\| \sigma(x,\cdot)\|_{\mathcal{MH}(S_p,2)} \right)   \, \eta(y^{-1}x)\frac{e^{-2/p( \,[y^{-1}x]^+)}}{([y^{-1}x]^+)^2}.
\end{align*}
Next, we shall estimate $\mathcal{K}^{\text{glo}}_{\text{error}}(x,y)$.  Proceeding in a similar way as in the proof of $ \mathcal{K}^{\text{glo}}_0$, we get
\begin{multline*}
    \mathcal{K}^{\text{glo}}_{\text{error}}(x,y^{-1}x) =C \, \eta(y^{-1}x)\frac{e^{-2/p( \,[y^{-1}x]^+)}}{([y^{-1}x]^+)^2}  \int_{\R}  \frac{\partial^2}{\partial \lambda^2}\left({\nu  (x,\lambda+i(\gamma_p- \gamma_p/[y^{-1}x]^+) )} \right.\\
     \left. \cdot\, a(\lambda+i(\gamma_p- \gamma_p/[y^{-1}x]^+), [y^{-1}x]^+) \right)e^{i\lambda([y^{-1}x]^+) } d\lambda.
\end{multline*}
We recall from \eqref{eqn:est_for_a(lambda,t)} that the function $ \lambda \mapsto a(\lambda,t) $ satisfies a favorable symbol-type estimate for $t\geq 1/10$. Moreover, by the Leibniz rule, one has  
\begin{equation}\label{eqn:est_of_sigma*a*c^{-1}}
		 \left| \frac{\partial^2}{\partial \lambda^2} \left( \frac{\sigma(x,\lambda)\, a(\lambda, t)}{c^{n,n}_{\tau}(-\lambda)}\right)\right| \leq \frac{C_n}{(1+|\lambda|)^{3/2}  }
	\end{equation}
for all $ \lambda \in \C$  with $ 0 \leq \Im \lambda \leq \gamma_p$, $ t\geq 1/10 $ and $ x\in G  $. By  \eqref{eqn:est_of_sigma*a*c^{-1}} in the expression of $\mathcal{K}^{\text{glo}}_{\text{error}}$, we obtain the required estimate of $\mathcal{K}^{\text{glo}}_{\text{error}}(x,y^{-1}x)$. This settles the lemma for $1<p<2$. 

Now we handle the $p=1$ case. We recall from Lemma \ref{c^n,n-}, for $n\in \Z^{\tau_+} \setminus \{0\}$, the function $\lambda \rightarrow c^{n,n}_{\tau}(-\lambda)^{-1} $ has a simple pole at $\lambda =i$ within the boundary of $ S_1$. This prevents us from directly using the estimate \eqref{est: cminusla} of $c^{n,n}_{\tau}(-\lambda)^{-1}$ near $\lambda=i$.  To overcome this obstacle, we will utilize the hypothesis on the symbol $\sigma(x,\lambda)$, which has a zero at $\lambda =i$. Exploiting this and the Leibniz rule, we  obtain  that
\begin{equation}\label{est_v_cnn_1}
		 \left| \frac{\partial^2}{\partial \lambda^2} \left( \frac{\sigma(x,\lambda)\, }{c^{n,n}_{\tau}(-\lambda)}\right)\right| \leq \frac{C_n}{(1+|\lambda|)^{3/2}  }.
	\end{equation}
With this inequality in hand and repeating the previous argument, we can complete the proof of \eqref{est_of_K_glo}. When $n\in \Z^{\tau_-}$, $\lambda \rightarrow c^{n,n}_{\tau}(-\lambda)^{-1} $ does not have any pole in $S_1$, so the proof follows similarly as in $1<p<2$ case, this also concludes the proof for $p=1$.

  Next, we address the case $2<p<\infty$, which follows a similar approach. We note that $\gamma_p=|(2/p-1)|=(2/p'-1)$ for all $p>2$, where  $1/{p'} =1-1/p$. Applying the same argument as before, we obtain \eqref{K_glo_0} with $e^{-2/p([y^{-1}x]^+)}$ replaced by $e^{-2/p'([y^{-1}x]^+)}$. Then, proceeding analogously to the previous case, we can establish \eqref{est_of_K_glo_p>2}.
\end{proof}
We now employ the lemma above to establish the $L^p$-boundedness of  $\Psi^{\text{glo}}_{\sigma}$, and, consequently, complete the proof of Theorem \ref{thm:pdo_for_n,n_type}.  Before that, we express $\Psi^{\text{glo}}_{\sigma}$ as a sum of two integral operators on the $\overline{N} A$ group. We then proceed to find their $L^p$ norm estimates.

Let $f$ and $\varphi$ be two smooth compactly supported $(n,n)$ type functions on the group $G$. Using the Iwasawa integration formula \eqref{eqn:integral-decom-nbar-a}, we can write from \eqref{defn_Psi_sigma_rep} the expression for $\Psi^{\text{glo}}_{\sigma}$  with $y =\overline{m}a_s k_1$ and $x =\overline{n}a_t k_2$ as follows:

\begin{equation}
\begin{aligned} 
\left \langle  \Psi^{ \text{glo}}_{\sigma}f ,{ \varphi}\right\rangle &= \int_G \Psi^{ \text{glo}}_{\sigma}f(x)  \overline{\varphi(x)} dx \\   &=\int_{\ol N} \int_{\ol N} \int_{\R }\int_{ \R} f(\overline{m}a_s)\ol {\varphi(\overline{n}a_t)} \mathcal{K}^{\text{glo}}(\overline{n}a_t,a_{-s}\overline{m}^{-1}\overline{n} a_{t}) e^{2(s+t)} ds\, d\ol m\, dt\,  d\ol{n}\\
&=   \int_{\ol N} \int_{\ol N} \int_{\R }\int_{ \R} f(\overline{m}a_s) \ol {\varphi(\overline{n}a_t) }\mathcal{K}^{\text{glo}}(\overline{n}a_t,\delta_{-s} (\overline{m}^{-1}\overline{n}) a_{t-s}) e^{2(s+t)} ds\, d\ol m\, dt\,  d\ol{n},
\end{aligned}
\end{equation}
where we recall that  $ \delta_{-s} $ is the dilation on the group $\ol{N}$ defined in \eqref{dil_delta}.

We observe that the integrations over the $k_2$ variable are canceled due to the right $K$-invariance of the symbol $\sigma$ with respect to the space variable. If we assume that $x\mapsto \sigma(x,\lambda)$ is any other right $l$ ($\not =0$)-type, then the above integral will be always zero.

Now, let $\chi^+$ and $\chi^-$ be the characteristic functions of the intervals $[0,\infty)$ and $(-\infty, 0)$, respectively. We can now express the above operator as follows:
\begin{align*} 
	 \left \langle  \Psi^{ \text{glo}}_{\sigma}f ,{ \varphi }\right\rangle &=  \left \langle \mathcal{I}^{-}_{\sigma} f ,{ \varphi }\right\rangle +  \left \langle  \mathcal{I}^{+}_{\sigma} f ,{ \varphi }\right\rangle   \quad \text{(say)},
\end{align*}
where  $ \mathcal{I}^{\pm}_{\sigma}f$ are the operators defined on $\ol {N}A $ group by the following formulae
\begin{equation}\label{Ipm_sigma}
       \mathcal{I}^{\pm}_{\sigma}f(\ol n a_t)   := \int_{\overline{N}} \int_{\mathbb{R}}  f(\overline{m}a_s)  \mathcal{K}^{\text{glo}}(\overline{n}a_t,\delta_{-s} (\overline{m}^{-1}\overline{n}) a_{t-s}) \chi^{\pm}(t-s) e^{2s}\, ds\, d{\overline{m}},
\end{equation}
and 
\begin{equation}\label{Ipm_sigma_prod}
       \left \langle \mathcal{I}^{\pm}_{\sigma} f ,{ \varphi}\right\rangle   = \int_{\overline{N}} \int_{\mathbb{R}}  \int_{\overline{N}} \int_{\mathbb{R}}  f(\overline{m}a_s) \ol {\varphi(\overline{n}a_t)}   \mathcal{K}^{\text{glo}}(\overline{n}a_t,\delta_{-s} (\overline{m}^{-1}\overline{n}) a_{t-s}) \chi^{\pm}(t-s) e^{2(s+t)}\, ds\, d{\overline{m}}\, dt\,  d\ol{n}.
\end{equation}
Consequently, the $L^{p}(G)_{n,n}$-boundedness of $\Psi^{ \text{glo}}_{\sigma}$ follows from that of the operators $ \mathcal{I}^{\pm}_{\sigma}$ on $L^p(\ol{N}{A})$, which we shall take up separately in the rest of this section.
 
\subsection{$ L^p$ operator norm estimates of  $ \mathcal{I}^{\pm}_{\sigma}$}
 We first shall establish the $L^p$-boundedness of the operator $\mathcal{I}^{-}_{\sigma}$  by proving the following lemma:
\begin{lemma}\label{lem_est_of_I-}
		Let $p\in [1,2) \cup (2,\infty)$. Suppose that $ \sigma $ satisfies the hypothesis  of Theorem \ref{thm:Lp_boundedness_of_T_sigma_global} and $\mathcal{I}^{-}_{\sigma}$ be as in \eqref{Ipm_sigma}.  Then we have $\mathcal{I}^{-}_{\sigma}$   is bounded from $L^{p}(\ol{N}A)$ to itself. Moreover, there exists a constant $ C_{p,n}>0 $  such that  \begin{align}\label{lem_Est_I_}
			 \|\mathcal{I}^{-}_{\sigma} f\|_{L^p (\ol {N}A)}&  \leq C_{p,n} \left( \sup_{x\in G}\| \sigma(x,\cdot)\|_{\mathcal{MH}(S_p,2)} \right)  \|f\|_{L^p(\ol {N}A)}
\end{align}
  for all $f \in L^p(\ol {N}A)$.
 \end{lemma}

 \begin{proof}
     To get the desired result, it suffices to prove that for any smooth compactly supported functions $ f,\varphi: \ol {N} A\rightarrow \C$, one has
     \begin{align}\label{If,phi_est}
         \left| \left\langle\mathcal{I}^{-}_{\sigma} f, \varphi \right\rangle\right|	 \leq C_{p,n} \left( \sup_{x\in G}\| \sigma(x,\cdot)\|_{\mathcal{MH}(S_p,2)} \right)  \| f\|_{L^p(\ol N A )} \,  \| { \varphi} \|_{L^{p'}(\ol N A)}.
     \end{align}
We recall the expression of $\left\langle\mathcal{I}^{-}_{\sigma} f, h \right\rangle$ from \eqref{Ipm_sigma_prod}
     \begin{align*}
       \left \langle \mathcal{I}^{-}_{\sigma} f ,{ \varphi }\right\rangle   = \int_{\overline{N}} \int_{\mathbb{R}}   \int_{\mathbb{R}}  \int_{\overline{N}}f(\overline{m}a_s)  \ol {\varphi(\overline{n}a_t)} \mathcal{K}^{\text{glo}}(\overline{n}a_t,\delta_{-s} (\overline{m}^{-1}\overline{n}) a_{t-s}) \chi^{-}(t-s) e^{2(s+t)}\, ds\, d{\overline{m}}\, dt\,  d\ol{n},
\end{align*}
which after a change of variable $\ol {n}\mapsto \ol{m}^{-1} \ol{n}$ yields
\begin{align*}
       \left \langle \mathcal{I}^{-}_{\sigma} f ,{ \varphi }\right\rangle   = \int_{\overline{N}} \int_{\mathbb{R}} \int_{\mathbb{R}}   \int_{\overline{N}}  f(\overline{m}a_s) \ol {\varphi ( \ol{m}\, \overline{n}a_t)} \mathcal{K}^{\text{glo}}(\ol{m} \,\overline{n}a_t,\delta_{-s} (\overline{n}) a_{t-s}) \chi^{-}(t-s) e^{2(s+t)} d{\overline{n}}\, ds\, dt\,  d\ol{m},
\end{align*}
We recall that the map $\ol {n} \mapsto \delta_{-s}(\ol {n})$ is a dilation of $\ol{N}$. Hence by using \eqref{eqn:dialation_formula_for_N-}, we get
\begin{align*}
       \left \langle \mathcal{I}^{-}_{\sigma} f ,{ \varphi}\right\rangle   = \int_{\overline{N}} \int_{\mathbb{R}}  \int_{\overline{N}} \int_{\mathbb{R}}  f(\overline{m}a_s) \ol {\varphi ( \ol{m} \delta_{s}(\overline{n})a_t)}  \mathcal{K}^{\text{glo}}(\ol{m}\delta_{s} (\overline{n})a_t,\overline{n} a_{t-s}) \chi^{-}(t-s) e^{2t} d{\overline{n}}\, ds\, dt\,  d\ol{m}.
\end{align*}
Next, by using Fubini's theorem and the change of variable $ t \mapsto s-r $ (for fixed $s$), it follows that  
\begin{align}\label{I_exp_final_var}
           \left \langle \mathcal{I}^{-}_{\sigma} f ,{ \varphi}\right\rangle   = \int_{\overline{N}} \int_{\mathbb{R}}  \int_{\overline{N}} \int_{r=0}^{\infty}  f(\overline{m}a_s) \ol {\varphi ( \ol{m} \delta_{s}(\overline{n})a_{s-r})} \mathcal{K}^{\text{glo}}(\ol{m} \delta_{s} (\overline{n})a_{s-r},\overline{n} a_{-r})   e^{2(s-r)} dr\, d{\overline{n}}\, ds\,  d\ol{m}.
\end{align}
To make use of the estimates of $\mathcal{K}^{glo}$, we divide the proof into two parts. 
 First, we will prove the lemma for $1\leq p<2$. Taking modulus on both sides of the above expression and plugging the estimate \eqref{est_of_K_glo} of $\mathcal{K}^{glo}$ for $1\leq p<2$, gives us
 \begin{multline}\label{eqn:equi_expn_of_I_-_2}
      \left| \left \langle \mathcal{I}^{-}_{\sigma} f ,{ \varphi }\right\rangle \right|  \leq  C_{p,n} \left( \sup_{x\in G }\| \sigma(x,\cdot)\|_{\mathcal{MH}(S_p,2)} \right)   \,  \int_{r=0}^{\infty} \int_{\mathbb{R}} \int_{\overline{N}}  \int_{\overline{N}}  |f(\overline{m}a_s)|  | {\varphi ( \ol{m} \delta_{s}(\overline{n})a_{s-r})}|\\
      \cdot \eta (\overline{n} a_{-r}) \, \frac{e^{-\frac{2}{p}( \,[\overline{n} a_{-r}]^+)}}{ {(1+[\overline{n} a_{-r}]^+)}^{2}} e^{2(s-r)} d\ol{m}\, d{\overline{n}}\, ds \,dr.
 \end{multline}
 Now let us define,
	\begin{equation}\label{eqn:def_of_F_and_varPhi}
			\begin{aligned}
				F(s) =& \left[ \int_{\ol N} |f(\ol m a_s)|^p d\ol m \right]^{\frac{1}{p}},\\
				\varPhi(s)= & \left[ \int_{\ol N} |{ \varphi}(\ol m a_s)|^{p'} d\ol m \right]^{\frac{1}{p'}}.
			\end{aligned}
		\end{equation}
Applying H\"older's inequality in \eqref{eqn:equi_expn_of_I_-_2} and using \eqref{eqn:def_of_F_and_varPhi}, we get

 \begin{multline}\label{I-bef_Abel} 
      \left| \left \langle \mathcal{I}^{-}_{\sigma} f ,{ \varphi }\right\rangle \right|  \leq  C_{p,n} \left( \sup_{x\in G }\| \sigma(x,\cdot)\|_{\mathcal{MH}(S_p,2)} \right)   \,  \int_{r=0}^{\infty} \int_{\mathbb{R}}   \int_{\overline{N}}   F(s) \varPhi(s-r)\\
      \cdot \eta (\overline{n} a_{-r}) \, \frac{e^{-\frac{2}{p}( \,[\overline{n} a_{-r}]^+)}}{ {(1+[\overline{n} a_{-r}]^+)}^{2}} e^{2(s-r)}  d{\overline{n}}\, ds\, dr.
 \end{multline}
Since the Abel transform of a $K$-biinvariant function is even, from \eqref{Abel_tran} it follows that  
\begin{equation}\label{app_abel_I-}
\int\limits_{\ol{N}}\eta (\overline{n} a_{-r}) \, \frac{e^{-\frac{2}{p}( \,[\overline{n} a_{-r}]^+)}}{ {(1+[\overline{n} a_{-r}]^+)}^{2}}\,  d\ol{n} = e^{2r}\int\limits_{\ol{N}}\eta (\overline{n} a_{r}) \, \frac{e^{-\frac{2}{p}( \,[\overline{n} a_{r}]^+)}}{ {(1+[\overline{n} a_{r}]^+)}^{2}} \, \,  d\ol{n}, \quad \text{for } r\geq 0.
\end{equation}
Putting \eqref{app_abel_I-} the formula  above  in \eqref{I-bef_Abel}, we obtain
\begin{multline*} 
      \left| \left \langle \mathcal{I}^{-}_{\sigma} f ,{ \varphi }\right\rangle \right|  \leq  C_{p,n} \left( \sup_{x\in G }\| \sigma(x,\cdot)\|_{\mathcal{MH}(S_p,2)} \right)   \,  \int_{r=0}^{\infty} \int_{\mathbb{R}}   \int_{\overline{N}}   F(s) \varPhi(s-r)\\
      \cdot \eta (\overline{n} a_{r}) \, \frac{e^{-\frac{2}{p}( \,[\overline{n} a_{r}]^+)}}{ {(1+[\overline{n} a_{r}]^+)}^{2}} e^{2s}  d{\overline{n}}\, ds\, dr.
 \end{multline*}
 Substituting the explicit expression of $[\overline{n} a_{r}]^+$ (for $r\geq 0$) from Lemma \ref{lemma:expression_of_v-a _r} in the inequality above yields
 \begin{align*} 
      \left| \left \langle \mathcal{I}^{-}_{\sigma} f ,{ \varphi }\right\rangle \right|  & \leq  C_{p,n} \left( \sup_{x\in G }\| \sigma(x,\cdot)\|_{\mathcal{MH}(S_p,2)} \right)   \,  \int_{r=0}^{\infty} \int_{\mathbb{R}}    F(s) \varPhi(s-r)\\
     &  \hspace{5cm}\cdot \int_{\overline{N}}  \eta (\overline{n} a_{r}) \,{e^{-\frac{2}{p}r}}{} \frac{e^{-\frac{2}{p}H(\ol{n})}}{ {(1+H(\overline{n})+r )}^{2}} e^{2s}  d{\overline{n}}\, ds\, dr\\
 &\leq C_{p,n} \left( \sup_{x\in G }\| \sigma(x,\cdot)\|_{\mathcal{MH}(S_p,2)} \right)  \left(\int_{r=0}^{\infty} \left(\int_{\R}F(s) e^{\frac{2}{p}s} \varPhi(s-r)e^{\frac{2}{p'}(s-r)}  ds \right) \,   \right. \\ & \hspace{8cm} \cdot \left. \int_{\ol N}{ e^{-\frac{2}{p} H(\ol n) }}\, d\ol n      \frac{e^{(\frac{2}{p'}-\frac{2}{p})
						r}}{{(1+r)^{2}} } dr \right).
 \end{align*}
Here we are using this fact $H(\ol{n}) \geq 0$ for all $\ol{n}\geq 0$ (see \eqref{eqn:H(v)-positive}). Finally using \eqref{eqn:L1_norm_P(v-) (1+epsilon)_is_finite} with $1\leq p<2$, and applying H\"older's inequality, we conclude that
 \begin{align*}
     \left| \left \langle \mathcal{I}^{-}_{\sigma} f ,{ \varphi }\right\rangle \right|   \leq  C_{p,n} \left( \sup_{x\in G }\| \sigma(x,\cdot)\|_{\mathcal{MH}(S_p,2)} \right)  \| f\|_{L^p(\ol N A )} \,  \| { \varphi} \|_{L^{p'}(\ol N A)}.   \,  
 \end{align*}
 We next turn to prove the estimate \eqref{If,phi_est} for $p>2$, which will complete the proof of this lemma. After taking modulus on both sides of  \eqref{I_exp_final_var} and plugging the estimate, \eqref{est_of_K_glo_p>2} of $\mathcal{K}^{glo}$ for $2<p<\infty$, we can write   
\begin{multline*} 
      \left| \left \langle \mathcal{I}^{-}_{\sigma} f ,{ \varphi }\right\rangle \right|  \leq  C_{p,n} \left( \sup_{x\in G }\| \sigma(x,\cdot)\|_{\mathcal{MH}(S_p,2)} \right)   \,  \int_{r=0}^{\infty} \int_{\mathbb{R}} \int_{\overline{N}}  \int_{\overline{N}}  |f(\overline{m}a_s)|  | {\varphi ( \ol{m} \delta_{s}(\overline{n})a_{s-r})}|\\
      \cdot \eta (\overline{n} a_{-r}) \, \frac{e^{-\frac{2}{p'}( \,[\overline{n} a_{-r}]^+)}}{ {(1+[\overline{n} a_{-r}]^+)}^{2}} e^{2(s-r)} d\ol{m}\, d{\overline{n}}\, ds \,dr.
 \end{multline*}
 Then again, after an application of H\"older's inequality, we use \eqref{app_abel_I-} to get 
 \begin{multline*} 
      \left| \left \langle \mathcal{I}^{-}_{\sigma} f ,{ \varphi }\right\rangle \right|  \leq  C_{p,n} \left( \sup_{x\in G }\| \sigma(x,\cdot)\|_{\mathcal{MH}(S_p,2)} \right)   \,  \int_{r=0}^{\infty} \int_{\mathbb{R}}     F(s) \varPhi(s-r)\\
      \cdot  \int_{\overline{N}}  \eta (\overline{n} a_{r}) \,{e^{-\frac{2}{p'}r}}{} \frac{e^{-\frac{2}{p'}H(\ol{n})}}{ {(1+H(\overline{n})+r )}^{2}} e^{2s}  d{\overline{n}}\, ds\, dr
 \end{multline*}
 Since $2<p<\infty$, so we have $p'<2$. Thus we can use \eqref{eqn:L1_norm_P(v-) (1+epsilon)_is_finite} and write 
 
 \begin{align*} 
      \left| \left \langle \mathcal{I}^{-}_{\sigma} f ,{ \varphi }\right\rangle \right| 
 &\leq C_{p,n} \left( \sup_{x\in G }\| \sigma(x,\cdot)\|_{\mathcal{MH}(S_p,2)} \right)  \left(\int_{r=0}^{\infty} \left(\int_{\R}F(s) e^{\frac{2}{p}s} \varPhi(s-r)e^{\frac{2}{p'}(s-r)}  ds \right) \,    \right. \\ & \hspace{8cm} \cdot \left. \int_{\ol N}{ e^{-\frac{2}{p'} H(\ol n) }}\, d\ol n      \frac{dr}{{(1+r)^2} }  \right)\\
 &\leq C_{p,n} \left( \sup_{x\in G }\| \sigma(x,\cdot)\|_{\mathcal{MH}(S_p,2)} \right)  \| f\|_{L^p(\ol N A )} \,  \| { \varphi} \|_{L^{p'}(\ol N A)},   
 \end{align*}
 completing the proof. 
 \end{proof}
 \begin{lemma}\label{lem_est_of_I+}
		Let $p\in [1,2) \cup (2,\infty)$. Suppose that $ \sigma $ satisfies the hypothesis  of Theorem \ref{thm:Lp_boundedness_of_T_sigma_global} and $\mathcal{I}^{+}_{\sigma}$ be as in \eqref{Ipm_sigma}.  Then  $\mathcal{I}^{+}_{\sigma}$   is bounded from $L^{p}(\ol{N}A)$ to itself. Moreover, there exists a constant $ C_{p,n}>0 $  such that  \begin{align}\label{lem_Est_I+}
			 \|\mathcal{I}^{+}_{\sigma} f\|_{L^p (\ol {N}A)}&  \leq C_{p,n} \left( \sup_{x\in G}\| \sigma(x,\cdot)\|_{\mathcal{MH}(S_p,2)} \right)  \|f\|_{L^p(\ol {N}A)}
\end{align}
  for all $f \in L^p(\ol {N}A)$.
 \end{lemma}

 \begin{proof}
 Analysis similar to that in the proof of the previous lemma shows
that it is enough to prove $\left \langle \mathcal{I}^{+}_{\sigma} f ,{ \varphi}\right\rangle$  satisfies the estimate in \eqref{If,phi_est} for $1\leq p<2$,  where we recall the expression of  $\left \langle \mathcal{I}^{+}_{\sigma} f ,{ \varphi}\right\rangle   $ from \eqref{Ipm_sigma_prod}
\begin{align*}
       \left \langle \mathcal{I}^{+}_{\sigma} f ,{ \varphi}\right\rangle   = \int_{\overline{N}} \int_{\mathbb{R}}  \int_{\overline{N}} \int_{\mathbb{R}}  f(\overline{m}a_s) \ol {\varphi(\overline{n}a_t)}   \mathcal{K}^{\text{glo}}(\overline{n}a_t,\delta_{-s} (\overline{m}^{-1}\overline{n}) a_{t-s}) \chi^{+}(t-s) e^{2(s+t)}\, ds\, d{\overline{m}}\, dt\,  d\ol{n},
\end{align*}
for all smooth compactly supported functions  $f,\varphi$ on $\ol{N}A$. Applying the same change of variables as in the previous lemma, we obtain
\begin{align*}
       \left \langle \mathcal{I}^{+}_{\sigma} f ,{ \varphi}\right\rangle   = \int_{\overline{N}} \int_{\mathbb{R}}  \int_{\overline{N}} \int_{\mathbb{R}}  f(\overline{m}a_s) \ol {\varphi ( \ol{m} \delta_{s}(\overline{n})a_t)}  \mathcal{K}^{\text{glo}}(\ol{m}\delta_{s} (\overline{n})a_t,\overline{n} a_{t-s}) \chi^{-}(t-s) e^{2t} d{\overline{n}}\, ds\, dt\,  d\ol{m}.
\end{align*}
Next, by using Fubini's theorem and the change of variable $ t \mapsto r+s $ (for fixed $s$), it follows that  
\begin{align*}
           \left \langle \mathcal{I}^{+}_{\sigma} f ,{ \varphi}\right\rangle   =  \int_{\overline{N}} \int_{\mathbb{R}}  \int_{\overline{N}} \int_{r=0}^{\infty}  f(\overline{m}a_s) \ol {\varphi ( \ol{m} \delta_{s}(\overline{n})a_{r+s})} \mathcal{K}^{\text{glo}}(\ol{m} \delta_{s} (\overline{n})a_{r},\overline{n} a_{r})   e^{2(r+s)} dr\, d{\overline{n}}\, ds\,  d\ol{m}.
\end{align*}
Taking modulus on both sides of the above expression and plugging the estimate \eqref{est_of_K_glo} of $\mathcal{K}^{\text{glo}}$, we get
 \begin{multline}\label{eqn:equi_expn_of_I_+_2}
      \left| \left \langle \mathcal{I}^{+}_{\sigma} f ,{ \varphi }\right\rangle \right|  \leq  C_{p,n} \left( \sup_{x\in G }\| \sigma(x,\cdot)\|_{\mathcal{MH}(S_p,2)} \right)   \,  \int_{r=0}^{\infty} \int_{\mathbb{R}} \int_{\overline{N}}  \int_{\overline{N}}  |f(\overline{m}a_s)|  | {\varphi ( \ol{m} \delta_{s}(\overline{n})a_{r+s})}|\\
      \cdot \eta (\overline{n} a_{r}) \, \frac{e^{-\frac{2}{p}( \,[\overline{n} a_{r}]^+)}}{ {(1+[\overline{n} a_{r}]^+)}^{2}} e^{2(r+s)} d\ol{m}\, d{\overline{n}}\, ds \,dr.
 \end{multline}
 Plugging the explicit expression of $[\overline{n} a_{r}]^+$ from Lemma \ref{lemma:expression_of_v-a _r} in the inequality above gives us
 \begin{align*}
      \left| \left \langle \mathcal{I}^{+}_{\sigma} f ,{ \varphi }\right\rangle \right|  &\leq  C_{p,n} \left( \sup_{x\in G }\| \sigma(x,\cdot)\|_{\mathcal{MH}(S_p,2)} \right)   \,  \int_{r=0}^{\infty} \int_{\mathbb{R}} \int_{\overline{N}}  \int_{\overline{N}}  |f(\overline{m}a_s)|  | {\varphi ( \ol{m} \delta_{s}(\overline{n})a_{r+s})}|\\
      &\hspace{4.5cm}\cdot \eta (\overline{n} a_{r}) \, {e^{-\frac{2}{p}r}}{} \frac{e^{-\frac{2}{p}H(\ol{n})}}{ {(1+H(\overline{n})+r )}^{2}} e^{2(r+s)} d\ol{m}\, d{\overline{n}}\, ds \,dr\\
      & \leq C_{p,n} \left( \sup_{x\in G }\| \sigma(x,\cdot)\|_{\mathcal{MH}(S_p,2)} \right)   \,  \int_{r=0}^{\infty} \int_{\mathbb{R}}   F(s) e^{\frac{2}{p}s} \varPhi(s+r)e^{\frac{2}{p'}(s+r)} ds \\
      &\hspace{6cm} \cdot \int_{\overline{N}} \eta (\overline{n} a_{r}) \,  \frac{e^{-\frac{2}{p}H(\ol{n})}}{ {(1+r )}^{2}}  \, d{\overline{n}} \,dr,
 \end{align*}
 where in the last inequality we used H\"older's inequality and the fact $H(\ol n) \geq 0$. Finally, using \eqref{eqn:L1_norm_P(v-) (1+epsilon)_is_finite} and another application  of H\"older's inequality gives us
 \begin{align*}
      \left| \left \langle \mathcal{I}^{+}_{\sigma} f ,{ \varphi }\right\rangle \right|  &\leq  C_{p,n} \left( \sup_{x\in G }\| \sigma(x,\cdot)\|_{\mathcal{MH}(S_p,2)} \right)   \| f\|_{L^p(\ol N A )} \,  \| { \varphi} \|_{L^{p'}(\ol N A)}, 
 \end{align*}
 which concludes our lemma.
 \end{proof}
\section{Rough analysis of $\Psi_{\sigma}$}\label{sec_nonreg}
This section focuses on extending the boundedness result for pseudo-differential operators $\Psi_{\sigma}$ in Theorem \ref{thm:pdo_for_n,n_type} to accommodate rough symbols $\sigma$, which have no regularity condition in the space variable. Specifically, our goal in this section is to prove Theorem \ref{thm_nonreg}.

To achieve this, we first recall that the boundedness of $\Psi^{\text{glo}}_{\sigma}$ and $\Psi^{\text{dis}}_{\sigma}$ does not require any regularity condition on the space variable of the symbol $\sigma$, so their proofs follow a similar approach. In proving the boundedness of the local part $\Psi^{\text{loc}}_{\sigma}$, we crucially utilize the relation between it and Euclidean pseudo-differential operators, employing the result of Kenig and Staubach (Theorem \ref{thm_carlos}).

To outline the proof, we start by decomposing the pseudo-differential operator $\Psi_\sigma$ into discrete, local, and global parts, using the same argument as in Section \ref{decomposition of PSi}. The required boundedness of the operators $\Psi^{\text{dis}}_{\sigma}$ and $\Psi^{\text{glo}}_{\sigma}$ follows using similar methods as observed in Section \ref{sec_dis_part} and \ref{sec_glo_par}, respectively.

Next, we proceed to prove the corresponding $L^p(G)_{n,n}$ operator norm estimate for the local part $\Psi^{\text{loc}}_{\sigma}$. By utilizing the generalized transference principle of Coifmann-Weiss and employing the exact same analysis as in the proof of Theorem \ref{thm:Lp_boundedness_of_T_sigma_local}, we observe that it is sufficient to prove the hypothesis \eqref{eqn:Coifman-Weiss_condn_satisified_by_K} in Corollary \ref{cor_trans_cw}.

Let us recall from \eqref{eqn:separation_of_kernel} that we can write
 $$    \mathcal{K}^{\text{CW}}(\mathcal{R}_s x, \mathcal{R}_{-t} \mathcal{R}_s x , t) =    \mathcal{K}_0(xa_s,t) +  \mathcal{K}_{\text{err}}(xa_s,t), \quad x\in G, \, t\geq 0,$$
	such that
	\begin{enumerate}\item[(i)]There is some constant $ C>0 $ and $ \kappa \in L^1(\R) $, such that
		\begin{equation*} 
			|\mathcal{K}_{\text{err}}(x, t)| \leq C\, \kappa(t), \text{ for all   $ x \in G$} ,
		\end{equation*} \text{and}
		\item [(ii)] $ 
			\mathcal{K}_0(xa_s,t)  =  C \eta^\circ(a_t )\,  \Delta(t) \, \left( \frac{t}{\Delta(t)}\right)^{1/2} \int_{0}^\infty \Phi(\lambda)  \sigma(x a_s,\lambda ) \mathcal{J}_{0} (\lambda t) |c(\lambda)|^{-2} d\lambda.
		 $
	\end{enumerate} 
 Next for each $x\in G$, we define 
		\bes
		a_x(s,\xi):=   \int_{-\infty}^{\infty} \mathcal{K}_0(xa_s,t) e^{-2\pi i t \xi }  dt, \quad s, \xi\in \R.
		\ees
    By considering the hypothesis $\sigma (x,\lambda) \in \mathcal{S}^m_{1,\infty}(S_p)$, we can deduce from the calculations in the proof of \eqref{eqn:condn_on_K_0} that the family of symbols $\{a_x(s,\xi): x\in G\}$ satisfies the following estimate:
    \begin{equation}
    \| \partial^{\alpha}_{\xi}a_x(s,\xi)\|_{L^{\infty}_{s}} \leq C_\alpha (1+|\xi|)^{m- \alpha},
    \end{equation}
    where the constants $C_\alpha$ are independent of $x$. In other words, $a_x(s,\xi) \in {\mathcal{S}^{m}_
{1,\infty}}$ for all $x\in G$. Consequently, Theorem \ref{thm_carlos} implies that the hypothesis \eqref{eqn:Coifman-Weiss_condn_satisified_by_K} in Corollary \ref{cor_trans_cw} is satisfied. This in turn, establishes the $L^p(G)_{n,n}$-boundedness for $1\leq p<\infty$ of the operator $\Psi_\sigma^{\text{loc}}$. This completes the proof of Theorem \ref{thm_nonreg}.
  \qed

\section*{Final remarks}
In conclusion, we would like to offer some observations and draw attention to few open questions that, from our perspective, warrant further investigation.
\begin{enumerate} 
\item  By employing the Fourier inversion formula \cite[(10.4), Theorem 10.4]{Barker}, we can define multiplier operators, or more generally, pseudo-differential operators, for functions of $(m,n)$ type on the group $G$ and investigate their boundedness properties. While we have successfully established the boundedness of both the discrete and global components of the pseudo-differential operator using a similar analytical approach as in this article, addressing the local part poses a challenge.  By performing a calculation akin to that in Lemma \ref{thm:localexpansion-phi}, we can express $\phi_{\tau, \lambda}^{m,n}$   in terms of the Bessel function $\mathcal{J}_{\frac{m-n}{2}}$, unlike the $(n,n)$ type case where it was $\mathcal{J}_0$. Consequently, the argument presented in Section \ref{Analysis on the local part} may not be applicable unless $m=n$, as $\mathcal{J}_{\frac{m-n}{2}} (t)$ introduces a singularity near  $t=0$. Therefore, it is necessary to develop an alternative strategy to handle the local part of the pseudo-differential operator.
\item  Recently, Wrobel \cite{wrobel} presented a multiplier theorem for rank one symmetric spaces, which improves upon the results of both \cite{Stanton_and_Tomas} and \cite{Ionescu_2000}. It would be interesting to investigate whether a similar result can be obtained in our context.

\item   While our current approach has not yet demonstrated the $L^2$-boundedness of pseudo-differential operators, we view this as an opportunity to explore alternative strategies. In the near future, we plan to investigate the $L^2$-boundedness for pseudo-differential operators, both within our current framework and in rank one symmetric spaces of noncompact type. 
\end{enumerate}
\section*{Acknowledgments}
The authors thank Prof. Sanjoy Pusti for several useful discussions during the course of this work. TR and MR are supported by the FWO Odysseus 1 grant G.0H94.18N: Analysis and Partial Differential Equations, the Methusalem program of the Ghent University Special Research Fund (BOF), (TR Project title: BOFMET2021000601). MR is also supported by EPSRC grants EP/R003025/2 and EP/V005529/7.

\bibliographystyle{plain}
\bibliography{ReferencesIII.bib}

\begin{thebibliography}{10}

\bibitem{Anker}
J.-Ph. Anker.
\newblock {${\bf L}_p$} {F}ourier multipliers on {R}iemannian symmetric spaces
  of the noncompact type.
\newblock {\em Ann. of Math. (2)}, 132(3):597--628, 1990.

\bibitem{Barker}
W.~H. Barker.
\newblock {$L^p$} harmonic analysis on {${\rm SL}(2,{ \R})$}.
\newblock {\em Mem. Amer. Math. Soc.}, 76(393):iv+110, 1988.

\bibitem{MR3180890}
F.~Bernicot and D.~Frey.
\newblock Pseudodifferential operators associated with a semigroup of
  operators.
\newblock {\em J. Fourier Anal. Appl.}, 20(1):91--118, 2014.

\bibitem{Botchway_Ruzhansky_2020}
L.~N.~A. Botchway, P.~Ga\"{e}l~Kibiti, and M.~Ruzhansky.
\newblock Difference equations and pseudo-differential operators on {$\Bbb
  Z^n$}.
\newblock {\em J. Funct. Anal.}, 278(11):108473, 41, 2020.

\bibitem{MR3680539}
H.~Bustos and M.~M\u{a}ntoiu.
\newblock Twisted pseudo-differential operator on type {I} locally compact
  groups.
\newblock {\em Illinois J. Math.}, 60(2):365--390, 2016.

\bibitem{MR4322546}
D.~Cardona, J.~Delgado, and M.~Ruzhansky.
\newblock {$L^p$}-bounds for pseudo-differential operators on graded {L}ie
  groups.
\newblock {\em J. Geom. Anal.}, 31(12):11603--11647, 2021.

\bibitem{CR_Mem}
D.~Cardona and M.~Ruzhansky.
\newblock {\em Subelliptic pseudo-differential operators and Fourier integral
  operators on compact Lie groups}.
\newblock Mathematical Society of Japan, Tokyo.
\newblock (to appear).

\bibitem{Clerc_Stein}
J.~L. Clerc and E.~M. Stein.
\newblock {$L^{p}$}-multipliers for noncompact symmetric spaces.
\newblock {\em Proc. Nat. Acad. Sci. U.S.A.}, 71:3911--3912, 1974.

\bibitem{Coifman78}
R.~R. Coifman and Y.~Meyer.
\newblock {\em Au del\`{a} des op\'{e}rateurs pseudo-diff\'{e}rentiels.}
\newblock Soci\'{e}t\'{e} Math\'{e}matique de France, Paris, 1978.
\newblock With an English summary.

\bibitem{coifman_73}
R.~R. Coifman and G.~Weiss.
\newblock Operators associated with representations of amenable groups,
  singular integrals induced by ergodic flows, the rotation method and
  multipliers.
\newblock {\em Studia Math.}, 47:285--303, 1973.

\bibitem{coifman_77}
R.~R. Coifman and G.~Weiss.
\newblock {\em Transference methods in analysis}.
\newblock Conference Board of the Mathematical Sciences Regional Conference
  Series in Mathematics, No. 31. American Mathematical Society, Providence,
  R.I., 1976.

\bibitem{MR4343620}
A.~Dasgupta and S.~K. Nayak.
\newblock Pseudo-differential operators, {W}igner transform and {W}eyl
  transform on the {S}imilitude group, {$\Bbb{SIM}(2)$}.
\newblock {\em Bull. Sci. Math.}, 174:Paper No. 103087, 25, 2022.

\bibitem{MR3723800}
J.~Delgado and M.~Ruzhansky.
\newblock Schatten classes and traces on compact groups.
\newblock {\em Math. Res. Lett.}, 24(4):979--1003, 2017.

\bibitem{MR3936641}
J.~Delgado and M.~Ruzhansky.
\newblock {$L^p$}-bounds for pseudo-differential operators on compact {L}ie
  groups.
\newblock {\em J. Inst. Math. Jussieu}, 18(3):531--559, 2019.

\bibitem{MR4565437}
E.~Ewert.
\newblock Pseudodifferential operators on filtered manifolds as generalized
  fixed points.
\newblock {\em J. Noncommut. Geom.}, 17(1):333--383, 2023.

\bibitem{MR3362017}
V.~Fischer and M.~Ruzhansky.
\newblock A pseudo-differential calculus on graded nilpotent {L}ie groups.
\newblock In {\em Fourier analysis}, Trends Math., pages 107--132.
  Birkh\"{a}user/Springer, Cham, 2014.

\bibitem{MR3167566}
V.~Fischer and M.~Ruzhansky.
\newblock A pseudo-differential calculus on the {H}eisenberg group.
\newblock {\em C. R. Math. Acad. Sci. Paris}, 352(3):197--204, 2014.

\bibitem{FR}
V.~Fischer and M.~Ruzhansky.
\newblock {\em Quantization on nilpotent {L}ie groups}, volume 314 of {\em
  Progress in Mathematics}.
\newblock Birkh\"{a}user/Springer, [Cham], 2016.

\bibitem{MR3610249}
M.~B. Ghaemi and M.~J. Birgani.
\newblock {$L^p$}-boundedness, compactness of pseudo-differential operators on
  compact {L}ie groups.
\newblock {\em J. Pseudo-Differ. Oper. Appl.}, 8(1):1--11, 2017.

\bibitem{Grafakos}
L.~Grafakos.
\newblock {\em Classical {F}ourier analysis}, volume 249 of {\em Graduate Texts
  in Mathematics}.
\newblock Springer, New York, second edition, 2008.

\bibitem{Helgason_GGA}
S.~Helgason.
\newblock {\em Groups and geometric analysis}, volume 113 of {\em Pure and
  Applied Mathematics}.
\newblock Academic Press, Inc., Orlando, FL, 1984.
\newblock Integral geometry, invariant differential operators, and spherical
  functions.

\bibitem{Helgason_GA}
S.~Helgason.
\newblock {\em Geometric analysis on symmetric spaces}, volume~39 of {\em
  Mathematical Surveys and Monographs}.
\newblock American Mathematical Society, Providence, RI, second edition, 2008.

\bibitem{Hormander_1965}
L.~H\"{o}rmander.
\newblock Pseudo-differential operators.
\newblock {\em Comm. Pure Appl. Math.}, 18:501--517, 1965.

\bibitem{Hormander_Ann_1966}
L.~H\"{o}rmander.
\newblock Pseudo-differential operators and non-elliptic boundary problems.
\newblock {\em Ann. of Math. (2)}, 83:129--209, 1966.

\bibitem{Hormander_class}
L.~H\"{o}rmander.
\newblock Pseudo-differential operators and hypoelliptic equations.
\newblock In {\em Singular {I}ntegrals ({P}roc. {S}ympos. {P}ure {M}ath.,
  {C}hicago, {I}ll., 1966)}, pages 138--183. 1967.

\bibitem{HormanderL271}
L.~H\"{o}rmander.
\newblock On the {$L\sp{2}$} continuity of pseudo-differential operators.
\newblock {\em Comm. Pure Appl. Math.}, 24:529--535, 1971.

\bibitem{Ionescu_2000}
A.~D. Ionescu.
\newblock Fourier integral operators on noncompact symmetric spaces of real
  rank one.
\newblock {\em J. Funct. Anal.}, 174(2):274--300, 2000.

\bibitem{Ionescu_2002}
A.~D. Ionescu.
\newblock Singular integrals on symmetric spaces of real rank one.
\newblock {\em Duke Math. J.}, 114(1):101--122, 2002.

\bibitem{Ionescu_2003}
Alexandru~D. Ionescu.
\newblock Singular integrals on symmetric spaces. {II}.
\newblock {\em Trans. Amer. Math. Soc.}, 355(8):3359--3378, 2003.

\bibitem{carlos_07}
C.~E. Kenig and W.~Saubach.
\newblock {$\Psi$}-pseudodifferential operators and estimates for maximal
  oscillatory integrals.
\newblock {\em Studia Math.}, 183(3):249--258, 2007.

\bibitem{Kohn_Nirenberg_1965}
J.~J. Kohn and L.~Nirenberg.
\newblock An algebra of pseudo-differential operators.
\newblock {\em Comm. Pure Appl. Math.}, 18:269--305, 1965.

\bibitem{MR0374832}
T.~H. Koornwinder.
\newblock A new proof of a {P}aley-{W}iener type theorem for the {J}acobi
  transform.
\newblock {\em Ark. Mat.}, 13:145--159, 1975.

\bibitem{MR0774055}
T.~H. Koornwinder.
\newblock Jacobi functions and analysis on noncompact semisimple {L}ie groups.
\newblock In {\em Special functions: group theoretical aspects and
  applications}, Math. Appl., pages 1--85. Reidel, Dordrecht, 1984.

\bibitem{Kumango70}
H.~Kumano-go.
\newblock A problem of {N}irenberg on pseudo-differential operators.
\newblock {\em Comm. Pure Appl. Math.}, 23:115--121, 1970.

\bibitem{Masson_14}
E.~Le~Masson.
\newblock Pseudo-differential calculus on homogeneous trees.
\newblock {\em Ann. Henri Poincar\'{e}}, 15(9):1697--1732, 2014.

\bibitem{mililic_1977}
D.~Mili\v{c}i\'{c}.
\newblock Asymptotic behaviour of matrix coefficients of the discrete series.
\newblock {\em Duke Math. J.}, 44(1):59--88, 1977.

\bibitem{MR3953844}
M.~M\u{a}ntoiu.
\newblock A positive quantization on type {I} locally compact groups.
\newblock {\em Math. Nachr.}, 292(5):1043--1055, 2019.

\bibitem{MR3741847}
M.~M\u{a}ntoiu and M.~Ruzhansky.
\newblock Pseudo-differential operators, {W}igner transform and {W}eyl systems
  on type {I} locally compact groups.
\newblock {\em Doc. Math.}, 22:1539--1592, 2017.

\bibitem{MR3969446}
M.~M\u{a}ntoiu and M.~Ruzhansky.
\newblock Quantizations on nilpotent {L}ie groups and algebras having flat
  coadjoint orbits.
\newblock {\em J. Geom. Anal.}, 29(3):2823--2861, 2019.

\bibitem{MR3923338}
M.~M\u{a}ntoiu and M.~Sandoval.
\newblock Pseudo-differential operators associated to general type {I} locally
  compact groups.
\newblock In {\em Analysis and partial differential equations: perspectives
  from developing countries}, volume 275 of {\em Springer Proc. Math. Stat.},
  pages 172--190. Springer, Cham, 2019.

\bibitem{Nagase77}
M.~Nagase.
\newblock The {$L\sp{p}$}-boundedness of pseudo-differential operators with
  non-regular symbols.
\newblock {\em Comm. Partial Differential Equations}, 2(10):1045--1061, 1977.

\bibitem{NIST_2010}
F.~W.~J. Olver, D.~W. Lozier, R.~F. Boisvert, and C.~W. Clark, editors.
\newblock {\em N{IST} handbook of mathematical functions}.
\newblock U.S. Department of Commerce, National Institute of Standards and
  Technology, Washington, DC; Cambridge University Press, Cambridge, 2010.
\newblock With 1 CD-ROM (Windows, Macintosh and UNIX).

\bibitem{PR_PDO_22}
S.~Pusti and T.~Rana.
\newblock {$L^p$}-boundedness of pseudo-differential operators on rank one
  {R}iemannian symmetric spaces of noncompact type.
\newblock {\em Math. Z.}, 305(1):Paper No. 2, 34, 2023.

\bibitem{Rana}
T.~Rana.
\newblock A genuine analogue of the {W}iener {T}auberian theorem for some
  {L}orentz spaces on {${\rm SL}(2,\Bbb R)$}.
\newblock {\em Forum Math.}, 33(1):213--243, 2021.

\bibitem{Rana_Rano}
T.~Rana and S.~K. Rano.
\newblock {$L^p$}-boundedness of pseudo-differential operators on homogeneous
  trees.
\newblock {\em Studia Math.}, 2023.
\newblock https://doi.org/10.4064/sm220816-27-3.

\bibitem{Ricci}
F.~Ricci and B.~Wr\'{o}bel.
\newblock Spectral multipliers for functions of fixed {$K$}-type on
  {$L^p(SL(2,\Bbb R))$}.
\newblock {\em Math. Nachr.}, 293(3):554--584, 2020.

\bibitem{Ruzhansky_PDO_SYM}
M.~Ruzhansky and V.~Turunen.
\newblock {\em Pseudo-differential operators and symmetries}, volume~2 of {\em
  Pseudo-Differential Operators. Theory and Applications}.
\newblock Birkh\"{a}user Verlag, Basel, 2010.
\newblock Background analysis and advanced topics.

\bibitem{Ruzhansky_torus_10}
M.~Ruzhansky and V.~Turunen.
\newblock Quantization of pseudo-differential operators on the torus.
\newblock {\em J. Fourier Anal. Appl.}, 16(6):943--982, 2010.

\bibitem{Ruzhansky_comp_13}
M.~Ruzhansky and V.~Turunen.
\newblock Global quantization of pseudo-differential operators on compact {L}ie
  groups, {$\rm SU(2)$}, 3-sphere, and homogeneous spaces.
\newblock {\em Int. Math. Res. Not. IMRN}, (11):2439--2496, 2013.

\bibitem{MR3217484}
M.~Ruzhansky, V.~Turunen, and J.~Wirth.
\newblock H\"{o}rmander class of pseudo-differential operators on compact {L}ie
  groups and global hypoellipticity.
\newblock {\em J. Fourier Anal. Appl.}, 20(3):476--499, 2014.

\bibitem{Ruzhansky_glo_14}
M.~Ruzhansky and J.~Wirth.
\newblock Global functional calculus for operators on compact {L}ie groups.
\newblock {\em J. Funct. Anal.}, 267(1):144--172, 2014.

\bibitem{MR3369343}
M.~Ruzhansky and J.~Wirth.
\newblock {$L^p$} {F}ourier multipliers on compact {L}ie groups.
\newblock {\em Math. Z.}, 280(3-4):621--642, 2015.

\bibitem{Stanton_and_Tomas}
R.~J. Stanton and P.~A. Tomas.
\newblock Expansions for spherical functions on noncompact symmetric spaces.
\newblock {\em Acta Math.}, 140(3-4):251--276, 1978.

\bibitem{Stein_93}
E.~M. Stein.
\newblock {\em Harmonic analysis: real-variable methods, orthogonality, and
  oscillatory integrals}, volume~43 of {\em Princeton Mathematical Series}.
\newblock Princeton University Press, Princeton, NJ, 1993.
\newblock With the assistance of Timothy S. Murphy, Monographs in Harmonic
  Analysis, III.

\bibitem{Stromberg}
J.-O. Str\"{o}mberg.
\newblock Weak type {$L^{1}$} estimates for maximal functions on noncompact
  symmetric spaces.
\newblock {\em Ann. of Math. (2)}, 114(1):115--126, 1981.

\bibitem{Titchmarsh}
E.~C. Titchmarsh.
\newblock {\em The theory of functions}.
\newblock Oxford University Press, Oxford, second edition, 1939.

\bibitem{trombi_1972}
P.~C. Trombi and V.~S. Varadarajan.
\newblock Asymptotic behaviour of eigen functions on a semisimple {L}ie group:
  the discrete spectrum.
\newblock {\em Acta Math.}, 129(3-4):237--280, 1972.

\bibitem{MR4078700}
J.~van Neerven and P.~Portal.
\newblock The {W}eyl calculus for group generators satisfying the canonical
  commutation relations.
\newblock {\em J. Operator Theory}, 83(2):253--298, 2020.

\bibitem{wrobel}
B.~Wr\'{o}bel.
\newblock Imrpoved multiplier theorem on rank one noncompact symmetric spaces.
\newblock {\em arXiv:2305.06039}, 2023.
\newblock https://doi.org/10.48550/arXiv.2305.06039.

\end{thebibliography}

\end{document}